\documentclass[12pt]{article}
\usepackage{mathtools}
\usepackage{amsthm}
\usepackage{csquotes}
\usepackage{amssymb}
\usepackage{graphicx}
\usepackage{subfig}
\usepackage{soul}  
\usepackage{xcolor}
\usepackage{empheq}
\usepackage{lipsum}
\usepackage{fullpage}
\usepackage{float} 

\usepackage[top= 2cm, bottom = 2 cm, left = 2.2 cm, right= 2.2 cm]{geometry}  
\usepackage[obeyspaces,hyphens,spaces]{url}
\usepackage[unicode]{hyperref}
\hypersetup{
    colorlinks=true,
    linkcolor=blue,
    citecolor=blue,
    filecolor=magenta,
    urlcolor=cyan
}
\usepackage{listings}

\definecolor{codegreen}{rgb}{0,0.6,0}
\definecolor{codegray}{rgb}{0.5,0.5,0.5}
\definecolor{codepurple}{rgb}{0.58,0,0.82}
\definecolor{backcolour}{rgb}{0.95,0.95,0.92}

\lstdefinestyle{mystyle}{
	backgroundcolor=\color{backcolour},   
	commentstyle=\color{codegreen},
	keywordstyle=\color{magenta},
	numberstyle=\footnotesize\color{codegray},
	stringstyle=\color{codepurple},
	basicstyle=\ttfamily\small,
	breakatwhitespace=false,         
	breaklines=true,                 
	captionpos=b,                    
	keepspaces=true,                 
	numbers=left,                    
	numbersep=5pt,                  
	showspaces=false,                
	showstringspaces=false,
	showtabs=false,                  
	tabsize=2
}

\definecolor{seagreen}{rgb}{0.18, 0.55, 0.34}
\definecolor{mediumviolet-red}{rgb}{0.78, 0.08, 0.52}
\definecolor{khaki}{rgb}{0.94, 0.9, 0.55}

\lstdefinelanguage{mypython}
{
	keywords=[1]{from, import, assert, not, print},
	keywordstyle=[1]{\color{mediumviolet-red}},
	keywords=[2]{surecr, torch, cp, lo, pl},
	keywordstyle=[2]{\color{seagreen}},
	numbers=none,
	upquote=true,
	showstringspaces=false,
	basicstyle=\ttfamily,
	columns=fullflexible,
	keepspaces=true,
	emph={True,False,as,def,return,float,class,match,switch,len},
	emphstyle={\color{seagreen}},
	frame=trBL,
	belowskip=1em,
	aboveskip=1em,
	captionpos=b
}

\usepackage[capitalize, nameinlink, noabbrev]{cleveref}
\crefname{equation}{}{}
\crefname{chapter}{Chapter}{Chapters}
\crefname{item}{item}{items}
\crefname{figure}{Figure}{Figures}
\crefname{theorem}{Theorem}{Theorems}
\crefname{lemma}{Lemma}{Lemmas}
\crefname{proposition}{Proposition}{Propositions}
\crefname{corollary}{Corollary}{Corollarys}
\crefname{definition}{Definition}{Definitions}
\crefname{fact}{Fact}{Facts}
\crefname{example}{Example}{Examples}
\crefname{algorithm}{Algorithm}{Algorithms}
\crefname{remark}{Remark}{Remarks}
\crefname{note}{Note}{Notes}
\crefname{notation}{Notation}{Notations}
\crefname{case}{Case}{Cases}
\crefname{exercise}{Exercise}{Exercises}
\crefname{question}{Question}{Questions}
\crefname{claim}{Claim}{Claims}
\crefname{enumi}{}{}

\numberwithin{equation}{section}

\usepackage{amsmath,xparse}
\makeatletter
\NewDocumentCommand{\lplabel}{o m}{%
	\makebox[0pt][r]{#2\hspace*{2em}}%
	\IfNoValueF{#1}
	{\def\@currentlabel{#2}\ltx@label{#1}}
}
\makeatother

\theoremstyle{plain}
\newtheorem{theorem}{Theorem}[section]
\newtheorem{corollary}{Corollary}[section]
\newtheorem{fact}{Fact}[section]
\newtheorem{lemma}{Lemma}[section]
\newtheorem{proposition}{Proposition}[section]
\theoremstyle{definition}

\newtheorem{example}{Example}[section]

\usepackage{enumitem}

\newcommand{\minimize}{\ensuremath{\operatorname{minimize}}}
\newcommand{\maximize}{\ensuremath{\operatorname{maximize}}}
\newcommand{\inte}{\ensuremath{\operatorname{int}}}
\newcommand{\weakly}{\ensuremath{{\;\operatorname{\rightharpoonup}\;}}}
\newcommand{\gra}{\ensuremath{\operatorname{gra}}}
\newcommand{\Id}{\ensuremath{\operatorname{Id}}}
\newcommand{\Pro}{\ensuremath{\operatorname{P}}}
\newcommand{\Prox}{\ensuremath{\operatorname{Prox}}}
\newcommand{\argmin}{\mathop{\rm argmin}}
\newcommand{\argmax}{\mathop{\rm argmax}}

\providecommand{\abs}[1]{\left|#1\right|}
\providecommand{\norm}[1]{\left\lVert#1\right\rVert}
\providecommand{\innp}[1]{\left\langle#1\right\rangle}

\begin{document}

\title{Alternating Proximity Mapping Method for Convex-Concave Saddle-Point 
Problems}

\author{
	 Hui Ouyang\thanks{Department of Electrical Engineering, Stanford University.
		E-mail: \href{mailto:houyang@stanford.edu}{\texttt{houyang@stanford.edu}}.}
}

\date{October 30, 2023}

\maketitle

\begin{abstract}
We proposed an iterate scheme for solving convex-concave saddle-point problems
associated with general convex-concave functions. 
We demonstrated that when our iterate scheme is applied to a special class of 
convex-concave functions, which are constructed by a bilinear coupling term plus a 
difference of two convex functions,
it becomes a generalization of several popular primal-dual algorithms
from constant involved parameters to involved parameters as general 
sequences.
For this specific class of convex-concave functions, 
we proved that the sequence of function values, 
taken over the averages of iterates generated by our scheme,
converges to the  value of the function at a saddle-point. 
Additionally,
we provided convergence results for both the sequence of averages of our iterates
and the sequence of our iterates.

In our numerical experiments, 
we implemented our algorithm in a matrix game, a linear program in inequality form,
and a least-squares problem with $\ell_{1}$ regularization.
In these examples, we also compared our algorithm with other primal-dual algorithms
where parameters in their iterate schemes were kept constant.
Our experimental results not only validated our theoretical findings but also 
demonstrated that our algorithm consistently outperforms various iterate schemes with 
constant involved parameters.
\end{abstract}

{\small

}


 \section{Introduction}
 In the whole work,  
 $\mathcal{H}_{1}$ and $\mathcal{H}_{2}$ are Hilbert spaces and   
 $X \subseteq  \mathcal{H}_{1}$ and $Y \subseteq \mathcal{H}_{2}$ 
 are nonempty closed and convex sets. 
 The Hilbert direct sum $\mathcal{H}_{1} \times \mathcal{H}_{2}$ of  $\mathcal{H}_{1}$ 
 and $\mathcal{H}_{2}$ is equipped with the inner product 
\[ 
(\forall (x, y) \in X \times Y ) (\forall (u,v) \in X \times Y) \quad \innp{(x,y), (u,v)} = 
\innp{x,u} +\innp{y,v}
\]
 and the induced norm
\[ 
(\forall (x, y) \in X \times Y ) \quad \norm{(x, y) }^{2} =\innp{(x,y), (x,y)}=\innp{x,x} 
 	+\innp{y,y} =\norm{x}^{2} + \norm{y}^{2}. 
\]
Throughout this work, $f : X \times Y \to \mathbf{R} \cup \{ - \infty, + \infty \}$ satisfies 
that 
$(\forall y \in Y)$ $f(\cdot, y) : X \to \mathbf{R}  \cup \{ + \infty\}$ is proper and convex, 
and  that 
$(\forall x \in X)$ $f(x, \cdot) : Y \to \mathbf{R}  \cup \{ - \infty\}$ is proper and concave. 
(It is referred to as convex-concave function from now on.)

Our goal in this work is to solve the following \emph{convex-concave saddle-point 
problem}
\begin{align}\label{eq:problem}
\underset{x \in X}{\minimize} ~\,	\underset{y \in Y}{\maximize}  ~f(x,y).
\end{align}
We assume  that  the solution set of \cref{eq:problem} is nonempty, 
 that is, there exists a \emph{saddle-point} $(x^{*}, y^{*}) \in X \times Y$ of $f$
 satisfying 
 \begin{align}\label{eq:solution}
 (\forall x \in X) (\forall y \in Y) \quad f(x^{*}, y)  \leq  f(x^{*}, y^{*})  \leq f(x, y^{*}).
 \end{align}

 We  point out some notation used in this work.  
 $\mathbf{R}$, $\mathbf{R}_{+}$, $\mathbf{R}_{++}$, and $\mathbf{N}$  are the set of all 
 real numbers, the set of all nonnegative real numbers, the set of all positive real numbers, 
 and the set of all nonnegative integers, respectively.  $\mathcal{B}(\mathcal{H}_{1} , 
 \mathcal{H}_{2} ) := \{ T: \mathcal{H}_{1} \to 
 \mathcal{H}_{2} ~:~ T \text{ is linear and continuous} \}$
 is the set of all linear and continuous operators from $\mathcal{H}_{1}$ to  
 $\mathcal{H}_{2}$.
 
 Let $\mathcal{H}$ be a Hilbert space. 
 $\Id :\mathcal{H}\to \mathcal{H}: x \mapsto x$ is the \emph{identity operator}.
 $2^{\mathcal{H}}$ is the \emph{power set of $\mathcal{H}$}, i.e., 
 the family of all subsets of $\mathcal{H}$.
 Let $F: \mathcal{H} \to 2^{\mathcal{H}}$ be a set-valued operator. 
 The \emph{graph of $F$} is $\gra F:= \{ (x,u) \in \mathcal{H} \times \mathcal{H} ~:~ u \in 
 Fx \}$.
 Let $C$ be a nonempty subset of $\mathcal{H}$. 
 $\inte C$ signifies the \emph{interior of the set $C$}.  If $C$ is   nonempty closed and 
 convex, 
 then the \emph{projector} (or \emph{projection operator}) onto $C$ is the operator, 
 denoted by $\Pro_{C}$,  that maps every point in $\mathcal{H}$ to its unique projection 
 onto $C$, that is, $(\forall x \in \mathcal{H})$ $\norm{x - \Pro_{C}x} = \inf_{c \in C} 
 \norm{x-c}$.  
 Let $g: \mathcal{H} \to \left]-\infty, 
 +\infty\right]$ be proper. The \emph{subdifferential of $g$} is the set-valued operator 
\[  
 \partial g: \mathcal{H} \to 2^{\mathcal{H}}: x \mapsto \{ u \in \mathcal{H} 
 ~:~ (\forall y \in \mathcal{H})~ \innp{u, y-x} + f(x) \leq f(y) \}.
\]
 Suppose that $g$ is proper, convex, and lower semicontinuous. 
 The \emph{proximity mapping $\Prox_{g}$ of $g$} is defined by
\[ 
\Prox_{g}: \mathcal{H} \to \mathcal{H} : x \mapsto \argmin_{y \in \mathcal{H}} \left(  g(y) 
+\frac{1}{2} \norm{x-y}^{2} \right).
\]
\subsection{Related work}
 Various algorithms for solving convex-concave saddle-point problems associated 
 with  a general convex-concave function can be found in 
 \cite{Korpelevich1976extragradient}, 
 \cite{Nesterov2009}, 
 \cite{NedicOzdaglar2009},   \cite{WandLi2020improved}, 
 \cite{HamedaniAybat2021},  \cite{ZhangWangLessardGrosse2022}, 
 \cite{BoctCsetnekSedlmayer2022accelerated},
 \cite{SLB2023},  and \cite{Oy2023subgradient}.  
 In this case, normally there are strong assumptions 
regarding the differentiability, strong convexity,
or     subgradients of $f$ in convergence proofs. 
 
 Let $f: X \times Y \to \mathbf{R} \cup \{ -\infty, + \infty\}$ defined as
\begin{align} \label{eq:fspecialform}
(\forall (x,y) \in X \times Y) \quad 	f(x,y) =\innp{Kx, y} + g(x)  -h(y),
\end{align} 
where $K \in \mathcal{B}(\mathcal{H}_{1} , \mathcal{H}_{2} )$, and 
$g: \mathcal{H}_{1} \to \mathbf{R} \cup \{ +\infty\}$ and $h: \mathcal{H}_{2} \to  
\mathbf{R} \cup \{ 
+\infty\}$  are proper, convex, and lower semicontinuous. 
Because convex-concave saddle-point problems associated  with $f$ in the form of 
\cref{eq:fspecialform} arise in a wide range of applications 
including image problems, total variation minimization problems, empirical risk 
minimization, reinforcement learning,  bilinear min-max problems, etc, 
there are many papers working on this special case.
(See, e.g., \cite{PockCremersBischofChambolle2009algorithm}, 
\cite{EsserZhangChan2010}, \cite{HeYuan2010convergence}, \cite{ChambollePock2011}, 
\cite{DuHu2019linear}, \cite{KovalevGasnikovRichtarik2022accelerated} and some 
references therein.)
In fact, our  iterate scheme \cref{eq:lemma:algorithmsymplity}   for 
convex-concave saddle-point problems associated with $f$ in the form of 
\cref{eq:fspecialform} is a generalization (in terms of involved parameters) 
of iterate schemes worked in 
\cite{PockCremersBischofChambolle2009algorithm}, \cite{EsserZhangChan2010},
 \cite{HeYuan2010convergence}, and \cite{ChambollePock2011} 
although we have different assumptions of involved parameters in convergence proofs.
 
 It is mentioned in multiple papers (see, e.g., \cite{Popov1980modification}, 
 \cite{ZhuChan2008efficient}, \cite{ChambollePock2011},  and \cite{HeXuYuan2022}) 
 that the Arrow-Hurwicz method  is a classical primal-dual methods for solving 
 saddle-point problems, and also the basis of multiple important algorithms such as 
 the extragradient method and the primal-dual hybrid gradient method. In 
 \cite{ZhuChan2008efficient} (similarly in\cite{ChambollePock2011} and 
 \cite{HeXuYuan2022}), the Arrow-Hurwicz method is described as a  primal-dual 
 proximal-point method:
 \begin{align}\label{eq:pdppm}
 	&x^{k+1} =\argmax_{x \in X} \Phi(y^{k},x) - \frac{1}{2 \lambda \tau_{k}} \norm{x 
 		-x^{k}}^{2}\\
 	&y^{k+1} = \argmin_{y \in \mathbf{R}^{N}} \Phi(y,x^{k+1}) + \frac{\lambda 
 		(1-\theta_{k})}{2 \theta_{k}} \norm{y-y^{k}}^{2}. 
 \end{align}
In the attached Appendix, we shall distinguish the Arrow-Hurwicz method and the  
primal-dual proximal-point method presented above. 
Additionally, we will show \cref{lemma:pmsm} below to conclude that 
it's an open question if all algorithms with projected subgradients 
 (e.g., some special cases of the Arrow-Hurwicz method)
have their counterparts  of algorithms with  proximity mappings 
(e.g., primal-dual proximal-point methods) 
in closed forms.

\subsection{Comparison with related work}
This work is motivated by \cite[Section~3]{ChambollePock2011}. 
The authors of \cite{ChambollePock2011} considered only the convex-concave 
saddle-point problem 
associated with a convex-concave function
in the form of \cref{eq:fspecialform}.
In \cite[Theorem~1]{ChambollePock2011}, they considered the 
convergence of the iterates $((x^{k},y^{k}))_{k \in \mathbf{N}}$ generated by 
the iterate scheme \cref{eq:lemma:algorithmsymplity} below with
$(\forall k \in \mathbf{N})$ $\sigma_{k} \equiv 
\sigma \in \mathbf{R}_{++}$, $\tau_{k} \equiv \tau \in \mathbf{R}_{++}$, $\alpha_{k} 
\equiv 1$, and $\beta_{k} \equiv 0$ 
under the assumption of  $\tau \sigma \norm{K}^{2} <1$. 
Moreover, in \cite[Section~5]{ChambollePock2011}, 
under some assumptions regarding the uniformly convexity of $g$ or $h$ in 
\cref{eq:fspecialform}, 
the authors worked on some acceleration of the
convergence of the iterates $((x^{k},y^{k}))_{k \in \mathbf{N}}$ generated by 
some special cases of  the iterate scheme
\cref{eq:lemma:algorithmsymplity} with some restrictions on involved  parameters. 
Suppose that $f$ is a general convex-concave function and that $(x^{*}, y^{*}) \in X 
\times Y$ is a saddle-point of $f$. 
We compare this work with  \cite{ChambollePock2011} as follows.
\begin{itemize}
	\item Our iterate scheme \cref{eq:PPalgorithm:1} is constructed for convex-concave 
	saddle-point problems presented in \cref{eq:problem} associated with a general 
	convex-concave function $f$. 
	Even our \cref{eq:lemma:algorithmsymplity} 
	(that is \cref{eq:PPalgorithm:1} used 
	specifically for $f$ in the special form of \cref{eq:fspecialform}) 
	is more general than all iterate schemes  considered in \cite{ChambollePock2011}. 
	
	Suppose that $((x^{k},y^{k}))_{k \in \mathbf{N}}$  is generated by 
	\cref{eq:PPalgorithm:1} and $((\hat{x}^{k}, \hat{y}^{k}))_{k \in \mathbf{N}}$ 
	is constructed by \cref{eq:x_ky_k} below.
	Although we didn't have any convergence result on $((x^{k},y^{k}))_{k \in 
	\mathbf{N}}$, $(f(x^{k},y^{k}))_{k \in 	\mathbf{N}}$, or $(f(\hat{x}^{k}, 
	\hat{y}^{k}))_{k \in \mathbf{N}}$ for a general convex-concave $f$, 
	we provide in \cref{proposition:fxyk+1-xyhat} upper and lower bounds for 
	the sequence $\left( f(\hat{x}_{k}, \hat{y}_{k}) -f(x^{*}, y^{*}) \right)_{k \in 
	\mathbf{N}}$,
	which plays a critical role in the proof of one of our convergence results later 
	and may be 
	useful to prove $f(\hat{x}_{k},\hat{y}_{k}) \to f(x^{*},y^{*})$ for a general
	 convex-concave function f.
	\item We consider $f$ in the form of \cref{eq:fspecialform} below. 
	\begin{itemize}
		\item  The authors in  \cite{ChambollePock2011}  worked only  on the convergence 
		of 	the sequence $((x^{k},y^{k}))_{k \in \mathbf{N}}$ of iterates  generated by 
		their schemes. 
		We not only proved the convergence of $\left((x^{k},y^{k}) \right)_{k 
		\in \mathbf{N}}$ to a saddle-point of $f$, but we also showed the convergence of
		$\left((\hat{x}_{k},\hat{y}_{k}) \right)_{k \in \mathbf{N}}$
		and  $f(\hat{x}_{k}, 	\hat{y}_{k}) 	\to f(x^{*}, y^{*})  $, where 
		$((\hat{x}^{k}, \hat{y}^{k}))_{k \in \mathbf{N}}$ is given in \cref{eq:x_ky_k} below. 
		
		\item In \cite[Section~5]{ChambollePock2011}, 
		the authors worked on some accelerated iterate schemes
		(these are some special cases of the iterate scheme 
		\cref{eq:lemma:algorithmsymplity} below
		with involved parameters satisfying certain rules) 
		of their original iterate scheme
		(that is the scheme \cref{eq:lemma:algorithmsymplity} 
		below with $(\forall k \in \mathbf{N})$ 
		$\sigma_{k} \equiv \sigma $, $\tau_{k} \equiv \tau $, 
		 $\alpha_{k} \equiv 1$, and $\beta_{k} \equiv 0$ ) 
		 under some uniformly convexity or  smoothness assumptions, 
		 but	in this work,  
		 we did not consider uniformly convexity or smoothness assumption. 
		
		\item \cite[Section~6]{ChambollePock2011}  provided applications of their 
		theoretical results in	image problems, 
		while we verified our theoretical results in numerical experiments on 
		matrix game,
		linear program in inequality form, and
		least-squares problem with $\ell_{1}$ regularization. 
		
		In our numerical experiments, we presented some examples in which 
		the sequences $\left(  f(x^{k}, y^{k})  \right)_{k \in \mathbf{N}}$ and 
		$\left(  f(\hat{x}_{k},\hat{y}_{k}) \right)_{k \in \mathbf{N}}$ satisfying assumptions 
		in our theoretical results perform better than 
		corresponding sequences satisfying  assumptions in   theoretical results of 
		\cite{ChambollePock2011}. 
		(See,  \cref{fig:appmg_proximity_methods_2_00.png},
		\cref{fig:proximal_point_algorithm_lp_tausigma_00.png}, and
		\cref{fig:proximal_point_algorithm_lsl1_tausigma_00.png}.)
		
		In our numerical experiments, we  also displayed some convergence results
		on the sequence $\left(  f(x^{k}, y^{k})  \right)_{k \in \mathbf{N}}$, 
		which, to the best of our knowledge, lacks of theoretical support as of now.  
		(See, 
		\cref{fig:appmg_proximity_methods_2_00.png}, 
		\cref{fig:appmg_proximity_methods_alphabeta_es_00.png},
	 \cref{fig:proximal_point_algorithm_lp_tausigma_00.png}, and
	 \cref{fig:proximal_point_algorithm_lsl1_tausigma_00.png}.)
	 Moreover, we discovered some interesting consistence of the sequences 
	 $\left(  f(x^{k}, y^{k})  \right)_{k \in \mathbf{N}}$ and 
	 $\left(  f(\hat{x}_{k},\hat{y}_{k}) \right)_{k \in \mathbf{N}}$  with different 
		 choices of  step-size. 
	 (See, \cref{fig:proximal_point_algorithm_lp_alphabetatausigma_00.png} and 
	 \cref{fig:proximal_point_algorithm_lsl1_alphabetatausigma_00.png}.)
	\end{itemize}
	
\end{itemize}

 \subsection{Outline}
This work is organized as follows. 
We present some preliminary results in \cref{section:preliminaries}.
Our algorithm for convex-concave saddle-point problems 
associated with a general convex-concave function 
and some related fundamental results are stated in
\cref{section:APMM}.
We show three convergence results from our algorithm  for
 the convex-concave saddle-point problem associated with a special 
 but popular and useful class of convex-concave functions in
\cref{section:ConvergenceResults}.
In \cref{section:NumericalExperiments}, 
we implement our iterate scheme and some other ones associated with constant
involved parameters  in examples 
of a matrix game, 
a linear program in inequality form, and a least-squares problem with $\ell_{1}$
regularization. 
We  briefly summarize this work in \cref{section:Conclusion}.

 \section{Preliminaries} \label{section:preliminaries}
We shall use the following easy result multiple times later. 
 \begin{fact} {\rm \cite[Fact~1.1]{Oy2023subgradient}} \label{fact:saddlepoint}
 Let $(\bar{x}, \bar{y}) \in X \times Y$. 	
 The following statements are equivalent. 
 \begin{enumerate}
 	\item  \label{fact:saddlepoint:sp} 
	$(\bar{x}, \bar{y})$ is a saddle-point of the function $f$.
 	\item  \label{fact:saddlepoint:leq} 
	$(\forall (x,y) \in X \times Y)$ $f(\bar{x},y)- f(x,\bar{y}) \leq 0$.
 	\item  \label{fact:saddlepoint:partial} 
	$0 \in \partial_{x} \left( f +\iota_{X} \right) (\bar{x}, \bar{y}) $ and $0 \in \partial_{y} 
	\left( -f +\iota_{Y} \right) (\bar{x}, \bar{y}) $.
 \end{enumerate}
 \end{fact}

Below we naturally extend the definition of the partial primal-dual gap 
introduced by Chambolle and Pock in \cite[Page~121]{ChambollePock2011}
from a convex-concave function in the form of \cref{eq:fspecialform} 
to a general convex-concave function.  
Let $B_{1}$ and $B_{2}$ be two  nonempty subsets of $X$ and $Y$, respectively.  
The \emph{partial primal-dual gap} 
(\emph{associated with} $(f, B_{1}, B_{2})$) 
$G^{f}_{B_{1} \times B_{2}} :  X \times Y \to  \mathbf{R} \cup \{-\infty, +\infty\}$ 
is defined as 
\begin{align} \label{eq:pp-d}
(\forall (x,y) \in X \times Y) \quad 	
G^{f}_{B_{1} \times B_{2}} (x,y) = \sup\limits_{y' \in B_{2}} 
f(x,y') - \inf\limits_{x' \in B_{1}} f(x', y). 
\end{align}

\begin{example}  \label{example:toy:x2-y2} 
Consider the convex-concave function  $(\forall (x,y) \in \mathbf{R}^{2})$ 
$f(x,y) = x^{2}-y^{2}$. The following assertions hold. 
	\begin{enumerate}
		\item \label{example:x2-y2:G} 
		Set $B_{1}=B_{2} =[-1,1]$ and $B_{3}=B_{4}=[1,2]$. 
		Then $G^{f}_{B_{1} \times B_{2}} (0,0) =G^{f}_{B_{3} \times B_{4}} (1,1)  =0$.
		\item \label{example:x2-y2:01} $(0,0)$ is a saddle-point of $f$ and $(1,1)$ 
		is not a saddle-point of $f$.
	\end{enumerate} 
\end{example}

\begin{proof}
	\cref{example:x2-y2:G}: It is clear that 
\begin{align*}
G^{f}_{B_{1} \times B_{2}} (0,0) &= \sup_{y \in [-1,1]} (-y^{2}) - \inf_{x \in [-1,1]} 
x^{2} = 0-0=0;\\
G^{f}_{B_{3} \times B_{4}} (1,1) &=  \sup_{y \in [1,2]} (1-y^{2}) - \inf_{x \in [1,2]} 
(x^{2}-1) = 0-0=0.
\end{align*}

\cref{example:x2-y2:01}:	Note that 
\[ 
(\forall x \in X) (\forall y \in Y) \quad -y^{2} = f(0, y)  \leq  0=f(0,0)  
\leq f(x, 0) =x^{2},
\]
which guarantees that $(0,0)$ is a saddle-point of $f$. 
On the other hand, because
\[  
f(1,0) =1 > 0 = f(1,1) \text{ and }  f(1,1) =0 > -1 =f(0,1),
\]
$(1,1)$ is not a saddle-point of $f$.
\end{proof}	

The following result will be used in our convergence proof in 
\cref{section:ConvergenceResults} 
below.
\begin{lemma}\label{lemma:clustersaddletkxy}
Suppose that $(\forall (x,y) \in X \times Y)$  $f(\cdot, y)$ and $-f(x, \cdot)$ 
are proper, convex, and lower semicontinuous. 
Let $((x^{k}, y^{k}))_{k \in \mathbf{N}}$ be a bounded sequence  in $X\times Y$. 
Suppose that 
\begin{align} \label{eq:lemma:clustersaddletkxy}
(\forall (x,y) \in X \times Y) 
\quad 
\limsup_{k \to \infty} f(x^{k},y)  - f(x,y^{k}) \leq 0.
\end{align}	 
Then $((x^{k}, y^{k}))_{k \in \mathbf{N}}$  has at least one weakly 
sequential cluster point and every weakly sequential cluster point 
$(\bar{x}, \bar{y})$ of $((x^{k}, y^{k}))_{k \in \mathbf{N}}$  is a saddle-point of $f$.
\end{lemma}
\begin{proof}
Because $((x^{k}, y^{k}))_{k \in \mathbf{N}}$  is bounded in a Hilbert space, 
via \cite[Lemma~2.45]{BC2017}, 
it has at least one weakly sequential cluster point  $(\bar{x}, \bar{y})$, 
that is, there exits a subsequence  $((x^{k_{i}}, y^{k_{i}}))_{i \in \mathbf{N}}$   
of $((x^{k}, y^{k}))_{k \in \mathbf{N}}$  such that 
\begin{align} \label{eq:lemma:partialprimaldual:clusterpoint:weakly}
x^{k_{i}} \weakly \bar{x} \quad \text{and} \quad y^{k_{i}} \weakly   \bar{y}.
\end{align} 
Because $(\forall (x,y) \in X \times Y)$  $f(\cdot, y)$ and $-f(x, \cdot)$ 
are proper, convex, and lower semicontinuous, due to \cite[Theorem~9.1]{BC2017}, 
they are weakly sequentially lower semicontinuous. 
This connected with \cref{eq:lemma:clustersaddletkxy} entails that 
for every $(x,y) \in X \times Y$,
\begin{align*}
f(\bar{x},y) - f(x, \bar{y}) 
&  \stackrel{\cref{eq:lemma:partialprimaldual:clusterpoint:weakly}}{\leq}
\liminf_{i \to \infty} f(x^{k_{i}},y) + \liminf_{i \to \infty}- f(x,y^{k_{i}})\\
 &\,  \leq \, \liminf_{i \to \infty} \left( f(x^{k_{i}},y)  - f(x,y^{k_{i}}) \right)\\
&\,  \leq \, \limsup_{k \to \infty} f(x^{k},y)  - f(x,y^{k})  
    \stackrel{\cref{eq:lemma:clustersaddletkxy}}{\leq}    0,
\end{align*}
which, connecting with  \cref{fact:saddlepoint}, 
necessitates that $(\bar{x}, \bar{y})$ is a saddle-point of  $f$.
\end{proof}

The following \cref{lemma:partialprimaldual} illustrates that 
results in \cref{example:toy:x2-y2} are not coincidences. 
Note that related results of 
\cref{lemma:partialprimaldual}\cref{lemma:partialprimaldual:>=0}$\&$\cref{lemma:partialprimaldual:saddlepoint}
were claimed on \cite[Page~122]{ChambollePock2011}
for a partial primal-dual gap associated with 
a convex-concave function in the form of \cref{eq:fspecialform}.
\begin{lemma} \label{lemma:partialprimaldual}
Let $B_{1}$ and $B_{2}$ be two  nonempty subsets of $X$ and $Y$, respectively.  
Consider the  partial primal-dual gap 
$G^{f}_{B_{1} \times B_{2}} :  X \times Y \to \mathbf{R} \cup \{-\infty, +\infty\}$
defined in \cref{eq:pp-d}.
The following statements hold. 
	\begin{enumerate}
		\item \label{lemma:partialprimaldual:>=0} 
		Suppose that a saddle-point $(x^{*},y^{*}) $ of $f$ is included in  
		$ B_{1} \times B_{2}$. 
		Then $(\forall (x,y) \in X \times Y)$ $G^{f}_{B_{1} \times B_{2}} (x,y) \geq 0$. 
		Moreover,  $	G^{f}_{B_{1} \times B_{2}} (x^{*},y^{*})  =0$.
		\item \label{lemma:partialprimaldual:saddlepoint} 
		Let $(\bar{x}, \bar{y}) \in \inte B_{1} \times \inte B_{2}$. 
		Suppose that  $	G^{f}_{B_{1} \times B_{2}}  (\bar{x}, \bar{y}) =0$. 
		Then $(\bar{x}, \bar{y}) $ is a saddle-point of $f$.
		\item \label{lemma:partialprimaldual:clusterpoint} 
		Suppose that $(\forall (x,y) \in X \times Y)$  $f(\cdot, y)$ 
		and $-f(x, \cdot)$ are proper, convex, and lower semicontinuous. 
		Let $((x^{k}, y^{k}))_{k \in \mathbf{N}}$ be a bounded sequence  in $X \times Y$.
 Suppose that, for every $k$ large enough, $G^{f}_{X \times Y} (x^{k},y^{k}) \leq t_{k}$, 
 where $(t_{k})_{k \in \mathbf{N}}$  is a sequence in $\mathbf{R}$ 
 with $\limsup_{k \to \infty} t_{k} =0$ 
 $($e.g., $(\forall k \in \mathbf{N})$ $t_{k} \equiv 0$$)$. 
Then $((x^{k}, y^{k}))_{k \in \mathbf{N}}$  has at least one weakly 
sequential cluster point and every weakly sequential cluster point 
$(\bar{x}, \bar{y})$ of $((x^{k}, y^{k}))_{k \in \mathbf{N}}$  is a saddle-point of $f$.
\end{enumerate}
\end{lemma}	
	
\begin{proof}
\cref{lemma:partialprimaldual:>=0}: 
Let $(x,y) \in X \times Y$.  Because  $(x^{*},y^{*}) \in  B_{1} \times B_{2} $ 
is a   saddle-point of $f$, via \cref{eq:pp-d}, 
\[ 
G^{f}_{B_{1} \times B_{2}} (x,y) 
= \sup\limits_{y' \in B_{2}} f(x,y') - \inf\limits_{x' \in B_{1}} f(x', y) \geq 
f(x,y^{*}) - f(x^{*}, y) \geq 0,
\]
where in the last inequality we use  \cref{fact:saddlepoint}  
and the assumption that $(x^{*},y^{*}) $ is a saddle-point of $f$. 

In addition, applying  the assumption 
$(x^{*},y^{*}) \in B_{1} \times B_{2} \subseteq X \times Y$, 
we observe that 
\begin{align*}
f(x^{*},y^{*})  \leq&	\sup\limits_{y' \in B_{2}} f(x^{*},y') \leq \sup\limits_{y' \in Y} 
f(x^{*},y') \leq f(x^{*},y^{*}), \text{ and}\\
 f(x^{*},y^{*})  \geq&  \inf\limits_{x' \in B_{1}} f(x', y^{*}) \geq \inf\limits_{x' \in X} f(x', 
 y^{*}) \geq   f(x^{*},y^{*}),
\end{align*}
which implies that 
$	G^{f}_{B_{1} \times B_{2}} (x^{*},y^{*}) 
=  \sup\limits_{y' \in B_{2}} f(x^{*},y') - \inf\limits_{x' \in B_{1}} f(x', y^{*}) 
=f(x^{*},y^{*}) -f(x^{*},y^{*})     =0$.
	
\cref{lemma:partialprimaldual:saddlepoint}: 
Let $(x,y) \in X \times Y$. Because $(\bar{x}, \bar{y}) \in \inte B_{1} \times \inte B_{2}$, 
there exists $\varepsilon \in ]0,1[$ such that 
\begin{align} \label{eq:lemma:partialprimaldual:saddlepoint:inB}
\bar{x} +\varepsilon (x- \bar{x}) \in B_{1} \quad \text{and} 
\quad  \bar{y} +\varepsilon (y- \bar{y}) \in B_{2}.
\end{align} 
By assumptions,
\begin{align} \label{eq:lemma:partialprimaldual:saddlepoint:supinf}
0 = G^{f}_{B_{1} \times B_{2}}  (\bar{x}, \bar{y}) 
= \sup\limits_{y' \in B_{2}} f(\bar{x} ,y') - \inf\limits_{x' \in B_{1}} f(x', \bar{y}). 
\end{align}
Applying the concavity of $f(\bar{x}, \cdot)$ and 
the convexity of $f(\cdot, \bar{y})$ and bearing 
\cref{eq:lemma:partialprimaldual:saddlepoint:inB} and 
\cref{eq:lemma:partialprimaldual:saddlepoint:supinf} in mind, we derive that 
\begin{subequations}
\begin{align}
& \varepsilon  f(\bar{x} ,y) + (1- \varepsilon) f (\bar{x}, \bar{y})  
\leq f(\bar{x} , \bar{y} +\varepsilon (y- \bar{y}))   
\stackrel{\cref{eq:lemma:partialprimaldual:saddlepoint:inB}}{\leq}	
\sup\limits_{y' \in B_{2}} f(\bar{x} ,y')
 \stackrel{\cref{eq:lemma:partialprimaldual:saddlepoint:supinf}}{=}
  \inf\limits_{x' \in B_{1}} f(x', \bar{y}) \leq f (\bar{x}, \bar{y}); 
  \label{eq:lemma:partialprimaldual:saddlepoint:y}\\
&  \varepsilon f(x,\bar{y}) 
+ (1- \varepsilon)  f (\bar{x}, \bar{y}) 
 \geq f(\bar{x} +\varepsilon (x- \bar{x}), \bar{y}) 
 \stackrel{\cref{eq:lemma:partialprimaldual:saddlepoint:inB}}{\geq}  
 \inf\limits_{x' \in B_{1}} f(x', \bar{y}) 
 \stackrel{\cref{eq:lemma:partialprimaldual:saddlepoint:supinf}}{=} 
 \sup\limits_{y' \in B_{2}} f(\bar{x} ,y') \geq f(\bar{x}, \bar{y}).
 \label{eq:lemma:partialprimaldual:saddlepoint:x}
	\end{align}
\end{subequations}
It is clear that 
\begin{align*}
& \cref{eq:lemma:partialprimaldual:saddlepoint:y} 
\Rightarrow f(\bar{x} ,y) \leq f (\bar{x}, \bar{y});\\
& \cref{eq:lemma:partialprimaldual:saddlepoint:x} 
\Rightarrow f(x,\bar{y})  \geq  f(\bar{x}, \bar{y}).
\end{align*}
Combine the last two inequalities with the definition \cref{eq:solution} 
of the saddle-point to conclude that $(\bar{x}, \bar{y}) $ is a saddle-point of $f$.
	
\cref{lemma:partialprimaldual:clusterpoint}: 
Because for every $k$ large enough   $G^{f}_{X \times Y} (x^{k},y^{k}) \leq t_{k}$, 
we have that  there exists $N \in \mathbf{N}$ such that  
\begin{align*}
(\forall (x,y) \in X\times Y)(\forall k \geq N) \quad	
t_{k}  & \geq G^{f}_{X \times Y} 
 (x^{k},y^{k})  = \sup\limits_{y' \in Y} f(x^{k},y') 
 - \inf\limits_{x' \in X} f(x', y^{k})\\
 &  \geq  f(x^{k},y) - f(x,y^{k}),
\end{align*} 
which ensures that $ \limsup_{k \to \infty} f(x^{k},y)  - f(x,y^{k})  \leq 
\limsup_{k \to \infty}  t_{k} \leq 0$.
 Then applying \cref{lemma:clustersaddletkxy},  we obtain the required result. 
\end{proof}

\section{Alternating proximity mapping method} \label{section:APMM}

In this section, we introduce an iterate scheme for 
solving convex-concave saddle-point problems associated with
a general convex-concave function.
We also provide some fundamental results related to our iterates,
which will be used in our convergence proofs in 
\cref{section:ConvergenceResults} below.

In the rest of this work, 
we assume that $(\forall (x,y) \in X \times Y)$  $f(\cdot, y)$ and $-f(x, \cdot)$ 
are proper, convex, and lower semicontinuous.

\subsection{Algorithm}
Henceforth, we assume that  
\[ 
(\tau_{k})_{k \in \mathbf{N}} \text{ and } (\sigma_{k})_{k \in \mathbf{N}} 
\text{	are in } \mathbf{R}_{++}, \text{and } 
(\alpha_{k})_{k \in \mathbf{N} \cup \{-1\}} \text{ and } (\beta_{k})_{k \in \mathbf{N} \cup 
\{-1\}}
\text{ are in } \mathbf{R}_{+}.
\]

Let $(x^{0}, y^{0}) \in X \times Y$. 
Set $\bar{x}^{0}=x^{0}$ and $\bar{y}^{0} =y^{0}$.
For every $k \in \mathbf{N}$, 
the alternating proximal mapping method  considered in 
this section is updated as follows 
 \begin{subequations} \label{eq:PPalgorithm:1}
 	\begin{align}
 		&y^{k+1} 
 		= \Prox^{Y}_{-\sigma_{k} f(\bar{x}^{k}, \cdot )}  y^{k} 
 		:= \argmax_{y \in Y}  \left\{  f(\bar{x}^{k}, y ) 
		- \frac{1}{2 \sigma_{k}} \norm{y- y^{k} }^{2} \right\};  
		\label{eq:PPalgorithm:1:y}\\
 		& \bar{y}^{k+1} = y^{k+1} + \beta_{k} (y^{k+1} -y^{k}); 
		\label{eq:PPalgorithm:1:bary}\\
 		&x^{k+1} 
 		= \Prox^{X}_{\tau_{k} f(\cdot, \bar{y}^{k+1}) } x^{k}
 		=\argmin_{x \in X} \left\{ f(x, \bar{y}^{k+1}) +\frac{1}{2 \tau_{k}} 
		\norm{x - x^{k}}^{2} \right\};  \label{eq:PPalgorithm:1:x}\\
 		& \bar{x}^{k+1} = x^{k+1} + \alpha_{k} (x^{k+1} -x^{k}).  
		\label{eq:PPalgorithm:1:barx}
 	\end{align}
 \end{subequations} 
Note that to coincide with our initialization $\bar{x}^{0}=x^{0}$ 
and $\bar{y}^{0} =y^{0}$, we introduce $x^{-1}=x^{0}$ and $y^{-1}=y^{0}$.

 \begin{fact}  \label{fact:proxsubdiff}
	Let $g: \mathcal{H} \to \mathbf{R}\cup \{- \infty\}$ be a proper,   convex, and lower 
	semicontinuous function, 
	and let $X$ be a nonempty, closed, and convex subset of $\mathcal{H}$. Let 
	$\bar{x}$ and 
	$p$ be in $\mathcal{H}$. Then  
	\begin{align*}
		p = \argmin_{x \in X}  \{ g(x) + \frac{1}{2} \norm{x -\bar{x}}^{2} \}
		\Leftrightarrow   (\forall x \in X) \quad \innp{\bar{x}-p, x-p} +g(p) \leq g(x).
	\end{align*}
\end{fact}
\begin{proof}
	The proof is almost the same as \cite[Proposition~12.26]{BC2017}.
\end{proof}

In view of \cref{fact:proxsubdiff} and \cref{eq:PPalgorithm:1:barx} and 
\cref{eq:PPalgorithm:1:bary},
 we observe that 
 \begin{subequations} \label{eq:PPalgorithm:2}
 	\begin{align}
 		&(\forall x \in X) \quad \innp{ \frac{x^{k}-x^{k+1}}{\tau_{k}}, x-x^{k+1}} 
 			+f(x^{k+1}, \bar{y}^{k+1}) \leq f(x, \bar{y}^{k+1}); 	\label{eq:PPalgorithm:2:x}\\
 		&(\forall y \in Y) \quad \innp{\frac{y^{k}-y^{k+1}}{\sigma_{k}}, y-y^{k+1}}	
 			- f (\bar{x}^{k}, y^{k+1}) \leq - f (\bar{x}^{k}, y). \label{eq:PPalgorithm:2:y}
 	\end{align}
 \end{subequations}

%

\begin{fact} \label{fact:abc}
Let $\mathcal{H}$ be a Hilbert space and let $a, b,$ and $c$
be points in $\mathcal{H}$. Then 
\[ 
\innp{a-b, c-b} = \frac{1}{2} \left(\norm{a-b}^{2} + \norm{c-b}^{2} -\norm{c-a}^{2}\right).
\]
\end{fact}
\begin{proof}
It is clear that
\begin{align*}
&\norm{a-b}^{2} + \norm{c-b}^{2} -\norm{c-a}^{2} \\
=& \norm{a}^{2} -2 \innp{a,b} +\norm{b}^{2} +\norm{c}^{2} 
-2 \innp{c,b} +\norm{b}^{2} -\norm{c}^{2} +2 \innp{c,a} -\norm{a}^{2} \\
=&  -2 \innp{a,b}+ 2 \innp{b,b} - 2 \innp{c,b} +2 \innp{c,a} \\
=& 2	\innp{a-b, c-b}.
\end{align*}
\end{proof}

\begin{lemma} \label{lemma:xkykinequalities}
Let $k \in \mathbf{N}$.	 
The following assertions hold. 
	 \begin{enumerate}
	 	\item \label{lemma:xkykinequalities:innpy}
		 $(\forall y \in Y)$ $\innp{\frac{y^{k}-y^{k+1}}{\sigma_{k}}, y - y^{k+1}} 
		 - f (\bar{x}^{k}, y^{k+1}) \leq - f(\bar{x}^{k},y)$.
	 	\item \label{lemma:xkykinequalities:innpx} 
		$(\forall x \in X)$ $\innp{\frac{x^{k} -x^{k+1}}{\tau_{k}}, x -x^{k+1}} 
		+ f(x^{k+1}, \bar{y}^{k+1}) \leq f(x, \bar{y}^{k+1})$.
	 	\item \label{lemma:xkykinequalities:normy} $(\forall y \in Y)$ 
		$\frac{1}{2 \sigma_{k}}\left( \norm{y^{k} -y^{k+1}}^{2} 
		+\norm{y-y^{k+1}}^{2} -\norm{y - y^{k}}^{2} \right) + f(\bar{x}^{k},y) 
		- f (\bar{x}^{k}, y^{k+1}) \leq 0$.
	 \item  \label{lemma:xkykinequalities:normx} $(\forall x \in X)$ 
	 $\frac{1}{2 \tau_{k}} \left( \norm{x^{k} -x^{k+1}}^{2} 
	 + \norm{x - x^{k+1}}^{2} -\norm{x-x^{k}}^{2} \right) + f(x^{k+1}, \bar{y}^{k+1}) 
	 - f(x, \bar{y}^{k+1}) \leq 0$.
	 \end{enumerate}
\end{lemma}

\begin{proof}
\cref{lemma:xkykinequalities:innpy}$\&$\cref{lemma:xkykinequalities:innpx}: 
These inequalities follow immediately from the definition of subdifferential 
and \cref{eq:PPalgorithm:2}.
	
\cref{lemma:xkykinequalities:normy}$\&$\cref{lemma:xkykinequalities:normx}:  \cref{lemma:xkykinequalities:normy} (resp.\,\cref{lemma:xkykinequalities:normx}) 
is clear from \cref{lemma:xkykinequalities:innpy}  
(resp.\,\cref{lemma:xkykinequalities:innpx}) and \cref{fact:abc}.
\end{proof}

\subsection{Auxiliary results}
In the rest of this section, $((x^{k},y^{k}))_{k \in \mathbf{N}}$ is generated by the
iterate scheme \cref{eq:PPalgorithm:1}; moreover
\begin{align}  \label{eq:x_ky_k}
	(\forall k \in \mathbf{N}) 
	~ \hat{x}_{k} := \frac{1}{k+1} \sum^{k}_{i=0} x^{i+1} 
	=  \frac{1}{k+1} \sum^{k+1}_{i=1} x^{i}   
	\text{ and }   \hat{y}_{k} := \frac{1}{k+1} \sum^{k}_{i=0} y^{i+1} 
	=   \frac{1}{k+1} \sum^{k+1}_{i=1} y^{i}.
\end{align}

\begin{lemma} \label{lemma:fsumxkykineq}
	 Let $k \in \mathbf{N}$ and let $(x,y) \in X \times Y$. 
	 The following hold.
	 \begin{enumerate}
	 	\item \label{lemma:fsumxkykineq:y} 
		We have that 
	 	\begin{align*}
	 		&f(\hat{x}_{k}, y) - \frac{1}{k+1} \sum^{k}_{i=0} \left(f(x^{i+1},y) 
			-f(\bar{x}^{i},y) + f(\bar{x}^{i}, y^{i+1})   \right) \\
	 		\leq & \frac{1}{k+1}  \sum^{k}_{i=0} \frac{1}{2 \sigma_{i}} 
			 \left( \norm{y- y^{i}}^{2} - \norm{y-y^{i+1}}^{2} - \norm{y^{i}-y^{i+1}}^{2} \right).
	 	\end{align*}
	 	\item \label{lemma:fsumxkykineq:x} We have that 
	 	\begin{align*}
	 		&-f(x, \hat{y}_{k}) + \frac{1}{k+1} \sum^{k}_{i=0} \left( f(x,y^{i+1}) 
			+ f(x^{i+1}, \bar{y}^{i+1}) -f(x,\bar{y}^{i+1}) \right) \\
	 		\leq &\frac{1}{k+1}  \sum^{k}_{i=0} \frac{1}{2 \tau_{i}} \left(  \norm{x -x^{i}}^{2} 
			-\norm{x -x^{i+1}}^{2} - \norm{x^{i} -x^{i+1}}^{2}\right).
	 	\end{align*}
	 	\item  \label{lemma:fsumxkykineq:f} We have that 
	 	\begin{align*}
	 		&f(\hat{x}_{k}, y)   -f(x, \hat{y}_{k}) \\
	 		\leq & \frac{1}{k+1}  \sum^{k}_{i=0} \frac{1}{2 \sigma_{i}}  \left( \norm{y- y^{i}}^{2} 
			- \norm{y-y^{i+1}}^{2} - \norm{y^{i}-y^{i+1}}^{2} \right)\\
	 		& +\frac{1}{k+1}  \sum^{k}_{i=0} \frac{1}{2 \tau_{i}} \left(  \norm{x -x^{i}}^{2} 
			-\norm{x -x^{i+1}}^{2} - \norm{x^{i} -x^{i+1}}^{2} \right) \\
	 		&+  \frac{1}{k+1} \sum^{k}_{i=0} \left(f(x^{i+1},y) -f(\bar{x}^{i},y) 
			+ f(\bar{x}^{i}, y^{i+1}) -f(x,y^{i+1}) - f(x^{i+1}, \bar{y}^{i+1}) + f(x,\bar{y}^{i+1})  \right). 
	 	\end{align*}
	 \end{enumerate}
\end{lemma}

\begin{proof}
	Because $f(\cdot, y)$ is convex and $f(x, \cdot)$ is concave, we observe that
	\begin{subequations} 
		\begin{align}
			&\frac{1}{k+1} \sum^{k}_{i=0}f(x^{i+1},y) 
			\geq f\left( \frac{1}{k+1}\sum^{k}_{i=0}x^{i+1},y\right)=f(\hat{x}_{k},y); \label{eq:lemma:fsumxkykineq:xk}\\
			&\frac{1}{k+1} \sum^{k}_{i=0}f(x,y^{i+1}) 
			\leq f \left(x, \frac{1}{k+1}\sum^{k}_{i=0} y^{i+1} \right) =f(x,\hat{y}_{k}). \label{eq:lemma:fsumxkykineq:yk}
		\end{align}
	\end{subequations}

\cref{lemma:fsumxkykineq:y}:	
According to \cref{lemma:xkykinequalities}\cref{lemma:xkykinequalities:normy}, 
for every $i \in \{0,1,\ldots,k\}$,
\[  
f(x^{i+1},y) -f(x^{i+1},y) 	 + f(\bar{x}^{i},y) - f (\bar{x}^{i}, y^{i+1}) \leq  
\frac{1}{2 \sigma_{i}}\left( \norm{y - y^{i}}^{2} -\norm{y-y^{i+1}}^{2}  
 - \norm{y^{i} -y^{i+1}}^{2}   \right).
\]
Sum  the inequalities above over $i$ from $0$ to $k$. 
Multiply both sides of the new inequality with $\frac{1}{k+1}$  and 
use \cref{eq:lemma:fsumxkykineq:xk} to deduce  \cref{lemma:fsumxkykineq:y}.

\cref{lemma:fsumxkykineq:x}: As a consequence of \cref{lemma:xkykinequalities}\cref{lemma:xkykinequalities:normx},  
for every $i \in \{0,1,\ldots,k\}$,
\[ 
-f(x,y^{i+1}) + f(x,y^{i+1}) + f(x^{i+1}, \bar{y}^{i+1}) 
- f(x, \bar{y}^{i+1}) \leq \frac{1}{2 \tau_{i}} \left(  \norm{x-x^{i}}^{2} 
-\norm{x - x^{i+1}}^{2}   - \norm{x^{i} -x^{i+1}}^{2} \right).
\]
Similarly with the proof of \cref{lemma:fsumxkykineq:y}, 
sum  the last inequalities above over $i$ from $0$ to $k$. 
Then multiply both sides of the new inequality with $\frac{1}{k+1}$  
and use \cref{eq:lemma:fsumxkykineq:yk} to deduce  \cref{lemma:fsumxkykineq:x}.

\cref{lemma:fsumxkykineq:f}: Add the inequalities in \cref{lemma:fsumxkykineq:y} and \cref{lemma:fsumxkykineq:x} to establish the required inequality 
in \cref{lemma:fsumxkykineq:f}.
\end{proof}

\begin{corollary} \label{corollary:basicxkyksumineq}
	Let $k \in \mathbf{N}$. The following statements hold. 
	\begin{enumerate}
		\item   \label{corollary:basicxkyksumineq:xyk+1} The following inequalities hold.
		\begin{subequations} \label{eq:corollary:basicxkyksumineq:xyk+1:yk+1}
			\begin{align}
				&f(\hat{x}_{k}, \hat{y}_{k}) - \frac{1}{k+1} \sum^{k}_{i=0} \left(f(x^{i+1},\hat{y}_{k}) 
				-f(\bar{x}^{i},\hat{y}_{k}) + f(\bar{x}^{i}, y^{i+1})   \right) \\
				\leq & \frac{1}{k+1}  \sum^{k}_{i=0} \frac{1}{2 \sigma_{i}}  
				\left( \norm{\hat{y}_{k}- y^{i}}^{2} - \norm{\hat{y}_{k}-y^{i+1}}^{2}
				 - \norm{y^{i}-y^{i+1}}^{2} \right),
			\end{align}
		\end{subequations}		
	and
	\begin{subequations} \label{eq:corollary:basicxkyksumineq:xyk+1:xk+1}
			\begin{align}
			&-f(\hat{x}_{k}, \hat{y}_{k}) 
			+ \frac{1}{k+1} \sum^{k}_{i=0} \left( f(\hat{x}_{k},y^{i+1}) 
			+ f(x^{i+1}, \bar{y}^{i+1}) -f(\hat{x}_{k},\bar{y}^{i+1}) \right)\\
			\leq & \frac{1}{k+1}  \sum^{k}_{i=0} \frac{1}{2 \tau_{i}} \left(  \norm{\hat{x}_{k} -x^{i}}^{2} -\norm{\hat{x}_{k} -x^{i+1}}^{2} 
			- \norm{x^{i} -x^{i+1}}^{2} \right).
		\end{align}
	\end{subequations}

		\item  \label{corollary:basicxkyksumineq:xyhat} 
		Suppose that $(x^{*}, y^{*}) \in X \times Y$ is a 
		saddle-point of $f$. Then 
\begin{subequations} \label{eq:corollary:basicxkyksumineq:xyhat:yhat} 
\begin{align}
&f(x^{*}, y^{*}) - \frac{1}{k+1} \sum^{k}_{i=0} \left(f(x^{i+1},  y^{*})
-f(\bar{x}^{i}, y^{*}) + f(\bar{x}^{i}, y^{i+1})   \right) \\
\leq & \frac{1}{k+1}  \sum^{k}_{i=0} \frac{1}{2 \sigma_{i}}  \left( \norm{y^{*} - y^{i}}^{2}
- \norm{y^{*} -y^{i+1}}^{2} - \norm{y^{i}-y^{i+1}}^{2} \right)
\end{align}
\end{subequations}
and 
\begin{subequations} \label{eq:corollary:basicxkyksumineq:xyhat:xhat} 
\begin{align}
&-f(x^{*}, y^{*}) + \frac{1}{k+1} \sum^{k}_{i=0} \left( f(x^{*} ,y^{i+1}) + f(x^{i+1}, 
\bar{y}^{i+1}) -f(x^{*},\bar{y}^{i+1}) \right) \\
\leq & \frac{1}{k+1}  \sum^{k}_{i=0} \frac{1}{2 \tau_{i}} \left(  \norm{x^{*} -x^{i}}^{2} 
-\norm{x^{*} -x^{i+1}}^{2} - \norm{x^{i} -x^{i+1}}^{2} \right).
\end{align}
\end{subequations}
\item \label{corollary:basicxkyksumineq:leqhat}
Suppose that $(x^{*}, y^{*}) \in X \times 
Y$ is a saddle-point of $f$. 
Then 
\begin{align*}
	&
  \sum^{k}_{i=0} \left( f(x^{*},y^{i+1}) + f(x^{i+1}, \bar{y}^{i+1}) -f(x^{*},\bar{y}^{i+1}) 
  -f(x^{i+1}, y^{*}) + f(\bar{x}^{i}, y^{*}) - f(\bar{x}^{i}, y^{i+1})    \right) \\
	\leq &    \sum^{k}_{i=0} \frac{1}{2 \sigma_{i}}  \left( \norm{ y^{*}- y^{i}}^{2} 
	- \norm{ y^{*}-y^{i+1}}^{2} - \norm{y^{i}-y^{i+1}}^{2} \right) \\
	& +  \sum^{k}_{i=0} \frac{1}{2 \tau_{i}} \left(  \norm{x^{*} -x^{i}}^{2}
	   -\norm{x^{*} -x^{i+1}}^{2} - \norm{x^{i} -x^{i+1}}^{2}\right).
\end{align*}
	\end{enumerate}
\end{corollary}

\begin{proof}
	\cref{corollary:basicxkyksumineq:xyk+1}: 
	Substitute $y$ in \cref{lemma:fsumxkykineq}\cref{lemma:fsumxkykineq:y} 
	(resp.\,$x$ in \cref{lemma:fsumxkykineq}\cref{lemma:fsumxkykineq:x})  
	with $y =\hat{y}_{k}$ (resp.\,$x=\hat{x}_{k}$) to obtain
	 \cref{eq:corollary:basicxkyksumineq:xyk+1:yk+1} 
	 (resp.\,\cref{eq:corollary:basicxkyksumineq:xyk+1:xk+1}).
	
	\cref{corollary:basicxkyksumineq:xyhat}: Because $(x^{*}, y^{*}) \in X \times Y$ is a 
	saddle-point of $f$, via \cref{eq:solution}, 
	we have that 
	\begin{subequations}
		\begin{align}
	&	f(x^{*}, y^{*}) \leq	f(\hat{x}_{k}, y^{*}); 
	\label{eq:corollary:basicxkyksumineq:xyhat:f}\\
	&  	-f(x^{*}, y^{*})  \leq -f(x^{*}, \hat{y}_{k}).  
	\label{eq:corollary:basicxkyksumineq:xyhat:-f}
		\end{align}
	\end{subequations}
	
	Replace $y$ in \cref{lemma:fsumxkykineq}\cref{lemma:fsumxkykineq:y} (resp.\,$x$ in 
	\cref{lemma:fsumxkykineq}\cref{lemma:fsumxkykineq:x})  by $y =y^{*}$ 
	(resp.\,$x=x^{*}$) and apply \cref{eq:corollary:basicxkyksumineq:xyhat:f} 
	(resp.\,\cref{eq:corollary:basicxkyksumineq:xyhat:-f})  to derive 
	\cref{eq:corollary:basicxkyksumineq:xyhat:yhat}  
	(resp.\,\cref{eq:corollary:basicxkyksumineq:xyhat:xhat}).
	
	\cref{corollary:basicxkyksumineq:leqhat}:  Substitute $y=y^{*}$ in 
	\cref{lemma:fsumxkykineq}\cref{lemma:fsumxkykineq:y} and replace $x=x^{*}$ in 
	\cref{lemma:fsumxkykineq}\cref{lemma:fsumxkykineq:x}. 
	Then add these two new 
	inequalities to get that
	\begin{align*}
		&f(\hat{x}_{k}, y^{*}) -f(x^{*}, \hat{y}_{k}) \\
		&  +
		\frac{1}{k+1} \sum^{k}_{i=0} \left( f(x^{*} ,y^{i+1}) + f(x^{i+1}, \bar{y}^{i+1}) 
		-f(x^{*},\bar{y}^{i+1}) -f(x^{i+1}, y^{*}) + f(\bar{x}^{i}, y^{*}) - f(\bar{x}^{i}, y^{i+1})    
		\right)  \\
		\leq  & \frac{1}{k+1}  \sum^{k}_{i=0} \frac{1}{2 \sigma_{i}}  \left( \norm{ y^{*} - 
		y^{i}}^{2} - \norm{ y^{*} -y^{i+1}}^{2} - \norm{y^{i}-y^{i+1}}^{2} \right)\\
		&+  \frac{1}{k+1}  \sum^{k}_{i=0} \frac{1}{2 \tau_{i}} \left(  \norm{x^{*} -x^{i}}^{2} 
		-\norm{x^{*} -x^{i+1}}^{2} - \norm{x^{i} -x^{i+1}}^{2} \right).
	\end{align*} 
Because $ (x^{*}, y^{*})$ is  a saddle-point of $f$, via \cref{fact:saddlepoint}, we have 
that 
$f(\hat{x}_{k}, y^{*}) -f(x^{*}, \hat{y}_{k}) \geq 0$. Altogether, we obtain the required 
inequality in \cref{corollary:basicxkyksumineq:leqhat}.
\end{proof}

The following result provides upper and lower bounds of the sequence $\left( 
f(\hat{x}_{k}, \hat{y}_{k}) -f(x^{*}, y^{*}) \right)_{k \in \mathbf{N}}$, 
which plays an essential role in the proof of  $f(\hat{x}_{k}, \hat{y}_{k}) \to f(x^{*}, 
y^{*})$
in \cref{theorem:convergencef}\cref{theorem:convergencef:fktohat} below.
This result may also be helpful to deduce more convergence results on
 $f(\hat{x}_{k},\hat{y}_{k}) \to f(x^{*},y^{*})$ associated with a
 general convex-concave function f.

\begin{proposition} \label{proposition:fxyk+1-xyhat}
Suppose that $(x^{*}, y^{*}) \in X \times Y$ is a saddle-point of $f$. 	Let $k \in 
\mathbf{N}$.  
\begin{enumerate}
\item  \label{proposition:fxyk+1-xyhat:k-hat} The following inequality holds.
\begin{align*}
& f(\hat{x}_{k}, \hat{y}_{k}) -f(x^{*}, y^{*})\\
 \leq  &\frac{1}{k+1}  \sum^{k}_{i=0} \frac{1}{2 \tau_{i}} \left(  \norm{x^{*} -x^{i}}^{2} 
 -\norm{x^{*} -x^{i+1}}^{2} - \norm{x^{i} -x^{i+1}}^{2} \right)\\
&+ \frac{1}{k+1}  \sum^{k}_{i=0} \frac{1}{2 \sigma_{i}}  \left( \norm{\hat{y}_{k}- y^{i}}^{2}
- \norm{\hat{y}_{k}-y^{i+1}}^{2} - \norm{y^{i}-y^{i+1}}^{2} \right) \\
&	+ \frac{1}{k+1} \sum^{k}_{i=0} \left(f(x^{i+1},\hat{y}_{k}) 
-f(\bar{x}^{i},\hat{y}_{k}) + f(\bar{x}^{i}, y^{i+1})   -f(x^{*},y^{i+1}) - f(x^{i+1}, 
\bar{y}^{i+1}) + f(x^{*},\bar{y}^{i+1}) \right).
\end{align*}
\item   \label{proposition:fxyk+1-xyhat:hat-k} We have the inequality below. 
\begin{align*}
&f(x^{*}, y^{*}) 	-f(\hat{x}_{k}, \hat{y}_{k}) \\
\leq &  \frac{1}{k+1}  \sum^{k}_{i=0} \frac{1}{2 \tau_{i}} \left(  \norm{\hat{x}_{k} -x^{i}}^{2} 
-\norm{\hat{x}_{k} -x^{i+1}}^{2} - \norm{x^{i} -x^{i+1}}^{2} \right)\\
&+  \frac{1}{k+1}  \sum^{k}_{i=0} \frac{1}{2 \sigma_{i}}  \left( \norm{y^{*} - y^{i}}^{2} - 
\norm{y^{*}-y^{i+1}}^{2} - \norm{y^{i}-y^{i+1}}^{2} \right)\\
& + \frac{1}{k+1} \sum^{k}_{i=0} \left(f(x^{i+1},y^{*}) -f(\bar{x}^{i},y^{*}) + f(\bar{x}^{i}, 
y^{i+1})   -f(\hat{x}_{k},y^{i+1}) - f(x^{i+1}, \bar{y}^{i+1}) + f(\hat{x}_{k},\bar{y}^{i+1}) 
\right).
\end{align*}
\end{enumerate}
\end{proposition}

\begin{proof}
\cref{proposition:fxyk+1-xyhat:k-hat}:  
Add \cref{eq:corollary:basicxkyksumineq:xyk+1:yk+1} 
and \cref{eq:corollary:basicxkyksumineq:xyhat:xhat}  
in \cref{corollary:basicxkyksumineq} to establish the required inequality 
in \cref{proposition:fxyk+1-xyhat:k-hat}. 
	
	\cref{proposition:fxyk+1-xyhat:hat-k}:  
	Sum \cref{eq:corollary:basicxkyksumineq:xyhat:yhat} 
	and \cref{eq:corollary:basicxkyksumineq:xyk+1:xk+1}   
	 in \cref{corollary:basicxkyksumineq} to deduce the desired inequality in  
	 \cref{proposition:fxyk+1-xyhat:hat-k}.
\end{proof}

\begin{lemma} \label{lemma:fKuv}
Consider the operator $K: \mathcal{H}_{1} \to \mathcal{H}_{2}$ and functions 
$g: \mathcal{H}_{1} \to \mathbf{R} \cup \{+\infty\}$ and 
$h: \mathcal{H}_{2} \to \mathbf{R} \cup \{+\infty\}$. 
Let $(\forall i \in \{1,2,3\})$ $(u_{i}, v_{i}) \in \mathcal{H}_{1} \times \mathcal{H}_{2} $.
 
Define $F: X\times Y \to \mathbf{R} \cup \{- \infty, +\infty\} : (u,v) \mapsto \innp{Ku, v}+ 
g(u) - h(v)$. 
Then 
\begin{align*}
	&F(u_{1},v_{1}) - F(u_{1},v_{2}) +F(u_{2},v_{2}) - F(u_{2},v_{3}) +F(u_{3},v_{3}) 
	-F(u_{3},v_{1})\\ 
	=&\innp{K(u_{1} -u_{2}), v_{3}-v_{2}} - \innp{K(u_{1}-u_{3}), v_{3}-v_{1}}.
\end{align*}
\end{lemma}

\begin{proof}
According to the definition of the function $F$, it is clear that
\begin{align*}
&F(u_{1},v_{1}) - F(u_{1},v_{2}) +F(u_{2},v_{2}) 
- F(u_{2},v_{3}) +F(u_{3},v_{3}) -F(u_{3},v_{1})\\
 =& \innp{Ku_{1}, v_{1}-v_{2}} - h(v_{1}) + h (v_{2}) 
 +\innp{Ku_{2}, v_{2}-v_{3}} - h(v_{2}) + h (v_{3}) 
 +\innp{Ku_{3}, v_{3}-v_{1}} - h(v_{3}) + h (v_{1})\\
=& \innp{Ku_{1}, v_{1}-v_{3} +v_{3}-v_{2}}  +\innp{Ku_{2}, v_{2}-v_{3}}  
+\innp{Ku_{3}, v_{3}-v_{1}}\\
=&\innp{K(u_{1} -u_{2}), v_{3}-v_{2}} - \innp{K(u_{1}-u_{3}), v_{3}-v_{1}}.
\end{align*} 
\end{proof}

\section{Convergence results} \label{section:ConvergenceResults}
Let $K \in \mathcal{B}(\mathcal{H}_{1} , \mathcal{H}_{2} )$, i.e., $K : \mathcal{H}_{1} \to 
\mathcal{H}_{2}$ is a continuous linear operator with the operator norm
\begin{align} \label{eq:definenormK}
\norm{K} := \sup \{ \norm{Kx} ~:~ x \in \mathcal{H}_{1} \text{ with } \norm{x} \leq 1 \}.
\end{align}
Let $g: \mathcal{H}_{1} \to \mathbf{R} \cup \{ +\infty\}$ and 
$h: \mathcal{H}_{2} \to  \mathbf{R} \cup \{ +\infty\}$ satisfy that 
both $g$ and $h$ are proper, convex, and lower semicontinuous. 
In the whole section, we assume that 
\begin{align}  \label{eq:fspecialdefine}
(\forall (x,y) \in X \times Y) \quad 	f(x,y) =\innp{Kx, y} + g(x)  -h(y).
\end{align}

In \cref{lemma:algorithmsymplity} below, we illustrate that
when we apply our iterate scheme \cref{eq:PPalgorithm:1} 
specifically to a convex-concave function in the form of \cref{eq:fspecialdefine}, 
our iterate scheme is actually a generalization of the iterate schemes studied in 
\cite{PockCremersBischofChambolle2009algorithm}, \cite{EsserZhangChan2010},
\cite{HeYuan2010convergence}, and \cite{ChambollePock2011}.
The references mentioned above worked on involved parameters as constants,
while our iterate scheme accommodates involved parameters as general sequences.

 In \cite[Theorem~1]{ChambollePock2011}, 
the authors considered the convergence of 
iterates generated by the iterate scheme \cref{eq:lemma:algorithmsymplity} 
associated with 
$(\forall k \in \mathbf{N})$ $\sigma_{k} \equiv \sigma $, $\tau_{k} \equiv \tau $, 
$\alpha_{k} \equiv 1$,  $\beta_{k} \equiv 0$, $X=\mathcal{H}_{1}$, and 
$Y=\mathcal{H}_{2}$.  
Multiple results in this section mimic the idea used in  the proof of
\cite[Theorem~1]{ChambollePock2011}.
In particular, \cref{theorem:pointsconvergence} is almost a generalization of 
\cite[Theorem~1]{ChambollePock2011} but our assumption $\norm{K} \limsup_{i \to 
	\infty} 
\tau_{i} < \frac{1}{2}$ and $\norm{K} \limsup_{i \to \infty} \sigma_{i} < \frac{1}{2}$ is 
stronger than that $\tau \sigma \norm{K}^{2} <1$ 
used in \cite[Theorem~1]{ChambollePock2011}. 
Although the authors in \cite{ChambollePock2011} 
worked only on the convergence of iterates
and didn't explicitly consider any   convergence of function values related to iterates, 
by using some techniques applied in \cite[Section~2]{Oy2023subgradient}, 
we shall provide the convergence 
$f(\hat{x}_{k}, \hat{x}_{k}) \to f(x^{*},y^{*}) $ in \cref{theorem:convergencef}, 
where $((x^{k},y^{k}))_{k \in \mathbf{N}}$ is generated by the 
iterate scheme \cref{eq:lemma:algorithmsymplity} below, 
$(\forall k \in \mathbf{N})$ 
$\hat{x}_{k} := \frac{1}{k+1} \sum^{k}_{i=0} x^{i+1} $ and $\hat{y}_{k} := 
\frac{1}{k+1} \sum^{k}_{i=0} y^{i+1} $, and $(x^{*},y^{*}) \in X \times Y$ is a 
saddle-point of $f$.

\begin{lemma} \label{lemma:algorithmsymplity}
Recall that $(\tau_{k})_{k \in \mathbf{N}}$ and $(\sigma_{k})_{k \in \mathbf{N}}$ 
are in $\mathbf{R}_{++}$ and that $(\alpha_{k})_{k \in \mathbf{N}\cup\{-1\}}$ and 
$(\beta_{k})_{k \in \mathbf{N}}$ are  in $\mathbf{R}_{+}$. 
Let $(x^{0}, y^{0}) \in X \times Y$. 
Set $\bar{x}^{0}=x^{0}$ and $\bar{y}^{0} =y^{0}$ 
$($i.e., $x^{-1} =x^{0}$ and $y^{-1} =y^{0}$$)$.  
The iterate scheme presented in  \cref{eq:PPalgorithm:1} associated with $f$ 
defined in \cref{eq:fspecialdefine} is equivalent to the iterate scheme below: 
for every $k \in \mathbf{N}$, 
	\begin{subequations}\label{eq:lemma:algorithmsymplity}
		\begin{align}
			&y^{k+1} =  \Prox^{Y}_{ \sigma_{k} h} (\sigma_{k} K\bar{x}^{k}  +  y^{k} );
 \label{lemma:algorithmsymplity:yk}\\
			& \bar{y}^{k+1} = y^{k+1} + \beta_{k} (y^{k+1} -y^{k}); 
			\label{eq:lemma:algorithmsymplity:bary}\\
			&x^{k+1} = \Prox^{X}_{\tau_{k}  g} ( -\tau_{k} K^{*} \bar{y}^{k+1}+ x^{k} ); 
\label{lemma:algorithmsymplity:xk}\\
			& \bar{x}^{k+1} = x^{k+1} + \alpha_{k} (x^{k+1} -x^{k}). 
			\label{eq:lemma:algorithmsymplity:barx} 
		\end{align}
	\end{subequations}
Moreover, we have that 
\begin{align*}
	&(\forall x \in X) \quad  
	\innp{\frac{x^{k} -x^{k+1}}{\tau_{k}} -K^{*}\bar{y}^{k+1}, x -x^{k+1}} 
 + g(x^{k+1}) \leq g(x);\\
	&(\forall y \in Y) \quad 
	\innp{\frac{y^{k}-y^{k+1}}{\sigma_{k}} +K\bar{x}^{k}, y - y^{k+1}} 
	  +h(y^{k+1}) \leq h(y).
\end{align*} 
\end{lemma}

\begin{proof}
	Applying \cref{fact:proxsubdiff}, \cref{eq:PPalgorithm:2:y}, and the construction of $f$ 
defined in  \cref{eq:fspecialdefine}, we know that  
	\begin{align*}
			&(\forall y \in Y) \quad \innp{\frac{y^{k}-y^{k+1}}{\sigma_{k}}, y-y^{k+1}}	
		\leq - f (\bar{x}^{k}, y)+ f (\bar{x}^{k}, y^{k+1})  \\
			\Leftrightarrow &(\forall y \in Y) \quad \innp{\frac{y^{k}-y^{k+1}}{\sigma_{k}}, 
			y-y^{k+1}}	
			\leq  \innp{K\bar{x}^{k} , y^{k+1}-y} 
			-h(y^{k+1}) +h(y)\\
		\Leftrightarrow &(\forall y \in Y) \quad 
		\innp{\frac{y^{k}-y^{k+1}}{\sigma_{k}}+K\bar{x}^{k}, y-y^{k+1}}	
		\leq   -h(y^{k+1}) +h(y)\\
		\Leftrightarrow &(\forall y \in Y) \quad 
		\innp{ y^{k} +\sigma_{k} K\bar{x}^{k} -y^{k+1}, y-y^{k+1}}	
		\leq   -\sigma_{k}h(y^{k+1}) +\sigma_{k}h(y)\\
		    \Leftrightarrow~ & y^{k+1} =  \argmin_{y  \in Y}  \{ \sigma_{k}h(y) + \frac{1}{2} 
		    \norm{y -\left(y^{k} +\sigma_{k} K\bar{x}^{k}\right)}^{2} \}
		=\Prox^{Y}_{ \sigma_{k} h} (\sigma_{k} K\bar{x}^{k}  +  y^{k} ).
	\end{align*}
	Employing  \cref{fact:proxsubdiff}, \cref{eq:PPalgorithm:2:x}, and the definition of $f$ 
 in  \cref{eq:fspecialdefine}, we observe that 
	\begin{align*}
&(\forall x \in X) \quad \innp{ \frac{x^{k}-x^{k+1}}{\tau_{k}}, x-x^{k+1}} 
		 \leq f(x, \bar{y}^{k+1}) -f(x^{k+1}, \bar{y}^{k+1}) \\ 
\Leftrightarrow &(\forall x \in X) \quad \innp{ \frac{x^{k}-x^{k+1}}{\tau_{k}}, x-x^{k+1}} 
 \leq  \innp{x - x^{k+1}, K^{*} \bar{y}^{k+1} } +g(x) -g(x^{k+1})\\
\Leftrightarrow &(\forall x \in X) \quad \innp{ \frac{x^{k}-x^{k+1}}{\tau_{k}} -K^{*} 
\bar{y}^{k+1}, x-x^{k+1}} 
\leq   g(x) -g(x^{k+1})\\
\Leftrightarrow &(\forall x \in X) \quad \innp{  x^{k}  -\tau_{k} K^{*} \bar{y}^{k+1} -x^{k+1}, 
x-x^{k+1}} 
\leq   \tau_{k} g(x) - \tau_{k}g(x^{k+1})\\
	   \Leftrightarrow  & x^{k+1} =
		   \argmin_{x \in X}  \{ \tau_{k} g(x) + \frac{1}{2} \norm{x - \left( x^{k}  -\tau_{k} K^{*} 
		   \bar{y}^{k+1} \right)}^{2} \}
		   = \Prox^{X}_{\tau_{k}  g} ( -\tau_{k} K^{*} \bar{y}^{k+1}+ x^{k} ). 
	\end{align*}	
\end{proof}

Below, we rewrite the iterate scheme \cref{eq:lemma:algorithmsymplity} for some 
examples so 
that the resulted expressions have neither proximity mapping nor subgradient operator.
Of course, when we try to rewrite and simplify the iterate scheme
\cref{eq:lemma:algorithmsymplity} for particular examples, not every example has such 
explicit expression without proximity mapping or subgradient operator. 
\begin{example} \label{example:proximity} 
	\begin{enumerate}
		\item  \label{example:proximity:x2-y2} 
		Consider the convex-concave function $(\forall (x,y) \in \mathbf{R}^{2})$ $f(x,y) = 
		x^{2}-y^{2}$. In this case, $X=Y=\mathbf{R}$, $K=0$, $(\forall x \in \mathbf{R})$ 
		$g(x) =x^{2}$, and 
		$(\forall y \in \mathbf{R})$ $h(y)=y^{2}$.
	Bearing \cref{lemma:algorithmsymplity} and \cref{eq:PPalgorithm:2}  in mind, 
	we deduce that in this case,  
	the iterate scheme \cref{eq:lemma:algorithmsymplity} is nothing 
	but $(x^{0}, y^{0}) \in \mathbf{R}^{2}$ and for every $k \in \mathbf{N}$, 
\[ 
y^{k+1} = \frac{1}{1+2\sigma_{k}} y^{k} \quad \text{and} \quad x^{k+1} 
= \frac{1}{1+2\tau_{k}} x^{k}.
\]
		
\item \label{example:proximity:lp}  Consider the Lagrangian $f: \mathbf{R}^{n} \times 
\mathbf{R}^{m}_{+} \to \mathbf{R}$ of the inequality form LP, which is defined as
\[ 
(\forall (x,y) \in \mathbf{R}^{n} \times \mathbf{R}^{m}_{+}) \quad f(x,y)=	y^{T}Ax 
+c^{T}x   -b^{T}y,
\]
where $A \in \mathbf{R}^{m \times n}$, $b  \in \mathbf{R}^{m}$, and $c \in 
\mathbf{R}^{n}$. 
(See, e.g., \cite[Section~5.2.1]{BV2004} or \cite[Example~2.2(ii)]{Oy2023subgradient} 
for details.)
In this case, $X=\mathbf{R}^{n}$, $Y=\mathbf{R}^{m}_+$,  $K=A$, $(\forall x \in 
\mathbf{R}^{n})$ $g(x)=c^{T}x$, and 
$(\forall y \in \mathbf{R}^{m}_{+})$ $h(y) = b^{T}y$. 
Note that both $g$ and $h$ are differentiable and that their derivatives are constants 
that are independent of the variables. 
Taking \cref{lemma:algorithmsymplity} and \cref{eq:PPalgorithm:2} into account, 
we obtain that 
the iterate scheme \cref{eq:lemma:algorithmsymplity} becomes  that given 
$(x^{0}, y^{0}) \in \mathbf{R}^{n} \times \mathbf{R}^{m}_{+}$,   for every $k \in 
\mathbf{N}$, 
\begin{align*}
y^{k+1}  &= y^{k} +\sigma_{k} A\bar{x}^{k} - \sigma_{k}b;\\
 \bar{y}^{k+1} &=  y^{k+1} + \beta_{k} (y^{k+1} -y^{k});\\
x^{k+1}  &=  x^{k} -\tau_{k}A^{T}\bar{y}^{k+1} -c\tau_{k};\\
\bar{x}^{k+1}  &=  x^{k+1} + \alpha_{k} (x^{k+1} -x^{k}).
\end{align*}
\end{enumerate}
\end{example}

In the rest of this section, $((x^{k},y^{k}))_{k \in \mathbf{N}}$ is generated by the
iterate scheme \cref{eq:lemma:algorithmsymplity}; moreover
\begin{align}  \label{eq:hatx_ky_k}
(\forall k \in \mathbf{N}) 
~ \hat{x}_{k} := \frac{1}{k+1} \sum^{k}_{i=0} x^{i+1} 
=  \frac{1}{k+1} \sum^{k+1}_{i=1} x^{i}   
\text{ and }   \hat{y}_{k} := \frac{1}{k+1} \sum^{k}_{i=0} y^{i+1} 
=   \frac{1}{k+1} \sum^{k+1}_{i=1} y^{i}.
\end{align}

\begin{lemma}\label{lemma:clustersaddle-point}
Suppose that $\sup_{k \in \mathbf{N}} \alpha_{k} < +\infty$, 
$\sup_{k \in \mathbf{N}} \beta_{k} < +\infty$,  $0<\inf_{k \in \mathbf{N}} \sigma_{k}  $, 
and $0<\inf_{k \in \mathbf{N}} \tau_{k}  $. 
Suppose that 
$\lim_{k \to \infty} \norm{x^{k} -x^{k+1}} =0$ and $\lim_{k \to \infty} \norm{y^{k} -y^{k+1}} =0$. 
Then all cluster points of $((x^{k}, y^{k}))_{k \in \mathbf{N}}$ are saddle-points of $f$.
\end{lemma}

\begin{proof}
Let $(x^{*}, y^{*})$ be a cluster point of $((x^{k}, y^{k}))_{k \in \mathbf{N}}$. 
Because $((x^{k}, y^{k}))_{k \in \mathbf{N}}$ is in $X\times Y$ and 
$X\times Y$  is closed, we know that $(x^{*}, y^{*}) \in X\times Y$. 
Then there is a subsequence  $((x^{k_{i}}, y^{k_{i}}))_{i \in \mathbf{N}}$ of 
$((x^{k}, y^{k}))_{k \in \mathbf{N}}$ satisfying 
\[ 
x^{k_{i}} \to x^{*} \quad \text{and} \quad y^{k_{i}} \to y^{*}.
\]
Combine this with the assumptions $\sup_{k \in \mathbf{N}} \alpha_{k} < +\infty$, 
$\sup_{k \in \mathbf{N}} \beta_{k} < +\infty$, $0<\inf_{k \in \mathbf{N}} \sigma_{k}  $, 
$0<\inf_{k \in \mathbf{N}} \tau_{k}  $, $\lim_{k \to \infty} \norm{x^{k} -x^{k+1}} =0$, 
 $\lim_{k \to \infty} \norm{y^{k} -y^{k+1}} =0$, and the definitions 
\cref{eq:lemma:algorithmsymplity:barx} and \cref{eq:lemma:algorithmsymplity:bary} 
of $(\bar{x}^{k})_{k \in \mathbf{N}}$ and $(\bar{y}^{k})_{k \in \mathbf{N}}$ to 
deduce that for every $i \in \mathbf{N}$,
\begin{subequations}  \label{eq:lemma:clustersaddle-point:ki}
	\begin{align}
		&\left(  y^{k_{i}+1}, K\bar{x}^{k_{i}} 
		+ \frac{y^{k_{i}} -y^{k_{i}+1}}{\sigma_{k_{i}}}  \right) \to (y^{*}, Kx^{*}); \\
		&\left( x^{k_{i}+1}, - K^{*} \bar{y}^{k_{i}+1} 
		+ \frac{x^{k_{i}} -x^{k_{i}+1}}{\tau_{k_{i}}}  \right) \to (x^{*}, -K^{*}y^{*}).
	\end{align}
\end{subequations}

On the other hand, as a consequence of \cref{lemma:algorithmsymplity:yk} 
and \cref{lemma:algorithmsymplity:xk}, we know that for every $i \in \mathbf{N}$,
\begin{subequations}\label{eq:lemma:clustersaddle-point:graph}
	\begin{align} 
		&\left(  y^{k_{i}+1}, K\bar{x}^{k_{i}} 
		+ \frac{y^{k_{i}} -y^{k_{i}+1}}{\sigma_{k_{i}}}  \right)   \in \gra \partial \left( 
		h+\iota_{Y}\right);\\
		&\left( x^{k_{i}+1}, - K^{*} \bar{y}^{k_{i}+1} 
		+ \frac{x^{k_{i}} -x^{k_{i}+1}}{\tau_{k_{i}}}  \right)  \in \gra \partial  \left( 
		g+\iota_{X}\right).
	\end{align}
\end{subequations}
According to \cite[Theorem~20.25 and Proposition~20.38]{BC2017}, 
the graphs of the maximally monotone operators $\partial \left( h+\iota_{Y}\right)$ and 
$\partial  \left( g+\iota_{X}\right)$ 
are both sequentially closed, which, connecting with \cref{eq:lemma:clustersaddle-point:ki} and \cref{eq:lemma:clustersaddle-point:graph},  ensures that 
\[ 
(y^{*}, Kx^{*}) \in \gra \partial \left( h+\iota_{Y}\right)  \quad \text{and} \quad  (x^{*}, 
-K^{*}y^{*})   \in \gra \partial  \left( g+\iota_{X}\right).
\]
Recall that $(x^{*}, y^{*}) \in X\times Y$.  Note that 
$
	(y^{*}, Kx^{*}) \in \gra  \partial \left( h+\iota_{Y}\right) \Leftrightarrow Kx^{*} \in \partial 
	\left( h+\iota_{Y}\right) 
	(y^{*}) 
	\Leftrightarrow 0 \in  -Kx^{*} + \partial \left( h+\iota_{Y}\right) (y^{*})  =  
	\partial_{y}\left(-f +\iota_{Y} \right) (x^{*},y^{*})
$
and that 
$
	(x^{*}, -K^{*}y^{*})   \in \gra \partial  \left( g+\iota_{X}\right) \Leftrightarrow -K^{*}y^{*} 
	\in \partial  \left( g+\iota_{X}\right) (x^{*}) 
	\Leftrightarrow 0 \in  K^{*}y^{*} + \partial  \left( g+\iota_{X}\right) (x^{*}) = \partial_{x}    
	\left( f+\iota_{X}\right) (x^{*},y^{*}).
$
Combine results above with \cref{fact:saddlepoint} (\cref{fact:saddlepoint:sp} $\Leftrightarrow$ \cref{fact:saddlepoint:partial} in \cref{fact:saddlepoint})  
to conclude that $(x^{*},y^{*})$ is a saddle-point of $f$.
\end{proof}

\begin{lemma} \label{lemma:CSInequalities}
Let $(x,y) \in X \times Y$ and let $k \in \mathbf{N}$.  
The following assertions hold. 
	\begin{enumerate}
		\item    $\abs{\innp{K(x^{k}-x^{k-1}), y^{k}-y^{k+1}}} 
		\leq \frac{\norm{K}}{2}\norm{x^{k}-x^{k-1}}^{2} 
		+ \frac{\norm{K}}{2}\norm{y^{k}-y^{k+1}}^{2}$. 
		\item \label{lemma:CSInequalities:KK*}  
		$\abs{\innp{K(x^{k+1}-x^{k}), y^{k+1}-y}} 
		\leq \frac{1}{2} \frac{\norm{x^{k+1}-x^{k}}^{2}}{\tau_{k}} 
		+ \frac{\norm{K}^{2}\tau_{k}}{2}\norm{y^{k+1}-y}^{2}$.
		\item   $\abs{(1-\alpha_{k-1}) \innp{K(x^{k} -x^{k-1}), 
		y^{k+1}-y} } 
		\leq \frac{\norm{K}}{2} \norm{x^{k} -x^{k-1}}^{2} 
		+ \frac{\norm{K}}{2} \norm{(1-\alpha_{k-1}) (y^{k+1}-y)}^{2}$.
		\item    $\abs{\beta_{k}\innp{K(x^{k+1} -x), 
		y^{k+1}-y^{k}}} 
		\leq \frac{\norm{K}}{2} \norm{ \beta_{k} (x^{k+1} -x)}^{2} 
		+   \frac{\norm{K}}{2} \norm{ y^{k+1}-y^{k}}^{2}$.
	\end{enumerate}
\end{lemma}

\begin{proof}
Proofs of inequalities above use  the Cauchy-Schwarz inequality, 
\cref{eq:definenormK},  
and the fact that 
\[ 
(\forall (a,b) \in \mathbf{R}_{+}) (\forall \kappa \in \mathbf{R}_{++}) \quad 
2ab = 2 (\sqrt{\kappa}a ) \left(\frac{b}{\sqrt{\kappa}}\right) 
\leq \kappa a^{2} + \frac{b^{2}}{\kappa}.
\]
In addition, note that to prove \cref{lemma:CSInequalities:KK*}  above,
we employ the definition of the adjoint  $K^{*}$ of the continuous linear operator $K$
and also the fact $\norm{K} =\norm{K^{*}}$ from \cite[Fact~2.25(ii)]{BC2017}.
\end{proof}

\begin{lemma} \label{lemma:MultiplefIneq}
Let $(x,y) \in X \times Y$. The following statements hold. 
\begin{enumerate}
\item \label{lemma:MultiplefIneq:fK} We have that for every $k \in \mathbf{N}$,
\begin{align*}
&f(x^{k+1},y) -f(\bar{x}^{k},y) + f(\bar{x}^{k}, y^{k+1})   
-f(x,y^{k+1}) - f(x^{k+1}, \bar{y}^{k+1}) + f(x,\bar{y}^{k+1})\\
=&-\left( \innp{K(x^{k+1}-\bar{x}^{k}), y^{k+1}-y} 
-\innp{K(x^{k+1}-x), y^{k+1}-\bar{y}^{k+1}} \right).
\end{align*}
	\item \label{lemma:MultiplefIneq:sum} 
	Let $p \in \mathbf{N}$ and $k \in \mathbf{N}$ with $p \leq k$.
Then
	\begin{align*}
	& \sum^{k}_{i=p} \left(	f(x^{i+1},y) -f(\bar{x}^{i},y) + f(\bar{x}^{i}, y^{i+1})   
	-f(x,y^{i+1}) - f(x^{i+1}, \bar{y}^{i+1}) + f(x,\bar{y}^{i+1}) \right) \\
	=& -\innp{K(x^{k+1}-x^{k}), y^{k+1}-y} +\innp{K(x^{p}-x^{p-1}), y^{p}-y}
	- \sum^{k}_{i=p} \innp{K(x^{i} -x^{i-1}),y^{i}- y^{i+1}} \\
	&- \sum^{k}_{i=p} (1- \alpha_{i-1} )  \innp{K(x^{i} -x^{i-1}),y^{i+1}- y}  
	 -  \sum^{k}_{i=p}\beta_{i} \innp{K(x^{i+1}-x), y^{i+1} -y^{i}}.
	\end{align*}
\item \label{lemma:MultiplefIneq:sumIneq} Let $p \in \mathbf{N}$ and 
$k \in \mathbf{N}$ with $p \leq k$. Then
\begin{align*}
	& \sum^{k}_{i=p} \left(	f(x^{i+1},y) -f(\bar{x}^{i},y) + f(\bar{x}^{i}, y^{i+1})  
	 -f(x,y^{i+1}) - f(x^{i+1}, \bar{y}^{i+1}) + f(x,\bar{y}^{i+1}) \right) \\
	\leq &  \frac{1}{2} \frac{\norm{x^{k+1}-x^{k}}^{2}}{\tau_{k}} 
	+ \frac{\norm{K}^{2}\tau_{k}}{2}\norm{y^{k+1}-y}^{2}  
	+ \norm{K} \sum^{k}_{i=p} \norm{x^{i}-x^{i-1}}^{2} 
	+  \norm{K}  \sum^{k}_{i=p} \norm{y^{i}-y^{i+1}}^{2}\\
	&   +  \frac{\norm{K}}{2} \sum^{k}_{i=p} \norm{(1-\alpha_{i-1}) (y^{i+1}-y)}^{2}  
	+    \frac{\norm{K}}{2} \sum^{k}_{i=p} \norm{ \beta_{i} (x^{i+1} -x)}^{2} 
	 +\abs{\innp{K(x^{p}-x^{p-1}), y^{p}-y}}.
\end{align*}
Consequently, 
\begin{align*}
	 &\sum^{k}_{i=0} \left(f(x^{i+1},y) -f(\bar{x}^{i},y) + f(\bar{x}^{i}, y^{i+1})   
	 -f(x,y^{i+1}) - f(x^{i+1}, \bar{y}^{i+1}) + f(x,\bar{y}^{i+1}) \right) \\
	\leq &  \frac{1}{2} \frac{\norm{x^{k+1}-x^{k}}^{2}}{\tau_{k}} 
	+ \frac{\norm{K}^{2}\tau_{k}}{2}\norm{y^{k+1}-y}^{2}  
	+ \norm{K} \sum^{k}_{i=0} \norm{x^{i}-x^{i-1}}^{2} 
	+  \norm{K}  \sum^{k}_{i=0} \norm{y^{i}-y^{i+1}}^{2}\\
	&  +  \frac{\norm{K}}{2} \sum^{k}_{i=0} \norm{(1-\alpha_{i-1}) (y^{i+1}-y)}^{2}  
	+    \frac{\norm{K}}{2} \sum^{k}_{i=0} \norm{ \beta_{i} (x^{i+1} -x)}^{2}. 
\end{align*} 
	\end{enumerate}
\end{lemma}

\begin{proof}
\cref{lemma:MultiplefIneq:fK}: By some rearrangement, we get that 
\begin{align*}
	&f(x^{k+1},y) -f(\bar{x}^{k},y) + f(\bar{x}^{k}, y^{k+1})   
	-f(x,y^{k+1}) - f(x^{k+1}, \bar{y}^{k+1}) + f(x,\bar{y}^{k+1}) \\
	=& -\left( f(x^{k+1}, \bar{y}^{k+1}) -f(x^{k+1},y) 
	+ f(\bar{x}^{k},y) -f(\bar{x}^{k}, y^{k+1}) +f(x,y^{k+1})  -f(x,\bar{y}^{k+1}) \right).
\end{align*}  
Then apply \cref{lemma:fKuv} with $F=f$, $u_{1}=x^{k+1}$, 
$u_{2}=\bar{x}^{k}$, $u_{3} = x$, $v_{1} =\bar{y}^{k+1}$, $v_{2} =y$, 
and $v_{3} =y^{k+1}$ to derive the required equation in \cref{lemma:MultiplefIneq:fK}.

\cref{lemma:MultiplefIneq:sum}: 
In view of the result obtained in \cref{lemma:MultiplefIneq:fK}, 
\begin{subequations}\label{eq:lemma:MultiplefIneq:sum:ftoInnpK}
	\begin{align}
		&\sum^{k}_{i=p} \left(	f(x^{i+1},y) -f(\bar{x}^{i},y) 
		+ f(\bar{x}^{i}, y^{i+1})   -f(x,y^{i+1}) 
		- f(x^{i+1}, \bar{y}^{i+1}) + f(x,\bar{y}^{i+1}) \right) \\
		=&- \sum^{k}_{i=p} \left( \innp{K(x^{i+1}-\bar{x}^{i}), y^{i+1}-y} 
		-\innp{K(x^{i+1}-x), y^{i+1}-\bar{y}^{i+1}} \right).
	\end{align}
\end{subequations}
Let $i \in \{p,p+1, \ldots, k\}$. Recall from \cref{eq:lemma:algorithmsymplity:barx} 
and \cref{eq:lemma:algorithmsymplity:bary}   that 
$\bar{x}^{i} = x^{i} + \alpha_{i-1} (x^{i} -x^{i-1})$ 
and $\bar{y}^{i+1} = y^{i+1} + \beta_{i} (y^{i+1} -y^{i})$. Then 
\begin{align*}
	&\innp{K(x^{i+1}-\bar{x}^{i}), y^{i+1}-y} 
	-\innp{K(x^{i+1}-x), y^{i+1}-\bar{y}^{i+1}} \\
	=&  \innp{K(x^{i+1}-x^{i}), y^{i+1}-y}- \alpha_{i-1} \innp{K(x^{i} -x^{i-1}),y^{i+1}- y} 
	+  \beta_{i} \innp{K(x^{i+1}-x), y^{i+1} -y^{i}} \\
	=&   \innp{K(x^{i+1}-x^{i}), y^{i+1}-y} 
	-  \innp{K(x^{i} -x^{i-1}),y^{i}- y} +  \innp{K(x^{i} -x^{i-1}),y^{i}- y^{i+1}+ y^{i+1} - y} \\
	&- \alpha_{i-1} \innp{K(x^{i} -x^{i-1}),y^{i+1}- y} 
	+  \beta_{i} \innp{K(x^{i+1}-x), y^{i+1} -y^{i}} \\
	=&  \innp{K(x^{i+1}-x^{i}), y^{i+1}-y} -  \innp{K(x^{i} -x^{i-1}),y^{i}- y} 
	+  \innp{K(x^{i} -x^{i-1}),y^{i}- y^{i+1}} \\
	&+ (1- \alpha_{i-1} )  \innp{K(x^{i} -x^{i-1}),y^{i+1}- y} 
	+ \beta_{i} \innp{K(x^{i+1}-x), y^{i+1} -y^{i}}.
\end{align*}
Sum the expression above over $i$ from $p$ to $k$  to establish that 
\begin{align*}
	&\sum^{k}_{i=p} \left( \innp{K(x^{i+1}-\bar{x}^{i}), y^{i+1}-y} 
	-\innp{K(x^{i+1}-x), y^{i+1}-\bar{y}^{i+1}} \right)\\
	=  &\innp{K(x^{k+1}-x^{k}), y^{k+1}-y} -\innp{K(x^{p}-x^{p-1}), y^{p}-y} 
	+ \sum^{k}_{i=p} \innp{K(x^{i} -x^{i-1}),y^{i}- y^{i+1}}\\
	&+ \sum^{k}_{i=p} (1- \alpha_{i-1} )  \innp{K(x^{i} -x^{i-1}),y^{i+1}- y}  
	 +  \sum^{k}_{i=p}\beta_{i} \innp{K(x^{i+1}-x), y^{i+1} -y^{i}}.
\end{align*}
Altogether, we obtain the required equation in \cref{lemma:MultiplefIneq:sum}  easily. 

\cref{lemma:MultiplefIneq:sumIneq}: 
Combine the result obtained in \cref{lemma:MultiplefIneq:sum} 
above with the four inequalities in \cref{lemma:CSInequalities} 
to establish the first inequality in \cref{lemma:MultiplefIneq:sumIneq} which, 
connected with $x^{-1}=x^{0}$, deduces the last assertion. 
\end{proof}

\begin{corollary} \label{corollary:sumnormsumf}
Let $(x,y) \in X \times Y$. Let $p \in \mathbf{N}$ and $k \in \mathbf{N}$ with $p \leq k$. 
Then 
\begin{align*}
 &  \sum^{k}_{i=p} \frac{1}{2 \tau_{i}} \left(  \norm{x -x^{i}}^{2} -\norm{x -x^{i+1}}^{2} 
- \norm{x^{i} -x^{i+1}}^{2} \right)\\
& +     \sum^{k}_{i=p} \frac{1}{2 \sigma_{i}}  \left( \norm{y - y^{i}}^{2} - 
 \norm{y-y^{i+1}}^{2} - \norm{y^{i}-y^{i+1}}^{2} \right)\\
& +   \sum^{k}_{i=p} \left(f(x^{i+1},y) -f(\bar{x}^{i},y) + f(\bar{x}^{i}, y^{i+1})  
 -f(x,y^{i+1}) - f(x^{i+1}, \bar{y}^{i+1}) + f(x,\bar{y}^{i+1}) \right)\\
\leq & \frac{\norm{K}^{2}\tau_{k}}{2}\norm{y^{k+1}-y}^{2}+ \sum^{k}_{i=p} \frac{1}{2 
\tau_{i}} \norm{x -x^{i}}^{2} -\sum^{k}_{i=p}  \left(  \frac{1}{2 \tau_{i}}  - 
\frac{\norm{K}\beta^{2}_{i} }{2} \right) \norm{x -x^{i+1}}^{2} \\
&+ \sum^{k-1}_{i=p} \left( \norm{K} - \frac{1}{2 \tau_{i}}   \right)  
\norm{x^{i}-x^{i+1}}^{2} +   \sum^{k}_{i=p} \frac{1}{2 \sigma_{i}} \norm{y - y^{i}}^{2} \\
& -  \sum^{k}_{i=p}\left(  \frac{1}{2 \sigma_{i}} 
-  \frac{\norm{K}}{2} (1-\alpha_{i-1})^{2} \right) \norm{y-y^{i+1}}^{2} 
+  \sum^{k}_{i=p} \left(  \norm{K} - \frac{1}{2 \sigma_{i}}   \right)  \norm{y^{i}-y^{i+1}}^{2}\\
&+ \norm{K}\norm{x^{p-1} -x^{p}}^{2}  +\abs{\innp{K(x^{p}-x^{p-1}), y^{p}-y}}.
\end{align*}
\end{corollary}

\begin{proof}
 Apply   some easy algebra to observe that 
\begin{align*}
&- \sum^{k}_{i=p} \frac{1}{2 \tau_{i}} \norm{x^{i} -x^{i+1}}^{2}  
+\frac{1}{2} \frac{\norm{x^{k+1}-x^{k}}^{2}}{\tau_{k}} 
+ \norm{K} \sum^{k}_{i=p} \norm{x^{i}-x^{i-1}}^{2} \\
 =&  \sum^{k-1}_{i=p} \left( \norm{K} - \frac{1}{2 \tau_{i}}   \right)  \norm{x^{i}-x^{i+1}}^{2} 
 + \norm{K}\norm{x^{p-1} -x^{p}}^{2}. 
	\end{align*}
Hence, it is clear that the desired inequality is an easy application of \cref{lemma:MultiplefIneq}\cref{lemma:MultiplefIneq:sumIneq}
\end{proof}	
	
\begin{corollary}\label{corollary:betaialphai-1Ineq}
Let $(x,y) \in X \times Y$.	Suppose that    $(\forall i \in \mathbf{N})$ 
$\norm{K} \beta^{2}_{i} = 	  \left( \frac{1}{ \tau_{i}} -\frac{1}{ \tau_{i+1}} \right)$
 and   $\norm{K} (1-\alpha_{i-1})^{2} = \left( \frac{1}{  \sigma_{i}} - \frac{1}{  \sigma_{i+1}} 
 \right)$. 
 Then for every $p \in \mathbf{N}$ and $k \in \mathbf{N}$ with $p \leq k$,
	\begin{align*}
		&  \sum^{k}_{i=p} \frac{1}{2 \tau_{i}} \left(  \norm{x -x^{i}}^{2} -\norm{x -x^{i+1}}^{2} 
		- 
		\norm{x^{i} -x^{i+1}}^{2} \right)\\
		& +     \sum^{k}_{i=p} \frac{1}{2 \sigma_{i}}  \left( \norm{y - y^{i}}^{2} - 
		\norm{y-y^{i+1}}^{2} - \norm{y^{i}-y^{i+1}}^{2} \right)\\
		& +   \sum^{k}_{i=p} \left(f(x^{i+1},y) -f(\bar{x}^{i},y) + f(\bar{x}^{i}, y^{i+1})  
		 -f(x,y^{i+1}) - f(x^{i+1}, \bar{y}^{i+1}) + f(x,\bar{y}^{i+1}) \right)\\
		\leq & \frac{1}{2 \tau_{p}} \norm{x -x^{p}}^{2} 
		-  \frac{1}{2 \tau_{k+1}} \norm{x -x^{k+1}}^{2} 
		+  \frac{1}{2 \sigma_{p}} \norm{y -y^{p}}^{2} - \left( \frac{1}{2 \sigma_{k+1}}  
		- \frac{\norm{K}^{2}\tau_{k}}{2}  \right)  \norm{y -y^{k+1}}^{2}\\
		& + \sum^{k-1}_{i=p} \left( \norm{K} - \frac{1}{2 \tau_{i}}   \right)  \norm{x^{i}-x^{i+1}}^{2} 
		+  \sum^{k}_{i=p} \left(  \norm{K} - \frac{1}{2 \sigma_{i}}   \right)  \norm{y^{i}-y^{i+1}}^{2}\\
		&+  \norm{K}\norm{x^{p-1} -x^{p}}^{2}  
		+\abs{\innp{K(x^{p}-x^{p-1}), y^{p}-y}}.
	\end{align*}
\end{corollary}

\begin{proof}
Based on our assumptions, it is easy to see that 
\[ 
(\forall i \in \mathbf{N}) \quad	\frac{1}{ \tau_{i+1}}   
=	\frac{1}{ \tau_{i}}  -  \norm{K}\beta^{2}_{i}  
\quad \text{and} \quad   \frac{1}{  \sigma_{i+1}} 
=  \frac{1}{  \sigma_{i}} -  \norm{K}  (1-\alpha_{i-1})^{2}.
\]
This connected  with \cref{corollary:sumnormsumf} deduces the required inequality. 
\end{proof}

\begin{corollary} \label{corollary:Ineqxy}
	Let $(x,y) \in X \times Y$.	Suppose that    
	$(\forall i \in \mathbf{N})$
	 $\norm{K} \beta^{2}_{i} = 	  \left( \frac{1}{ \tau_{i}} -\frac{1}{ \tau_{i+1}} \right)$ 
	 and    $\norm{K} (1-\alpha_{i-1})^{2} =   \left( \frac{1}{  \sigma_{i}} - \frac{1}{  
	 \sigma_{i+1}} \right)$. 
	   Let $p \in \mathbf{N}$ and and let $k \in \mathbf{N}$ with $p \leq k$.
	   Then
	\begin{align*}
		&\frac{1}{2 \tau_{p}} \norm{x -x^{p}}^{2} 
		+\frac{1}{2 \sigma_{p}} \norm{y -y^{p}}^{2}  \\
		\geq & \sum^{k}_{i=p} \left( f(x^{i+1},y)-f(x,y^{i+1}) \right)
		+ \frac{1}{2 \tau_{k+1}} \norm{x -x^{k+1}}^{2} 
		+ \left( \frac{1}{2 \sigma_{k+1}}  
		- \frac{\norm{K}^{2}\tau_{k}}{2}  \right)  \norm{y -y^{k+1}}^{2} \\
		&-\sum^{k-1}_{i=p} \left( \norm{K} - \frac{1}{2 \tau_{i}}   \right) 
		 \norm{x^{i}-x^{i+1}}^{2} -  \sum^{k}_{i=p} \left(  \norm{K} - \frac{1}{2 \sigma_{i}}   \right)  
		 \norm{y^{i}-y^{i+1}}^{2}\\
		&- \norm{K}\norm{x^{p-1} -x^{p}}^{2}  -\abs{\innp{K(x^{p}-x^{p-1}), y^{p}-y}}.
	\end{align*}
\end{corollary}

\begin{proof}
Applying  \cref{lemma:xkykinequalities}\cref{lemma:xkykinequalities:normy}  
and \cref{lemma:xkykinequalities}\cref{lemma:xkykinequalities:normx} in the   inequality below, 
we have that   for every $i \in \mathbf{N}$, 
\begin{align*}
0 \geq &	\frac{1}{2 \sigma_{i}}\left( \norm{y^{i} -y^{i+1}}^{2} 
+\norm{y-y^{i+1}}^{2} -\norm{y - y^{i}}^{2} \right) 
+\frac{1}{2 \tau_{i}} \left( \norm{x^{i} -x^{i+1}}^{2} 
+ \norm{x - x^{i+1}}^{2} -\norm{x-x^{i}}^{2} \right)\\
		& + f(\bar{x}^{i},y) - f (\bar{x}^{i}, y^{i+1})  
		+ f(x^{i+1}, \bar{y}^{i+1}) - f(x, \bar{y}^{i+1})     \\
		= &f(x^{i+1},y)-f(x,y^{i+1})\\
		 &+ \frac{1}{2 \sigma_{i}}\left( \norm{y^{i} -y^{i+1}}^{2} 
		 +\norm{y-y^{i+1}}^{2} -\norm{y - y^{i}}^{2} \right) 
		 +\frac{1}{2 \tau_{i}} \left( \norm{x^{i} -x^{i+1}}^{2} 
		 + \norm{x - x^{i+1}}^{2} -\norm{x-x^{i}}^{2} \right)\\
		& -f(x^{i+1},y)+f(x,y^{i+1}) + f(\bar{x}^{i},y) 
		- f (\bar{x}^{i}, y^{i+1})  + f(x^{i+1}, \bar{y}^{i+1}) - f(x, \bar{y}^{i+1}).     
	\end{align*}  
Sum the inequality over $i$ from $p$ to $k$ and 
employ \cref{corollary:betaialphai-1Ineq} to establish that 
\begin{align*}
&\sum^{k}_{i=p} f(x^{i+1},y)-f(x,y^{i+1})\\
\leq &\sum^{k}_{i=p} \frac{1}{2 \tau_{i}} \left(  \norm{x -x^{i}}^{2} -\norm{x -x^{i+1}}^{2}
 - \norm{x^{i} -x^{i+1}}^{2} \right)\\
& +     \sum^{k}_{i=p} \frac{1}{2 \sigma_{i}}  \left( \norm{y - y^{i}}^{2} - 
\norm{y-y^{i+1}}^{2} - \norm{y^{i}-y^{i+1}}^{2} \right)\\
& +   \sum^{k}_{i=p} \left(f(x^{i+1},y) -f(\bar{x}^{i},y) 
+ f(\bar{x}^{i}, y^{i+1})   -f(x,y^{i+1}) - f(x^{i+1}, \bar{y}^{i+1}) + f(x,\bar{y}^{i+1}) \right)\\
\leq & \frac{1}{2 \tau_{p}} \norm{x -x^{p}}^{2} 
-  \frac{1}{2 \tau_{k+1}} \norm{x -x^{k+1}}^{2} 
+  \frac{1}{2 \sigma_{p}} \norm{y -y^{p}}^{2} 
- \left( \frac{1}{2 \sigma_{k+1}}  -
 \frac{\norm{K}^{2}\tau_{k}}{2}  \right)  \norm{y -y^{k+1}}^{2}\\
& + \sum^{k-1}_{i=p} \left( \norm{K} - \frac{1}{2 \tau_{i}}   \right)  \norm{x^{i}-x^{i+1}}^{2} 
+  \sum^{k}_{i=p} \left(  \norm{K} - \frac{1}{2 \sigma_{i}}   \right)  \norm{y^{i}-y^{i+1}}^{2}\\
&+  \norm{K}\norm{x^{p-1} -x^{p}}^{2}  +\abs{\innp{K(x^{p}-x^{p-1}), y^{p}-y}}.
\end{align*} 
After some rearrangement, we obtain the required inequality.
\end{proof}

\begin{lemma} \label{lemma:f-fbounded}
Let $(x,y) \in X \times Y$.
Suppose that    $(\forall i \in \mathbf{N})$ $\norm{K} \beta^{2}_{i} 
= \left( \frac{1}{ \tau_{i}} -\frac{1}{ \tau_{i+1}} \right)$ 
and  $\norm{K} (1-\alpha_{i-1})^{2} =   \left( \frac{1}{  \sigma_{i}} - \frac{1}{  \sigma_{i+1}} 
\right)$. 
Then for every $k \in \mathbf{N}$,
\begin{align*}
&f(\hat{x}_{k}, y)   -f(x, \hat{y}_{k})  \\
\leq &\frac{1}{2  (k+1)} \left( \frac{\norm{x -x^{0}}^{2} }{\tau_{0}} 
-  \frac{\norm{x -x^{k+1}}^{2}}{\tau_{k+1}} 
 +  \frac{ \norm{y -y^{0}}^{2}}{ \sigma_{0} }  \right) 
 -\frac{1}{k+1} \left( \frac{1}{2 \sigma_{k+1}}  
 - \frac{\norm{K}^{2}\tau_{k}}{2}  \right)  \norm{y -y^{k+1}}^{2}\\ 
 &+\frac{1}{k+1} \sum^{k-1}_{i=0} \left( \norm{K} 
 - \frac{1}{2 \tau_{i}}   \right)  \norm{x^{i}-x^{i+1}}^{2} 
 +  \frac{1}{k+1} \sum^{k}_{i=0} \left(  \norm{K} - \frac{1}{2 \sigma_{i}}   \right) 
 \norm{y^{i}-y^{i+1}}^{2}.
	\end{align*} 
\end{lemma}

\begin{proof}
 Recall that $x^{-1}=x^{0}$. Take  
 \cref{lemma:fsumxkykineq}\cref{lemma:fsumxkykineq:f}  and 
 \cref{corollary:betaialphai-1Ineq} into account to derive that for every $k \in 
 \mathbf{N}$,
 \begin{align*}
 &f(\hat{x}_{k}, y)   -f(x, \hat{y}_{k}) \\
 	\leq & \frac{1}{k+1}  \sum^{k}_{i=0} \frac{1}{2 \sigma_{i}}  \left( \norm{y- y^{i}}^{2} 
	- \norm{y-y^{i+1}}^{2} - \norm{y^{i}-y^{i+1}}^{2} \right)\\
 	& +\frac{1}{k+1}  \sum^{k}_{i=0} \frac{1}{2 \tau_{i}} \left(  \norm{x -x^{i}}^{2} 
	-\norm{x -x^{i+1}}^{2} - \norm{x^{i} -x^{i+1}}^{2} \right) \\
 	&+  \frac{1}{k+1} \sum^{k}_{i=0} \left(f(x^{i+1},y) -f(\bar{x}^{i},y) 
	+ f(\bar{x}^{i}, y^{i+1}) -f(x,y^{i+1}) - f(x^{i+1}, \bar{y}^{i+1}) 
	+ f(x,\bar{y}^{i+1})  \right)\\
 	\leq & \frac{1}{2  (k+1)} \left( \frac{\norm{x -x^{0}}^{2} }{\tau_{0}} 
	-  \frac{\norm{x -x^{k+1}}^{2}}{\tau_{k+1}}  
	+  \frac{ \norm{y -y^{0}}^{2}}{ \sigma_{0} }  \right) 
	-\frac{1}{k+1} \left( \frac{1}{2 \sigma_{k+1}}  
	- \frac{\norm{K}^{2}\tau_{k}}{2}  \right)  \norm{y -y^{k+1}}^{2}\\
 	& + \frac{1}{k+1} \sum^{k-1}_{i=0} \left( \norm{K} 
	- \frac{1}{2 \tau_{i}}   \right)  \norm{x^{i}-x^{i+1}}^{2} 
	+  \frac{1}{k+1} \sum^{k}_{i=0} \left(  \norm{K} - \frac{1}{2 \sigma_{i}}   \right)  
	\norm{y^{i}-y^{i+1}}^{2}.
 \end{align*}
\end{proof}

The following easy result will be used several times later. 
\begin{lemma}\label{lemma:infsup}
Suppose that $\limsup_{i \to \infty} \tau_{i}  < \infty$, 
$\limsup_{i \to \infty} \sigma_{i} < \infty$,
 	$\norm{K} \limsup_{i \to \infty} \tau_{i} < \frac{1}{2}$, 
	and $\norm{K} \limsup_{i \to \infty} 
 	\sigma_{i} < \frac{1}{2}$.
	Then there exist $N \in \mathbf{N}$,  $\bar{\tau} \in \mathbf{R}_{++}$, 
	and $\bar{\sigma} \in \mathbf{R}_{++}$ such that for every  $k \in \mathbf{N}$ with 
	$k \geq N$,
	\begin{subequations} \label{eq:lemma:infsup}
\begin{align} 
&\frac{1}{2 \tau_{k+1}} \geq  \frac{1}{2 \bar{\tau}}  >0;\\
& \frac{1}{2 \sigma_{k+1}}  - \frac{\norm{K}^{2}\tau_{k}}{2}   
\geq  \frac{1 }{2 \bar{\sigma}} - \frac{\norm{K}^{2} \bar{\tau}}{2}  >0;\\
& \frac{1}{2 \tau_{k}} -\norm{K}  \geq \frac{1}{2 \bar{\tau}} -\norm{K}   >0;\\
&\frac{1}{2 \sigma_{k}} - \norm{K}  \geq \frac{1}{2 \bar{\sigma} }- \norm{K}>0.
\end{align}
\end{subequations}
\end{lemma}

\begin{proof}
Take $\bar{\tau} > \limsup_{i \to \infty} \tau_{i} \geq 0$ 
and $\bar{\sigma} > \limsup_{i \to \infty} \sigma_{i} \geq 0$  
such that $  \norm{K}\bar{\tau}  < \frac{1}{2}$ and $  \norm{K} \bar{\sigma} < \frac{1}{2}$. 
Note that 
\[ 
 \norm{K}\bar{\tau}  < \frac{1}{2} \text{ and } \norm{K} \bar{\sigma} < \frac{1}{2}	
 \Rightarrow	\norm{K}^{2} \bar{\tau} \bar{\sigma} < \frac{1}{4} <1 
 \Rightarrow \frac{1}{ \bar{\sigma} } - \norm{K}^{2} \bar{\tau} >0.
\]
Hence, it is clear that 
\begin{align*}
&\liminf_{i \to \infty}	 \frac{1}{2 \tau_{i+1}}  >  \frac{1}{2 \bar{\tau}}  >0;\\
& \liminf_{i \to \infty} \left( \frac{1}{2 \sigma_{i+1}}  
- \frac{\norm{K}^{2}\tau_{i}}{2}  \right)   >  \frac{1 }{2 \bar{\sigma}} 
-\frac{\norm{K}^{2} \bar{\tau}}{2} >0;\\
& \liminf_{i \to \infty} \left( \frac{1}{2 \tau_{i}}  - \norm{K}  \right) 
>\frac{1}{2 \bar{\tau}} -  \norm{K} >0;\\
&  \liminf_{i \to \infty} \left(  \frac{1}{2 \sigma_{i}}  - \norm{K}  \right)  
 >   \frac{1}{2 \bar{\sigma}}  -\norm{K} >0.
\end{align*}
Hence,    
there exists $N \in \mathbf{N}$ such that for every  $k \in \mathbf{N}$ with 
$k \geq N$, 
 inequalities in \cref{eq:lemma:infsup} hold. 
\end{proof}

\begin{theorem} \label{theorem:convergencef}
Suppose that $(x^{*}, y^{*}) \in X \times Y$ is a saddle-point of $f$, 
that    $(\forall i \in \mathbf{N})$ 
$\norm{K} \beta^{2}_{i} = 	  
\left( \frac{1}{ \tau_{i}} -\frac{1}{ \tau_{i+1}} \right)$ and  
$\norm{K} (1-\alpha_{i-1})^{2} =   \left( \frac{1}{  \sigma_{i}} 
- \frac{1}{  \sigma_{i+1}} \right)$, 
and that
 $\limsup_{i \to \infty} \tau_{i}  < \infty$, $\limsup_{i \to \infty} \sigma_{i} < \infty$,
$\norm{K} \limsup_{i \to \infty} \tau_{i} < \frac{1}{2}$,
 and $\norm{K} \limsup_{i \to \infty} 
\sigma_{i} < \frac{1}{2}$.
 Then the following statements hold. 
 \begin{enumerate}
 	\item  \label{theorem:convergencef:k} 
	$(\forall k \in \mathbf{N})$  
	$\frac{1}{2 \tau_{k+1}} \norm{x^{*} -x^{k+1}}^{2} +    
	\left( \frac{1}{2 \sigma_{k+1}}  - 
 	\frac{\norm{K}^{2}\tau_{k}}{2}  \right)  \norm{y^{*} -y^{k+1}}^{2} 
 	+ \sum^{k-1}_{i=0} \left( \frac{1}{2 \tau_{i}} -\norm{K}   \right)  \norm{x^{i}-x^{i+1}}^{2} 
	+  \sum^{k}_{i=0} \left(  \frac{1}{2 \sigma_{i}} - \norm{K}   \right)  \norm{y^{i}-y^{i+1}}^{2}
 	\leq  	\frac{1}{2 \tau_{0}} \norm{x^{*} -x^{0}}^{2}  
	+  \frac{1}{2 \sigma_{0}} \norm{y^{*} -y^{0}}^{2}$.
 
 	\item \label{theorem:convergencef:ineq}  
 		Take $\bar{\tau} > \limsup_{i \to \infty} \tau_{i} \geq 0$ 
		and $\bar{\sigma} > \limsup_{i 
 		\to \infty} \sigma_{i} \geq 0$  such that $  \norm{K}\bar{\tau}  < \frac{1}{2}$ and $  
 		\norm{K} \bar{\sigma} < \frac{1}{2}$. 
 Then there exists $N \in \mathbf{N}$ such that for all $k \in \mathbf{N}$ 
 with $k \geq  N+1$ we have  $\frac{1}{2 \bar{\tau}}  \norm{x^{*} -x^{k+1}}^{2} 
 +     \left( \frac{1 }{2 \bar{\sigma}} - \frac{\norm{K}^{2}
 \bar{\tau}}{2}  \right) \norm{y^{*} -y^{k+1}}^{2} + 
 	 \left( \frac{1}{2 \bar{\tau}} 
	 -\norm{K}   \right) \sum^{k-1}_{i=N+1}   
 	 \norm{x^{i}-x^{i+1}}^{2} 
	 + \left(  \frac{1}{2 \bar{\sigma}} - \norm{K}   \right)  
 	 \sum^{k}_{i=N+1}  \norm{y^{i}-y^{i+1}}^{2}
 	\leq   \frac{1}{2 \tau_{0}} \norm{x^{*} -x^{0}}^{2}  
	+  \frac{1}{2 \sigma_{0}} \norm{y^{*} 
 	-y^{0}}^{2}  +
 	\\  \sum^{N}_{i=0} \left( \norm{K} - \frac{1}{2 \tau_{i}}   \right)  
 	\norm{x^{i}-x^{i+1}}^{2} 
	+  \sum^{N}_{i=0} \left(  \norm{K} - \frac{1}{2 \sigma_{i}}   \right) 
	 \norm{y^{i}-y^{i+1}}^{2}$. 
 	
Consequently, $\lim_{k \to \infty} \norm{x^{k}-x^{k+1}} =0$ 
and $\lim_{k \to \infty} \norm{y^{k} -y^{k+1}} =0$.
\item  \label{theorem:convergencef:bounded} $\left((x^{k},y^{k}) \right)_{k \in 
\mathbf{N}}$ 
and $\left( (\hat{x}_{k},\hat{y}_{k}) \right)_{k \in \mathbf{N}}$ are bounded. 
\item \label{theorem:convergencef:fktohat} $f(\hat{x}_{k}, \hat{y}_{k}) \to f (x^{*}, y^{*}) $.
 \end{enumerate}
\end{theorem}

\begin{proof}
\cref{theorem:convergencef:k}:   
Invoke  \cref{corollary:basicxkyksumineq}\cref{corollary:basicxkyksumineq:leqhat} and \cref{corollary:betaialphai-1Ineq}  to establish that 
\begin{align*}
	0 \leq &  
	\sum^{k}_{i=0} \frac{1}{2 \sigma_{i}}  \left( \norm{ y^{*} - y^{i}}^{2} - \norm{ 
	y^{*}-y^{i+1}}^{2} - \norm{y^{i}-y^{i+1}}^{2} \right) \\
	&+  \sum^{k}_{i=0} \frac{1}{2 \tau_{i}} \left(  \norm{x^{*} -x^{i}}^{2} -\norm{x^{*}
	-x^{i+1}}^{2} - \norm{x^{i} -x^{i+1}}^{2}\right)\\
	& +\sum^{k}_{i=0} \left(  f(x^{i+1}, y^{*} ) - f(\bar{x}^{i}, y^{*} ) + f(\bar{x}^{i}, y^{i+1}) 
	-f(x^{*},y^{i+1}) - f(x^{i+1}, \bar{y}^{i+1}) + f(x^{*},\bar{y}^{i+1})   \right) \\
	\leq    & \frac{1}{2 \tau_{0}} \norm{x^{*} -x^{0}}^{2} -  \frac{1}{2 \tau_{k+1}} 
	\norm{x^{*} -x^{k+1}}^{2} +  \frac{1}{2 \sigma_{0}} \norm{y^{*}  -y^{0}}^{2} - \left( 
	\frac{1}{2 \sigma_{k+1}}  - \frac{\norm{K}^{2}\tau_{k}}{2}  \right)  \norm{y^{*}  
	-y^{k+1}}^{2}\\
	& + \sum^{k-1}_{i=0} \left( \norm{K} - \frac{1}{2 \tau_{i}}   \right)  \norm{x^{i}-x^{i+1}}^{2} 
	+  \sum^{k}_{i=0} \left(  \norm{K} - \frac{1}{2 \sigma_{i}}   \right)  \norm{y^{i}-y^{i+1}}^{2}. 
\end{align*}
This entails that for every $k\in \mathbf{N}$,
\begin{align*}
	&\frac{1}{2 \tau_{k+1}} \norm{x^{*} -x^{k+1}}^{2} +    \left( \frac{1}{2 \sigma_{k+1}}  - 
	\frac{\norm{K}^{2}\tau_{k}}{2}  \right)  \norm{y^{*} -y^{k+1}}^{2}\\
	\leq &	\frac{1}{2 \tau_{0}} \norm{x^{*} -x^{0}}^{2}  +  \frac{1}{2 \sigma_{0}} 
	\norm{y^{*} -y^{0}}^{2}  
	+ \sum^{k-1}_{i=0} \left( \norm{K} - \frac{1}{2 \tau_{i}}   \right)  \norm{x^{i}-x^{i+1}}^{2} 
	+  \sum^{k}_{i=0} \left(  \norm{K} - \frac{1}{2 \sigma_{i}}   \right)  \norm{y^{i}-y^{i+1}}^{2}.
\end{align*}
After some rearrangement, we deduce the required inequality in \cref{theorem:convergencef:k}.

\cref{theorem:convergencef:ineq}: 	
 As a consequence of \cref{lemma:infsup},
  there exists $N \in \mathbf{N}$ such that for all $i \in \mathbf{N}$ with $i \geq N$,
\begin{subequations}\label{eq:theorem:convergencef:<}
	\begin{align} 
		&\frac{1}{2 \tau_{i+1}} \geq  \frac{1}{2 \bar{\tau}}  >0;\\
		& \frac{1}{2 \sigma_{i+1}}  - \frac{\norm{K}^{2}\tau_{i}}{2}   
		\geq  \frac{1 }{2 \bar{\sigma}} - \frac{\norm{K}^{2} \bar{\tau}}{2}  >0;\\
		& \frac{1}{2 \tau_{i}} -\norm{K}  \geq \frac{1}{2 \bar{\tau}} -\norm{K}   >0;\\
		&\frac{1}{2 \sigma_{i}} - \norm{K}  \geq \frac{1}{2 \bar{\sigma} }- \norm{K}>0.
	\end{align}
\end{subequations}

 Combine    \cref{eq:theorem:convergencef:<}   with the inequality 
 in \cref{theorem:convergencef:k} above to yield that for all $k \in \mathbf{N}$ with $k 
 \geq N+1$,
 \begin{align*}
 	&\frac{1}{2 \bar{\tau}}  \norm{x^{*} -x^{k+1}}^{2} +     
	\left( \frac{1 }{2 \bar{\sigma}} - 
 	\frac{\norm{K}^{2} \bar{\tau}}{2}  \right) \norm{y^{*} -y^{k+1}}^{2} \\
 	 &+ \sum^{N}_{i=0} \left(  \frac{1}{2 \tau_{i}} -\norm{K}   \right)  \norm{x^{i}-x^{i+1}}^{2} 
 	+ \left( \frac{1}{2 \bar{\tau}} -\norm{K}   \right) \sum^{k-1}_{i=N+1}   \norm{x^{i}-x^{i+1}}^{2} \\
 	&+  \sum^{N}_{i=0} \left(  \frac{1}{2 \sigma_{i}}  - \norm{K}  \right)  \norm{y^{i}-y^{i+1}}^{2}
 	+ \left(  \frac{1}{2 \bar{\sigma}} - \norm{K}   \right)  \sum^{k}_{i=N+1}  \norm{y^{i}-y^{i+1}}^{2}\\
 	\leq    &\frac{1}{2 \tau_{k+1}} \norm{x^{*} -x^{k+1}}^{2} 
	+    \left( \frac{1}{2 \sigma_{k+1}}  - 
 	\frac{\norm{K}^{2}\tau_{k}}{2}  \right)  \norm{y^{*}  -y^{k+1}}^{2} \\
 	&+ \sum^{k-1}_{i=0} \left( \frac{1}{2 \tau_{i}} -\norm{K}   \right)  \norm{x^{i}-x^{i+1}}^{2} 
	+  \sum^{k}_{i=0} \left(  \frac{1}{2 \sigma_{i}} - \norm{K}   \right)  \norm{y^{i}-y^{i+1}}^{2}\\
 	\leq  	&\frac{1}{2 \tau_{0}} \norm{x^{*} -x^{0}}^{2}  +  \frac{1}{2 \sigma_{0}} 
 	\norm{y^{*}  -y^{0}}^{2},
 \end{align*}
which leads to the desired inequality in \cref{theorem:convergencef:ineq}.
Because 
$\frac{1}{2 \bar{\tau}} -  \norm{K} >0$,  $\frac{1}{2 \bar{\sigma}} - \norm{K}  >0$, 
and $\frac{1 }{2 \bar{\sigma}} -\frac{\norm{K}^{2} \bar{\tau}}{2} >0$,
it is easy to see that 
$\lim_{k \to \infty} \norm{x^{k}-x^{k+1}} =0$ and 
$\lim_{k \to \infty} \norm{y^{k} -y^{k+1}} =0$.
	
\cref{theorem:convergencef:bounded}: 
By  the inequality in \cref{theorem:convergencef:ineq}, 
bearing $\frac{1}{2 \bar{\tau}} -  \norm{K} >0$,  $\frac{1}{2 \bar{\sigma}} - \norm{K}  >0$, 
and $\frac{1 }{2 \bar{\sigma}} -\frac{\norm{K}^{2} \bar{\tau}}{2} >0$ in mind, 
we see clearly the boundedness of $\left(  (x^{k},y^{k}) \right)_{k \in \mathbf{N}}$.

 Taking the construction of $\left( (\hat{x}_{k},\hat{y}_{k}) \right)_{k \in \mathbf{N}}$ 
 presented in \cref{eq:hatx_ky_k} into account, 
 we know that the boundedness of $\left( (\hat{x}_{k},\hat{y}_{k}) \right)_{k \in 
 \mathbf{N}}$ 
 follows clearly from that of  $\left(  (x^{k},y^{k}) \right)_{k \in \mathbf{N}}$.  
	 	
\cref{theorem:convergencef:fktohat}:  Let $k \in \mathbf{N}$.

Applying  \cref{proposition:fxyk+1-xyhat}\cref{proposition:fxyk+1-xyhat:k-hat} and 
employing \cref{corollary:betaialphai-1Ineq} with $x=x^{*}$ and $y=\hat{y}_{k}$, 
we obtain that
\begin{subequations} \label{eq:theorem:convergencef:khat}
	\begin{align}
		& f(\hat{x}_{k}, \hat{y}_{k}) -f(x^{*}, y^{*})  + \frac{1}{2 \tau_{k+1} (k+1)} 
		\norm{x^{*}  
		-x^{k+1}}^{2} +\frac{1}{k+1} \left( \frac{1}{  2\sigma_{k+1}}  -  
		\frac{\norm{K}^{2}\tau_{k}}{2}   \right)  \norm{\hat{y}_{k}-y^{k+1}}^{2} \\
		\leq &
		\frac{1}{2 \tau_{0}(k+1)} \norm{x^{*} -x^{0}}^{2} +  \frac{1}{2 \sigma_{0}(k+1)} 
		\norm{\hat{y}_{k} -y^{0}}^{2}\\   
		& +  \frac{1}{k+1} \sum^{k-1}_{i=0} \left( \norm{K} - \frac{1}{2 \tau_{i}}   \right) 
		 \norm{x^{i}-x^{i+1}}^{2} +  
		  \frac{1}{k+1} \sum^{k}_{i=0} \left(  \norm{K} 
		  - \frac{1}{2 \sigma_{i}} \right) \norm{y^{i}-y^{i+1}}^{2}.
	\end{align}
\end{subequations} 
Similarly, invoke \cref{proposition:fxyk+1-xyhat}\cref{proposition:fxyk+1-xyhat:hat-k} and 
apply  \cref{corollary:betaialphai-1Ineq} with $x=\hat{x}_{k}$ and $y=y^{*}$ to establish
\begin{subequations}\label{eq:theorem:convergencef:hatk}
	\begin{align}
		& f  (x^{*}, y^{*})	-f(\hat{x}_{k}, \hat{y}_{k}) + \frac{1}{2 \tau_{k+1} (k+1)} 
		\norm{\hat{x}_{k} 
		-x^{k+1}}^{2} +\frac{1}{k+1}  \left( \frac{1}{2 \sigma_{k+1}}  - 
		\frac{\norm{K}^{2}\tau_{k}}{2}  \right)  \norm{y^{*} -y^{k+1}}^{2}\\
		\leq &  \frac{1}{2 \tau_{0}(k+1)} \norm{\hat{x}_{k} -x^{0}}^{2} +  \frac{1}{2 
		\sigma_{0}(k+1)} \norm{y^{*} -y^{0}}^{2} \\
		& +  \frac{1}{k+1} \sum^{k-1}_{i=0} \left( \norm{K} 
		- \frac{1}{2 \tau_{i}}   \right)  \norm{x^{i}-x^{i+1}}^{2} 
		+   \frac{1}{k+1}  \sum^{k}_{i=0} \left(  \norm{K} 
		- \frac{1}{2 \sigma_{i}}   \right)  \norm{y^{i}-y^{i+1}}^{2}. 
	\end{align}
\end{subequations}

 Combine \cref{eq:theorem:convergencef:khat} and 
 \cref{eq:theorem:convergencef:hatk} with \cref{eq:theorem:convergencef:<} to obtain 
 that  for all $k \in \mathbf{N}$ with $k \geq N+1$,  
\begin{align*}
	&f(\hat{x}_{k}, \hat{y}_{k}) -f (x^{*}, y^{*})  \\
	 \leq &
	\frac{1}{2 \tau_{0}(k+1)} \norm{x^{*} -x^{0}}^{2} +  \frac{1}{2 \sigma_{0}(k+1)} 
	\norm{\hat{y}_{k} -y^{0}}^{2}\\   
	& +  \frac{1}{k+1} \sum^{N}_{i=0} \left( \norm{K} 
	- \frac{1}{2 \tau_{i}}   \right)  \norm{x^{i}-x^{i+1}}^{2} 
	+   \frac{1}{k+1} \sum^{N}_{i=0} \left(  \norm{K} - \frac{1}{2 \sigma_{i}}   \right)  \norm{y^{i}-y^{i+1}}^{2}
\end{align*}
and that 
\begin{align*}
	& f  (x^{*}, y^{*})	-f(\hat{x}_{k}, \hat{y}_{k}) \\ 
	\leq &  \frac{1}{2 \tau_{0}(k+1)} \norm{\hat{x}_{k} -x^{0}}^{2} 
	+  \frac{1}{2 \sigma_{0}(k+1)} \norm{y^{*} -y^{0}}^{2} \\
	& +  \frac{1}{k+1} \sum^{N}_{i=0} \left( \norm{K} 
	- \frac{1}{2 \tau_{i}}   \right)  \norm{x^{i}-x^{i+1}}^{2} 
	+   \frac{1}{k+1}  \sum^{N}_{i=0} \left(  \norm{K} - \frac{1}{2 \sigma_{i}}   \right)  \norm{y^{i}-y^{i+1}}^{2}. 
\end{align*}
The last two inequalities   together with the boundedness results proved in 
\cref{theorem:convergencef:bounded} above guarantee that $f(\hat{x}_{k}, \hat{y}_{k}) 
\to f(x^{*}, y^{*})  $. 
\end{proof}

\begin{theorem} \label{theorem:pointsconvergence}
	Suppose that $(x^{*}, y^{*}) \in X \times Y$ is a saddle-point of $f$, 
	that    $(\forall i \in \mathbf{N})$ 
	$\norm{K} \beta^{2}_{i} = 	  \left( \frac{1}{ \tau_{i}} -\frac{1}{ \tau_{i+1}} 
	\right)$ and  $\norm{K} (1-\alpha_{i-1})^{2} =   \left( \frac{1}{  \sigma_{i}} - \frac{1}{  
	\sigma_{i+1}} \right)$, and that
	$\limsup_{i \to \infty} \tau_{i}  < \infty$, $\limsup_{i \to \infty} \sigma_{i} < \infty$,
	$\norm{K} \limsup_{i \to \infty} \tau_{i} < \frac{1}{2}$, 
	and $\norm{K} \limsup_{i \to \infty} \sigma_{i} < \frac{1}{2}$.
	Then the following statements hold. 
\begin{enumerate}
\item \label{theorem:pointsconvergence:cluster} 
$\left((\hat{x}_{k},\hat{y}_{k}) \right)_{k \in \mathbf{N}}$  
has at least one weakly sequential cluster point and every weakly sequential 
cluster point of $\left((\hat{x}_{k},\hat{y}_{k}) \right)_{k \in \mathbf{N}}$  
is a saddle-point of $f$.
\item  \label{theorem:pointsconvergence:clusterfinitedim} Suppose that 
$\mathcal{H}_{1}$ and $\mathcal{H}_{2}$ are finite-dimensional  and that $\sup_{k 
\in \mathbf{N}} \alpha_{k} < +\infty$, $\sup_{k \in \mathbf{N}} \beta_{k} < +\infty$, 
$0<\inf_{k \in \mathbf{N}} \sigma_{k} \leq \sup_{k \in \mathbf{N}} \sigma_{k} < 
\infty$, and $0<\inf_{k \in \mathbf{N}} \tau_{k} \leq \sup_{k \in \mathbf{N}} \tau_{k} 
< \infty$.  Then $\left((x^{k},y^{k}) \right)_{k \in \mathbf{N}}$ converges to a 
saddle-point of $f$.
\end{enumerate}
\end{theorem}

\begin{proof}
According to \cref{lemma:infsup}, 
because  $\norm{K}  \limsup_{i \to \infty} \tau_{i} < \frac{1}{2}$ 
and $\norm{K}  \limsup_{i \to \infty} \sigma_{i} < \frac{1}{2}$,  
there exists $N \in \mathbf{N}$ such that for all $k \in \mathbf{N}$ with $k \geq N$ 
\begin{align} \label{eq:theorem:pointsconvergence}
\frac{1}{2 \tau_{k+1}}   >0,
\frac{1}{2 \sigma_{k+1}}  - \frac{\norm{K}^{2}\tau_{k}}{2}     >0,
\frac{1}{2 \tau_{k}} -\norm{K}    >0,
\text{ and } \frac{1}{2 \sigma_{k}} - \norm{K}  >0.
\end{align}

\cref{theorem:pointsconvergence:cluster}: 
Employ \cref{eq:theorem:pointsconvergence} and  
\cref{lemma:f-fbounded} to derive that for all $k \in \mathbf{N}$ with 
$k \geq N+1$ and $(x,y) \in X \times Y $,
\begin{align*}
&f(\hat{x}_{k}, y)   -f(x, \hat{y}_{k}) \\
\leq &\frac{1}{2  (k+1)} \left( \frac{\norm{x -x^{0}}^{2} }{\tau_{0}}  
+  \frac{ \norm{y -y^{0}}^{2}}{ \sigma_{0} }  \right)  
+\frac{1}{k+1} \sum^{N}_{i=0} \left( \norm{K} 
- \frac{1}{2 \tau_{i}}   \right)  \norm{x^{i}-x^{i+1}}^{2} \\
&+  \frac{1}{k+1} \sum^{N}_{i=0} \left(  \norm{K} - \frac{1}{2 \sigma_{i}}   \right)  
\norm{y^{i}-y^{i+1}}^{2}.
\end{align*}
which immediately leads to 
\begin{align} \label{eq:theorem:pointsconvergence:tk}
(\forall (x,y) \in X \times Y) \quad \limsup_{k \to \infty}f(\hat{x}_{k}, y)   -f(x, \hat{y}_{k})  
\leq 0.
\end{align}
According to the definition \cref{eq:fspecialdefine} of $f$ and our assumptions therein, 
it is clear that $(\forall (x,y) \in X \times Y)$  $f(\cdot, y)$ and $-f(x, \cdot)$ 
are proper, convex, and  lower semicontinuous. Hence,
employing \cref{theorem:convergencef}\cref{theorem:convergencef:bounded} 
and \cref{eq:theorem:pointsconvergence:tk}, and applying 
\cref{lemma:clustersaddletkxy} with  $\left((x^{k},y^{k}) \right)_{k \in \mathbf{N}}$  being  
$\left((\hat{x}^{k},\hat{y}^{k}) \right)_{k \in \mathbf{N}}$,  we obtain the required results. 
	
\cref{theorem:pointsconvergence:clusterfinitedim}: 
We divide the proof into the following steps. 
	
\emph{Step~1}:
 Based on \cref{theorem:convergencef}\cref{theorem:convergencef:bounded}, 
 $\left((x^{k},y^{k}) \right)_{k \in \mathbf{N}}$  is a bounded sequence in a 
 finite-dimensional space. 
 So, due to  \cite[Lemma~2.45]{BC2017}, 
 there exists at least one cluster point $(x^{*}, y^{*})$ of $\left((x^{k}, y^{k}) \right)_{k 
 \in \mathbf{N}}$. 
 Moreover, bearing \cref{theorem:convergencef}\cref{theorem:convergencef:ineq}       
 and 	\cref{lemma:clustersaddle-point} in mind, 
 we observe that every cluster point of $\left((x^{k}, y^{k}) \right)_{k \in \mathbf{N}}$
  is a saddle-point of $f$.
	
\emph{Step~2}: In view of the result provided in Step~1 above, 
there exists one cluster point $(x^{*}, y^{*})$ 
(that is also a saddle-point of $f$) and one subsequence 
$\left( (x^{k_{i}},y^{k_{i}})\right)_{k \in \mathbf{N}}$   
of $\left((x^{k}, y^{k}) \right)_{k \in \mathbf{N}}$ such that 
\begin{align} \label{eq:theorem:pointsconvergence:convergentsequence:to}
		x^{k_{i}} \to x^{*} \quad \text{and} \quad y^{k_{i}} \to y^{*}.
\end{align}
Notice that the fact that  $(x^{*}, y^{*})$ is a saddle-point implies 
that $(\forall j \in \mathbf{N})$ $ f(x^{j},y^{*})-f(x^{*},y^{j}) \geq 0 $, 
by \cref{fact:saddlepoint}. 
Let $i \in \mathbf{N}$.
Apply  \cref{corollary:Ineqxy} with $x=x^{*}$, $y=y^{*}$, and $p=k_{i}>N$ 
with $k \geq k_{i}$  to observe that
\begin{subequations} \label{eq:theorem:pointsconvergence:convergentsequence:>=}
\begin{align}
&	\frac{1}{2 \tau_{k_{i}}} \norm{x^{*} -x^{k_{i}}}^{2} 
+\frac{1}{2 \sigma_{k_{i}}} \norm{y^{*} -y^{k_{i}}}^{2}  \\
		\geq &   \frac{1}{2 \tau_{k+1}} \norm{x^{*} -x^{k+1}}^{2} 
		+ \left( \frac{1}{2 \sigma_{k+1}}  
		- \frac{\norm{K}^{2}\tau_{k}}{2}  \right)  \norm{y^{*} -y^{k+1}}^{2}  
		-\sum^{k-1}_{j=k_{i}} \left( \norm{K} 
		- \frac{1}{2 \tau_{j}}   \right)  \norm{x^{j}-x^{j+1}}^{2} \\
		&-  \sum^{k}_{j=k_{i}} \left(  \norm{K} - \frac{1}{2 \sigma_{j}}   \right) 
		 \norm{y^{j}-y^{j+1}}^{2}
		- \norm{K}\norm{x^{k_{i}-1} -x^{k_{i}}}^{2}  
		-\abs{\innp{K(x^{k_{i}}-x^{k_{i}-1}), y^{k_{i}}-y^{*}}}\\
	\geq & \frac{1}{2 \tau_{k+1}} \norm{x^{*} -x^{k+1}}^{2} 
	+ \left( \frac{1}{2 \sigma_{k+1}}  
	- \frac{\norm{K}^{2}\tau_{k}}{2}  \right)  \norm{y^{*} -y^{k+1}}^{2}  \\
&	- \norm{K}\norm{x^{k_{i}-1} -x^{k_{i}}}^{2}  
-\abs{\innp{K(x^{k_{i}}-x^{k_{i}-1}), y^{k_{i}}-y^{*}}},
	\end{align}
\end{subequations}
where in the last inequality above we use that 
\[ 
	(\forall j \in \mathbf{N} \text{ with } j \geq k_{i} > N) 
	\quad \frac{1}{2 \tau_{j}} -\norm{K}    >0 
	\text{ and } \frac{1}{2 \sigma_{j}} - \norm{K}  >0.
\]
Take the convergence \cref{eq:theorem:pointsconvergence:convergentsequence:to}, 
$\lim_{k \to \infty} \norm{x^{k}-x^{k+1}} =0$ and 
$\lim_{k \to \infty} \norm{y^{k} -y^{k+1}} =0$ proved   in \cref{theorem:convergencef}\cref{theorem:convergencef:ineq},  
the boundedness of $((x^{k},y^{k}))_{k \in \mathbf{N}}$ 
proved in \cref{theorem:convergencef}\cref{theorem:convergencef:bounded},    
and the assumptions $0<\inf_{k \in \mathbf{N}} \sigma_{k}  $ 
and $0<\inf_{k \in \mathbf{N}} \tau_{k}  $ into account to deduce that 
\begin{subequations}\label{eq:theorem:pointsconvergence:convergentsequence}
	\begin{align}
		&\frac{1}{2 \tau_{k_{i}}} \norm{x^{*} -x^{k_{i}}}^{2} +\frac{1}{2 \sigma_{k_{i}}} 
		\norm{y^{*} -y^{k_{i}}}^{2} \to 0;\\  
		&  \norm{K}\norm{x^{k_{i}-1} -x^{k_{i}}}^{2}  
		+\abs{\innp{K(x^{k_{i}}-x^{k_{i}-1}), y^{k_{i}}-y^{*}}} \to 0.
	\end{align}
\end{subequations}
In view of 
\cref{eq:theorem:pointsconvergence:convergentsequence:>=} and 
\cref{eq:theorem:pointsconvergence:convergentsequence},
we have that
\begin{align}   \label{eq:theorem:pointsconvergence:convergentsequence:limsup}
\limsup_{k \to \infty}	  \frac{1}{2 \tau_{k+1}} \norm{x^{*} -x^{k+1}}^{2} 
+ \left( \frac{1}{2 \sigma_{k+1}}  
- \frac{\norm{K}^{2}\tau_{k}}{2}  \right)  \norm{y^{*} -y^{k+1}}^{2}  \leq 0.
\end{align}
On the other hand, according to \cref{lemma:infsup}, 
there exist $\bar{\tau} \in 
\mathbf{R}_{++}$ and $\bar{\sigma} \in \mathbf{R}_{++}$ such that 
for every  $k \in \mathbf{N}$ 
with $k \geq N$,
\begin{align}  \label{eq:theorem:pointsconvergence:clusterfinitedim:coefficients}
0<\frac{1}{2 \bar{\tau}} \leq\frac{1}{2 \tau_{k+1}} \quad \text{and} 
\quad 0<  \frac{1 }{2 \bar{\sigma}} - \frac{\norm{K}^{2} \bar{\tau}}{2} 
<	\frac{1}{2 \sigma_{k+1}}  - \frac{\norm{K}^{2}\tau_{k}}{2},
\end{align}
which implies that 
\begin{align} \label{eq:theorem:pointsconvergence:convergentsequence:liminf}
	\liminf_{k \to \infty}	  \frac{1}{2 \tau_{k+1}} \norm{x^{*} -x^{k+1}}^{2} 
	+ \left( \frac{1}{2 \sigma_{k+1}}  
	- \frac{\norm{K}^{2}\tau_{k}}{2}  \right)  \norm{y^{*} -y^{k+1}}^{2}  \geq 0.
\end{align}
Combine \cref{eq:theorem:pointsconvergence:convergentsequence:limsup}, \cref{eq:theorem:pointsconvergence:convergentsequence:liminf}, and \cref{eq:theorem:pointsconvergence:clusterfinitedim:coefficients} to yield that
\begin{align}  \label{eq:theorem:pointsconvergence:clusterfinitedim:step2}
\lim_{k \to \infty} \frac{1}{2 \tau_{k+1}} \norm{x^{*} -x^{k+1}}^{2}=0 
\quad \text{and} \quad
		\lim_{k \to \infty}	\left( \frac{1}{2 \sigma_{k+1}}  
		- \frac{\norm{K}^{2}\tau_{k}}{2}  \right)  \norm{y^{*} -y^{k+1}}^{2} =0.
\end{align}

\emph{Step~3}: Based on our assumptions, 
$(\forall k \in \mathbf{N})$ 
$0< \frac{1}{2 \tau_{k+1}} \leq \frac{1}{2 \inf_{i \in \mathbf{N} }\tau_{i+1}}
 <+\infty$ 
and $0<  \frac{1}{2 \sigma_{k+1}}  - \frac{\norm{K}^{2}\tau_{k}}{2}   \leq 
\frac{1}{2 \inf_{i \in \mathbf{N} }\sigma_{i}}
<+\infty$. 
This combined with  \cref{eq:theorem:pointsconvergence:clusterfinitedim:step2} 
proved in Step~2 above leads to
\[ 
		x^{k} \to x^{*} \quad \text{and} \quad y^{k} \to y^{*},
\]
	which, connecting with our result proved in Step~1 above, ensures the required result. 
	\end{proof}

\section{Numerical experiments}\label{section:NumericalExperiments}
 
 In this section, unless stated otherwise, $\left((x^{k}, y^{k})\right)_{0 \leq k \leq N}$    
 is generated by our iterate scheme \cref{eq:lemma:algorithmsymplity} with $N$ being 
 the total number of iterates.
  Set
\[ 
 	(\forall k \in \mathbf{N})\quad (\hat{x}_{k},\hat{y}_{k})
	=\left(\frac{1}{k+1} \sum^{k+1}_{i=1} x^{i},  \frac{1}{k+1} \sum^{k+1}_{i=1} y^{i}\right).
\]
As we did in \cref{section:ConvergenceResults}, here
we consider also  the convex-concave function in the form of 
\begin{align}  \label{eq:num:fspecialdefine}
	(\forall (x,y) \in X \times Y) \quad 	f(x,y) =\innp{Kx, y} + g(x)  -h(y),
\end{align} 
where $K \in \mathcal{B}(\mathcal{H}_{1} , \mathcal{H}_{2} )$, 
and $g: \mathcal{H}_{1} \to \mathbf{R} \cup \{ +\infty\}$ and 
$h: \mathcal{H}_{2} \to  \mathbf{R} \cup \{ +\infty\}$  
are proper, convex, and lower semicontinuous. 

We denote a saddle-point of $f$ as $(x^{*}, y^{*})$.
In \cref{theorem:convergencef}\cref{theorem:convergencef:fktohat}, we 
showed the convergence of  $f(\hat{x}_{k}, \hat{y}_{k}) \to f(x^{*}, y^{*})  $.
Moreover, 
in \cref{theorem:pointsconvergence}\cref{theorem:pointsconvergence:clusterfinitedim}, 
we proved that $\left((x^{k},y^{k}) \right)_{k \in \mathbf{N}}$ converges to a 
saddle-point of $f$. 
So if only $g$ and $h$ in \cref{eq:num:fspecialdefine} are 
continuous, then we have $f(x^{k},y^{k})  \to f(x^{*}, y^{*})$.
To verify \cref{theorem:convergencef,theorem:pointsconvergence}, 
in our experiments  below,
we consider parameters used in scheme \cref{eq:lemma:algorithmsymplity} as follows:
\[ 
(\forall k \in \mathbf{N}) \quad \tau_{k} =\sigma_{k} = \frac{1}{4 \norm{K}} \left(1 - 
\frac{1}{k+1}\right) \text{ and } 1- \alpha_{k} =\beta_{k} = 2 \sqrt{\frac{1}{k(k+1)}}.
\]
This setting implies $(\forall k \in \mathbf{N})$ $1- \alpha_{k} = \beta_{k} =  
\sqrt{\frac{1}{\norm{K}} \left( \frac{1}{\tau_{k}}  -\frac{1}{\tau_{k+1}}   \right)}  
=\sqrt{\frac{1}{\norm{K}} \left( \frac{1}{\sigma_{k}}  -\frac{1}{\sigma_{k+1}}   \right)}  $ 
and $\lim_{k \to \infty} \tau_{k} =\lim_{k \to \infty} \sigma_{k} = \frac{1}{4 \norm{K}}$, 
which are required in our convergence proofs.

As we mentioned before, in \cite[Theorem~1]{ChambollePock2011}, the authors 
considered the convergence of iterates $((x^{k},y^{k}))_{k \in \mathbf{N}}$ 
generated by the iterate scheme 
\cref{eq:lemma:algorithmsymplity} with $(\forall k \in \mathbf{N})$ $\sigma_{k} \equiv 
\sigma $, $\tau_{k} \equiv \tau $, $\alpha_{k} \equiv 1$, and $\beta_{k} \equiv 0$ 
under the assumption of  $\tau \sigma \norm{K}^{2} <1$. 
Moreover, in \cite[Section~5]{ChambollePock2011}, 
under some assumptions on the uniformly convexity of $g$ or $h$ in  
\cref{eq:num:fspecialdefine}, 
the authors studied the acceleration of the
convergence of iterates $((x^{k},y^{k}))_{k \in \mathbf{N}}$ generated by 
the scheme \cref{eq:lemma:algorithmsymplity} associated
with $(\forall k \in \mathbf{N})$  $\beta_{k} \equiv 0$,
$\alpha_{k} \in [0,1]$ or $\alpha_{k} \equiv \alpha $ in a certain range with 
a upper bound $1$, 
and some other constraints on the involved parameters
 $((\alpha_{k}, \tau_{k},\sigma_{k}))_{k \in \mathbf{N}}$. 
 In particular, their assumptions imply that $\tau \sigma \norm{K}^{2} \leq 1$ 
or $(\forall k \in \mathbf{N})$ $\tau_{k} \sigma_{k} \norm{K}^{2} \leq 1$. 

Motivated by \cite{ChambollePock2011}, in our numerical experiments, we will also 
consider  iterate schemes associated
with constant parameters and compare these iterate schemes
with ours. 

\subsection{Matrix game} \label{subsection:matrixgame}
In this subsection, we consider an easy version of the game interpretation of 
saddle-point problems  
(see, e.g., \cite[Section~5.4.3]{BV2004} for details).  The subject function in this case 
is  defined as 
\[ 
f:  X \times  Y \to \mathbf{R} :  (x,y) \mapsto f(x,y) = x^{T} Cy,
\]
where $C \in \mathbf{R}^{m \times n}$, and  $X \subseteq \mathbf{R}^{n}$ and $Y 
\subseteq \mathbf{R}^{m}$ are closed and convex.
Now, we have $K=C^{T}$, $g=0$, and $h=0$ in \cref{eq:num:fspecialdefine}.

We borrow the example of matrix game in 
\cite[Sections~5.2 and 5.3]{SLB2023} and set 
\begin{align*}
	&C = \begin{pmatrix}
		1 &2\\
		3&1
	\end{pmatrix},\\
	&X:= \{ x \in \mathbf{R}^{2} ~:~ \sum^{2}_{i=1} x_{i} =1 \text{ and }(\forall i \in \{1,2\})~ 
	x_{i} 
	\geq 0  \}, \text{ and}\\
	&Y:= \{ y \in \mathbf{R}^{2} ~:~ \sum^{2}_{i=1} y_{i} =1 \text{ and }(\forall i \in \{1,2\}) 
	~y_{i} 
	\geq 0  \}.
\end{align*}
According to \cite[Section~5.3]{SLB2023}, the optimal value (that is the value of $f$ 
over the saddle-point) is $1.6667$ with keeping 4 decimal places.
(This example is also used in our \cite[Section~4.4]{Oy2023subgradient}.)

In our experiment associated with 
\cref{fig:appmg_proximity_methods_random_00.png},  
we calculate  $\left(  f(x^{k}, y^{k})  
\right)_{1 \leq k \leq 500}$ and $\left(  f(\hat{x}_{k},\hat{y}_{k}) \right)_{1 \leq k \leq 
	500}$ 
with a random initial point $(x^{0},y^{0}) \in X \times Y$ and with $(\forall k \in 
\mathbf{N})$ 
$\tau_{k} =\sigma_{k} = \frac{1}{4 \norm{K}} \left(1 - \frac{1}{k+1}\right)$
and  $1- \alpha_{k} =\beta_{k} = 2 \sqrt{\frac{1}{k(k+1)}}$.  
Clearly, the convergence result displayed in 
\cref{fig:appmg_proximity_methods_random_00.png} 
validates our theoretical results shown in 
\cref{theorem:convergencef,theorem:pointsconvergence}. 
Moreover, it is clear from \cref{fig:appmg_proximity_methods_random_00.png} that
$\left(  f(\hat{x}_{k},\hat{y}_{k}) \right)_{1 \leq k \leq 	500}$ 
converges faster than $\left(  f(x^{k}, y^{k})  \right)_{1 \leq k \leq 500}$ in this case.
\begin{figure}[H]
	\centering
	\includegraphics[width=0.6\textwidth]{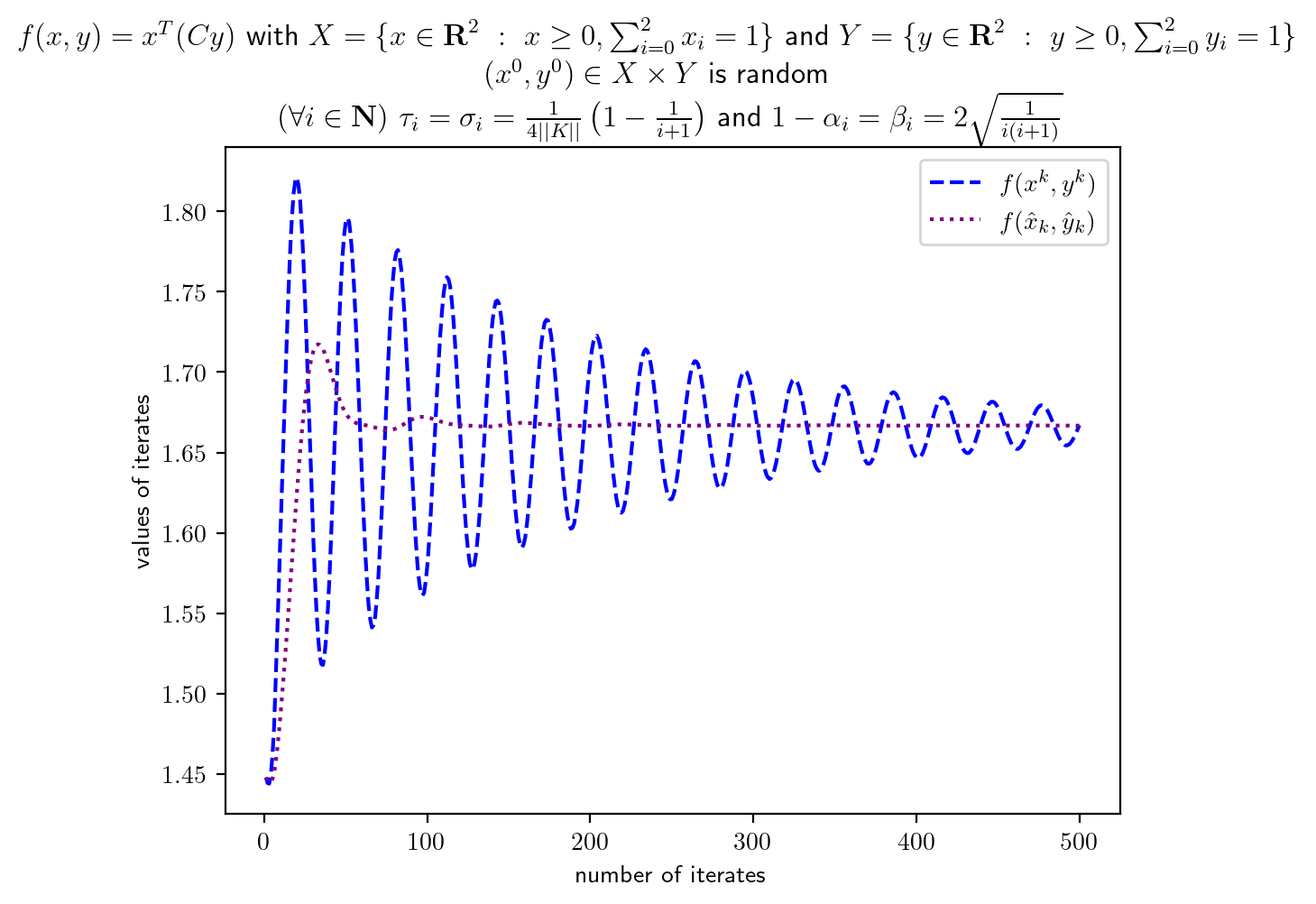}
	\caption{Convergence of matrix game with a random initial point}
	\label{fig:appmg_proximity_methods_random_00.png}
\end{figure}

Motivated by theoretical results in  \cite{ChambollePock2011}, 
we consider $(\forall k \in \mathbf{N})$  $\alpha_{k} \equiv 1$,  $\beta_{k} \equiv 0$, 
$\sigma_{k} \equiv 
\frac{1}{\eta \norm{K}} $, and $\tau_{k} \equiv  
\frac{1}{\eta \norm{K}} $ with $\eta \geq 1$ 
to satisfy $(\forall k \in \mathbf{N})$ $\tau_{k} \sigma_{k} \norm{K}^{2} \leq 1$. We 
randomly choose an initial point  $(x^{0},y^{0}) \in X \times Y$ and 
calculate $\left(  f(x^{k}, y^{k})  \right)_{1 \leq k \leq 500}$ 
and $\left(  f(\hat{x}_{k},\hat{y}_{k}) \right)_{1 \leq k \leq 	500}$
for each group of constant involved parameters. 
One result is presented in  \cref{fig:appmg_proximity_methods_2_00.png} 
in which we show also $\left(  f(x^{k}, y^{k})  \right)_{1 \leq k \leq 500}$ and 
$\left(  f(\hat{x}_{k},\hat{y}_{k}) \right)_{1 \leq k \leq 500}$ 
with $(\forall k \in \mathbf{N})$ $\tau_{k} =\sigma_{k} = \frac{1}{4 \norm{K}} \left(1 - 
\frac{1}{k+1}\right)$ and $1- \alpha_{k} =\beta_{k} = 2 \sqrt{\frac{1}{k(k+1)}}$ 
as a reference. 
In view of \cref{fig:appmg_proximity_methods_2_00.png}, when 
$(\forall k \in \mathbf{N})$  $(\alpha_{k}, \beta_{k}) \equiv (1,0)$ and $\eta =5$,
neither  $\left(  f(x^{k}, y^{k})  \right)_{1 \leq k \leq 500}$ nor
$\left(  f(\hat{x}_{k},\hat{y}_{k}) \right)_{1 \leq k \leq 500}$ converges to the required 
value; 
when $(\alpha_{k}, \beta_{k}) \equiv (1,0)$ and $\eta =100$, both
$\left(  f(x^{k}, y^{k})  \right)_{1 \leq k \leq 500}$ and
$\left(  f(\hat{x}_{k},\hat{y}_{k}) \right)_{1 \leq k \leq 500}$
show little convergent performance in 500 iterates;
when $(\alpha_{k}, \beta_{k}) \equiv (1,0)$ and $\eta =10$,
although $\left(  f(\hat{x}_{k},\hat{y}_{k}) \right)_{1 \leq k \leq 500}$ converges, 
we could not see  clear convergence of $\left(  f(x^{k}, y^{k})  \right)_{k \in \mathbf{N}}$
 in $500$ iterates.
Based on \cref{fig:appmg_proximity_methods_2_00.png}, 
$\left(  f(\hat{x}_{k},\hat{y}_{k}) \right)_{1 \leq k \leq 500}$ performs better than
$\left(  f(x^{k}, y^{k})  \right)_{1 \leq k \leq 500}$ in all cases; 
our iterate scheme is a reasonable choice if we are unsure which constants should 
be used for parameters in iterate schemes. 
\begin{figure}[H]
	\centering
	\includegraphics[width=0.8\textwidth]{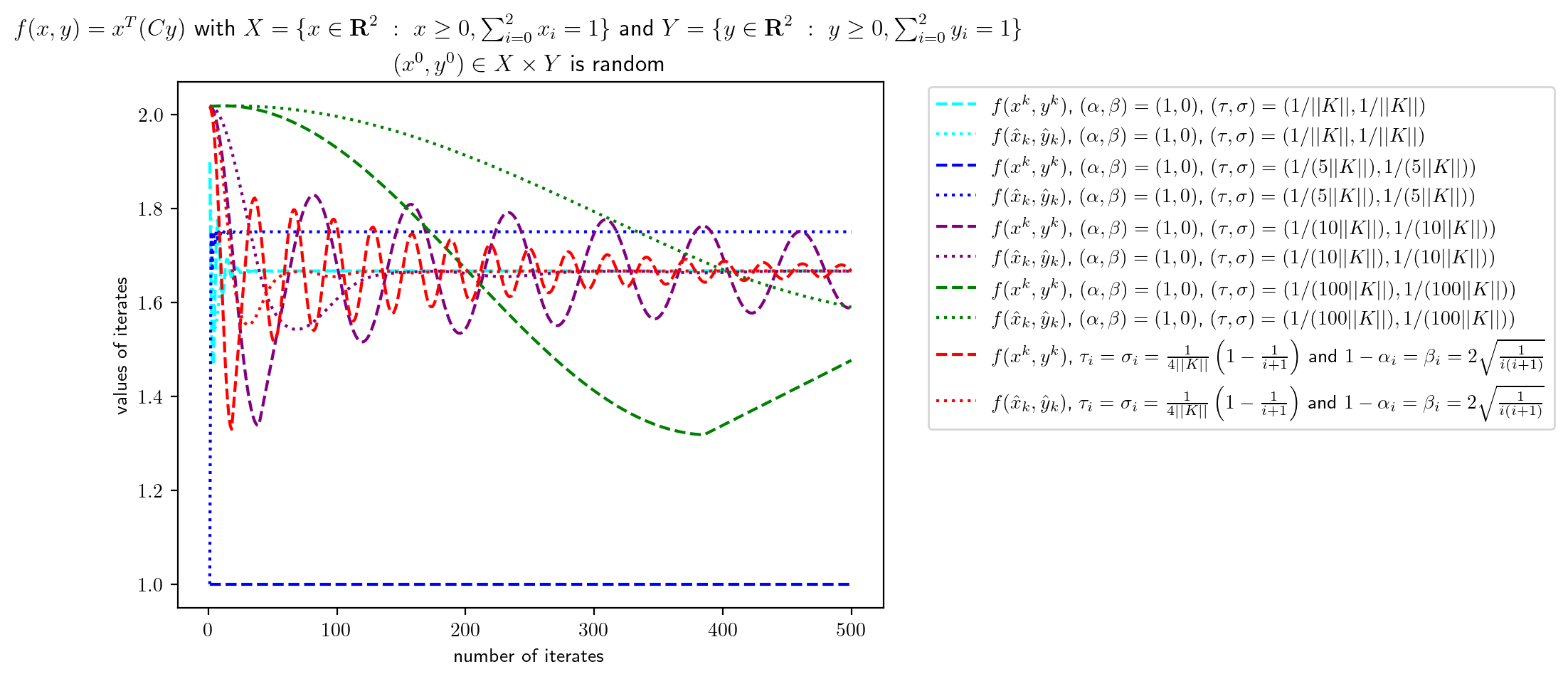}
	\caption{Comparison with different $((\alpha_{k}, \beta_{k}, \tau_{k},\sigma_{k}))_{k 
			\in \mathbf{N}}$}
	\label{fig:appmg_proximity_methods_2_00.png}
\end{figure} 

In addition, we also consider $(\forall k \in \mathbf{N})$  $(\alpha_{k}, \beta_{k}) = 
(1,0)$, $(\alpha_{k}, \beta_{k}) = (0.1,0)$, and $(\alpha_{k}, \beta_{k}) = (2,0)$. 
After choosing a random initial point  $(x^{0},y^{0}) \in X \times Y$, 
we  calculate $\left(  f(x^{k}, y^{k})  \right)_{k \in \mathbf{N}}$ and 
$\left(  f(\hat{x}_{k},\hat{y}_{k}) \right)_{k \in \mathbf{N}}$ 
for $(\forall k  \in \mathbf{N})$   $(\tau_{k},\sigma_{k} ) =  
(\frac{1}{\norm{K}},\frac{1}{\norm{K}} )$,  $(\tau_{k},\sigma_{k} ) =  (\frac{1}{ 0.1 
	\norm{K}},\frac{1}{0.1 \norm{K}} )$, and $(\tau_{k},\sigma_{k} ) =  
(\frac{1}{10 \norm{K}},\frac{1}{10 \norm{K}} )$.
The first subplot of \cref{fig:appmg_proximity_methods_alphabeta_es_00.png} 
tells us that 
when $(\forall k \in \{1, \ldots, 50 \})$  $(\tau_{k},\sigma_{k} ) =  (\frac{1}{ 0.1 
	\norm{K}},\frac{1}{0.1 \norm{K}} )$, 
neither $\left(  f(x^{k}, y^{k})  \right)_{k \in \mathbf{N}}$ nor
$\left(  f(\hat{x}_{k},\hat{y}_{k}) \right)_{k \in \mathbf{N}}$ converges to the required 
value for our three choices of $((\alpha_{k},\beta_{k}))_{k \in \mathbf{N}}$.   
This is not a big surprise because in this case 
$(\forall k \in \{1, 2, \ldots, 50\})$ $\tau_{k}\sigma_{k} \norm{K}^{2} =100 >1$, 
which doesn't satisfy the requirement of convergence in 
\cite[Theorem~1]{ChambollePock2011}.
We zoom in the third subplot of  
\cref{fig:appmg_proximity_methods_alphabeta_es_00.png} in the range of iterate 
numbers in $[1,100]$ 
and get the second subplot of 
\cref{fig:appmg_proximity_methods_alphabeta_es_00.png}.  
In view of the last two subplots of  
\cref{fig:appmg_proximity_methods_alphabeta_es_00.png}, 
we observe that when $(\forall k  \in \mathbf{N})$   $(\tau_{k},\sigma_{k} ) =  
(\frac{1}{\norm{K}},\frac{1}{\norm{K}} )$  and $(\tau_{k},\sigma_{k} ) =  
(\frac{1}{10 \norm{K}},\frac{1}{10 \norm{K}} )$, 
for our three choices of $((\alpha_{k},\beta_{k}))_{k \in \mathbf{N}}$, 
although some sequences converge very slowly, 
it seems that they will all be close enough to the desired optimal value after a large 
enough  number  of iterates. 
In addition, in view of \cref{fig:appmg_proximity_methods_alphabeta_es_00.png}, 
$\left(  f(\hat{x}_{k},\hat{y}_{k}) \right)_{k \in \mathbf{N}}$ converges faster than 
$\left(  f(x^{k}, y^{k})  \right)_{k \in \mathbf{N}}$ in all cases for this particular example,
which is the opposite situation of our examples below.
In addition, we notice that changes of $(\alpha, \beta)$
do affect the convergence rate of $\left(  f(x^{k}, y^{k})  \right)_{1 \leq k \leq 5000}$ 
and 
$\left(  f(\hat{x}_{k},\hat{y}_{k}) \right)_{1 \leq k \leq 5000}$ in 
\cref{fig:appmg_proximity_methods_alphabeta_es_00.png}, 
which is different from the situation of our examples below.

\begin{figure}[H]
	\centering
	\subfloat{{\includegraphics[width=12cm]{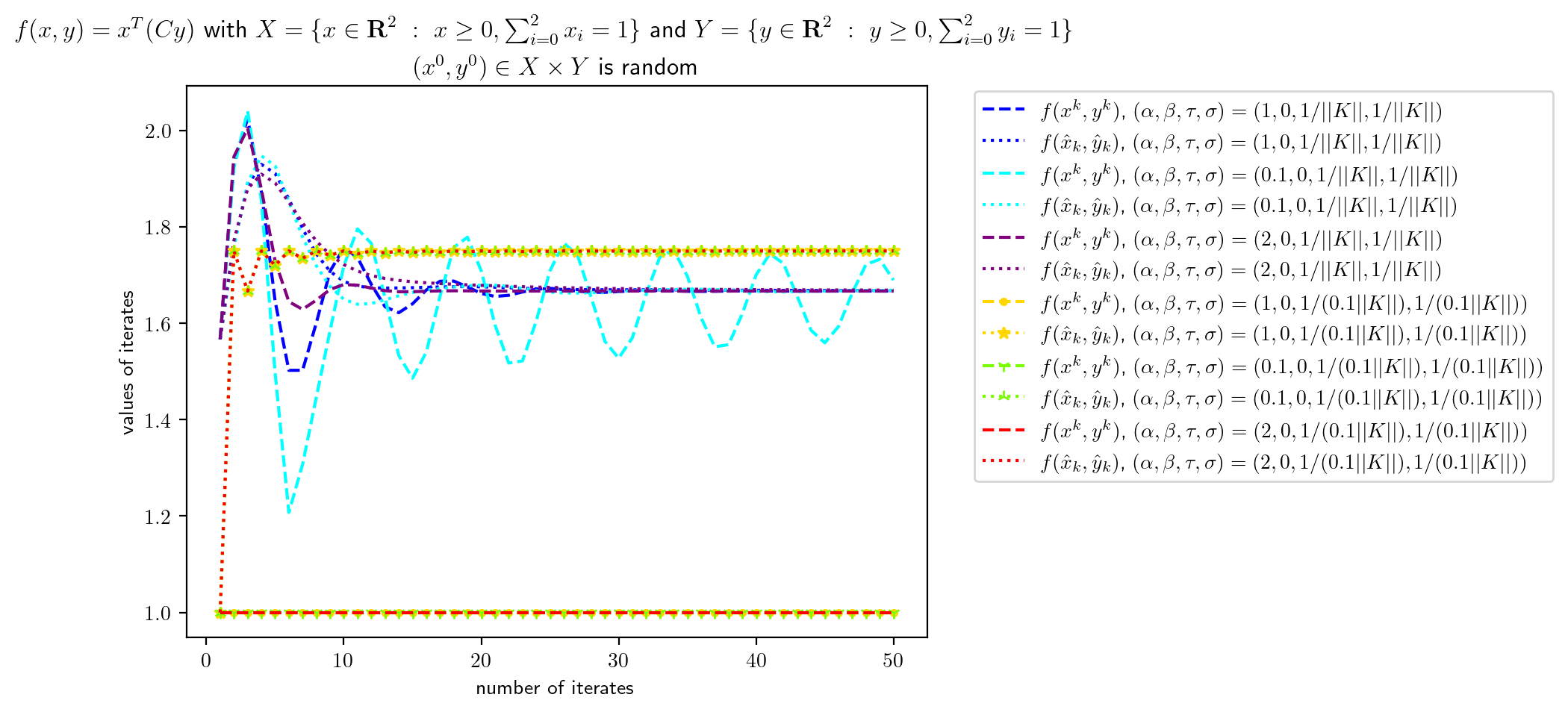}
	}}%
	\qquad
	\subfloat{{\includegraphics[width=12cm]{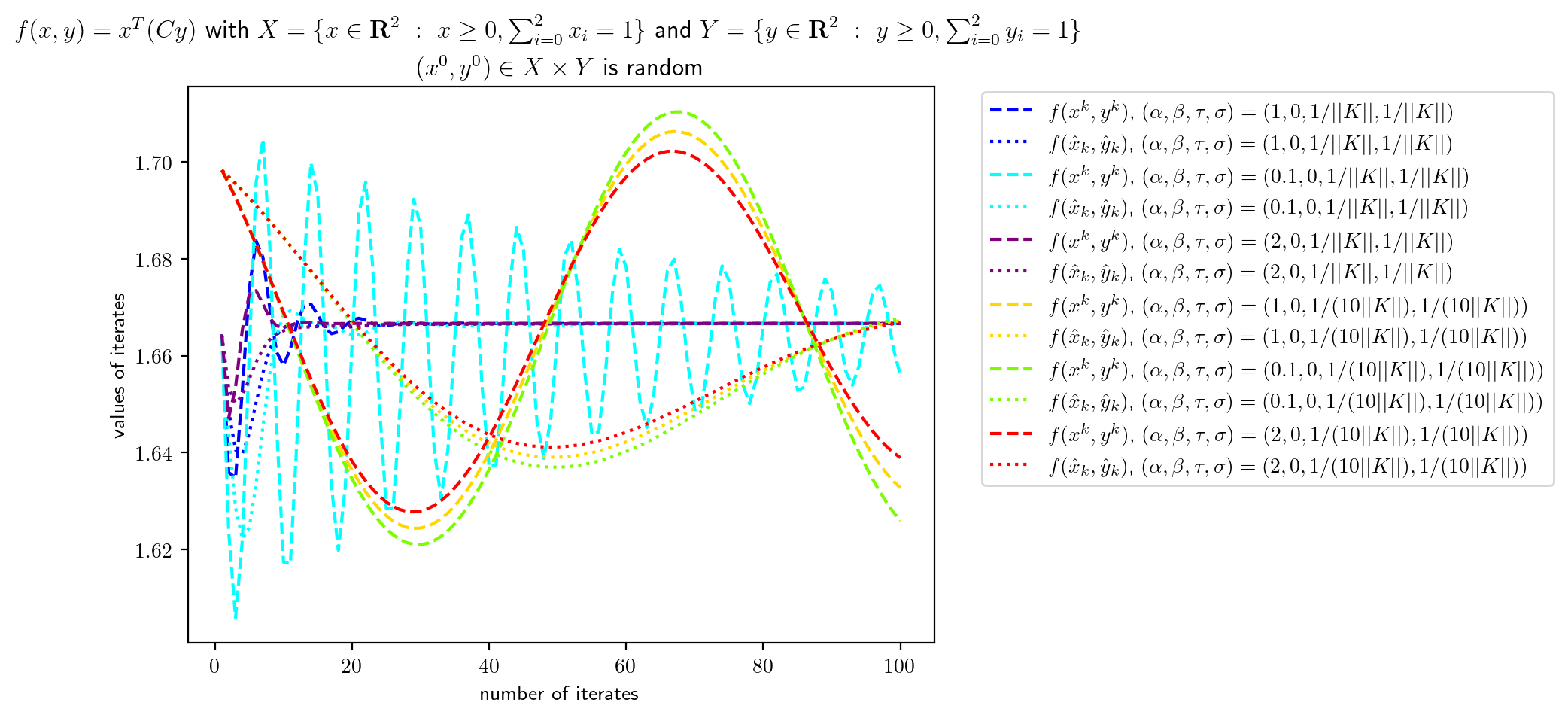}
	}}%
	\qquad
	\subfloat{{\includegraphics[width=12cm]{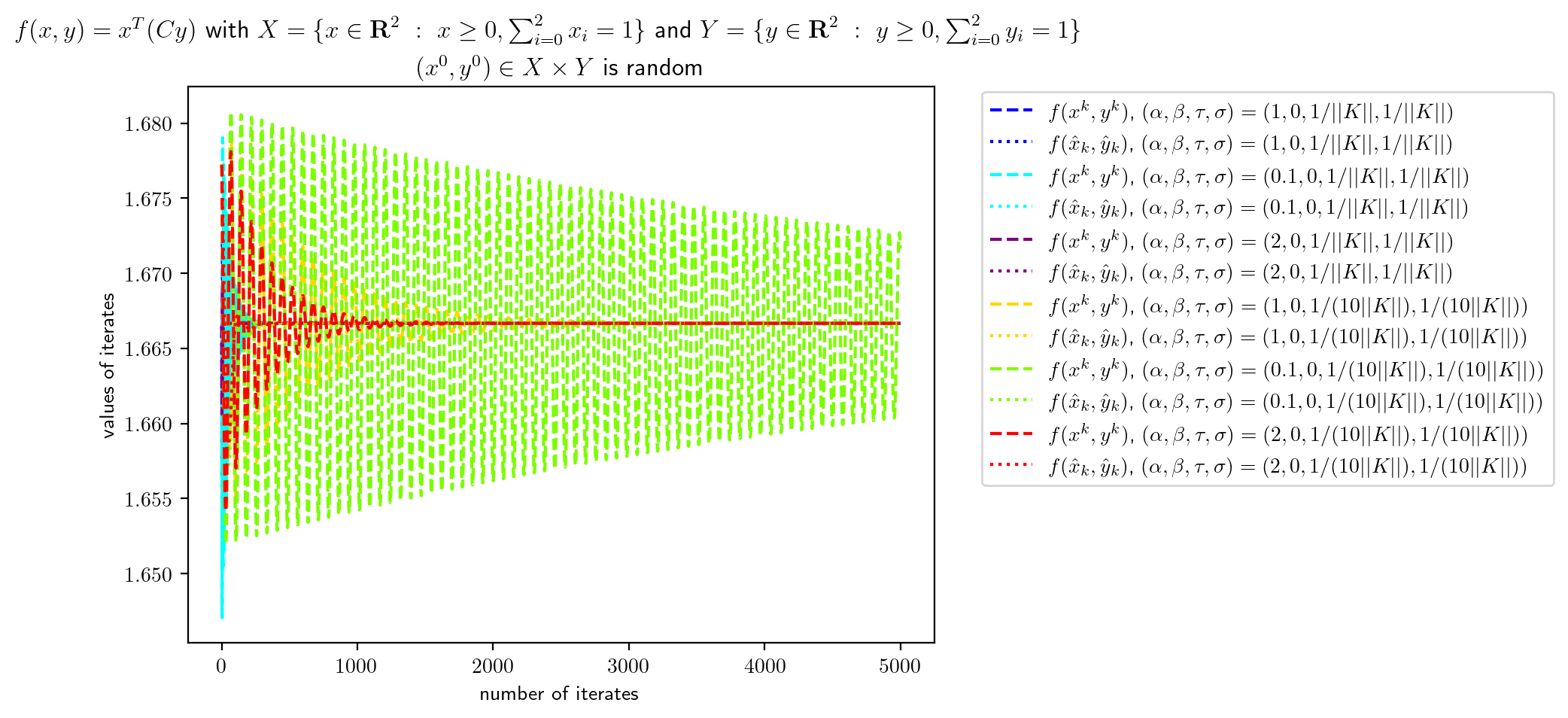}
	}}%
	\caption{Parameters $((\alpha_{k}, \beta_{k}, \tau_{k},\sigma_{k}))_{k 
			\in \mathbf{N}}$ are constants}
	\label{fig:appmg_proximity_methods_alphabeta_es_00.png}
\end{figure}

\subsection{Linear program in inequality form}
Consider $A \in \mathbf{R}^{m \times n}$,  $b  \in \mathbf{R}^{m}$,  
and $c \in \mathbf{R}^{n}$. We work on the Lagrangian  
\[ 
f : \mathbf{R}^{n}  \times \mathbf{R}^{m}_{+}  \to \mathbf{R}: 
(x,y) \mapsto  y^{T}Ax 	+c^{T}x   -b^{T}y
\]
of the inequality form LP (see, e.g., \cite[Example~1.2(ii)]{Oy2023subgradient} for 
details).

 In our experiments, after randomly choosing $A \in \mathbf{R}^{100 \times 10}$, $b \in 
\mathbf{R}^{100}$, and $c \in \mathbf{R}^{10}$, 
we employ the Python-embedded modeling language CVXPY (see 
\cite{DiamondBoyd2016} for details) 
to find bounded and feasible problems. 
Note that the value $f(x^{*}, y^{*})$ of $f$  over a saddle-point $(x^{*}, y^{*})$ equals 
to the optimal value of associated primal and dual problems. 
So  we also apply CVXPY to find the optimal solution  of  associated  primal and dual 
problems to check the correctness of our calculation.

After finding problems with optimal solutions by CVXPY, 
we randomly choose initial points $(x^{0},y^{0}) \in \mathbf{R}^{10} \times \mathbf{R}^{100}$ 
and calculate  the sequences $\left(  f(x^{k}, y^{k})  \right)_{1 \leq k \leq 1000}$ and 
$\left(  f(\hat{x}_{k},\hat{y}_{k}) \right)_{1 \leq k \leq 1000}$  
with $(\forall k \in \mathbf{N})$ 
$\tau_{k} =\sigma_{k} = \frac{1}{4 \norm{K}} \left(1 - \frac{1}{k+1}\right)$
and  $1- \alpha_{k} =\beta_{k} = 2 \sqrt{\frac{1}{k(k+1)}}$. 
In our experiments, both $\left(  f(x^{k}, y^{k})  \right)_{1 \leq k \leq 1000}$ and 
$\left(  f(\hat{x}_{k},\hat{y}_{k}) \right)_{1 \leq k \leq 1000}$
converge to the optimal value obtained from CVXPY. 
In view of our result presented in 
\cref{fig:proximal_point_algorithm_lp_nocomparison_00.png},
we confirm our theoretical results of  
\cref{theorem:convergencef,theorem:pointsconvergence}
and also observe that $\left(  f(x^{k}, y^{k})  \right)_{1 \leq k \leq 1000}$ performs better 
than $\left(  f(\hat{x}_{k},\hat{y}_{k}) \right)_{1 \leq k \leq 1000}$.
 
\begin{figure}[H]
	\centering
	\includegraphics[width=0.6\textwidth]{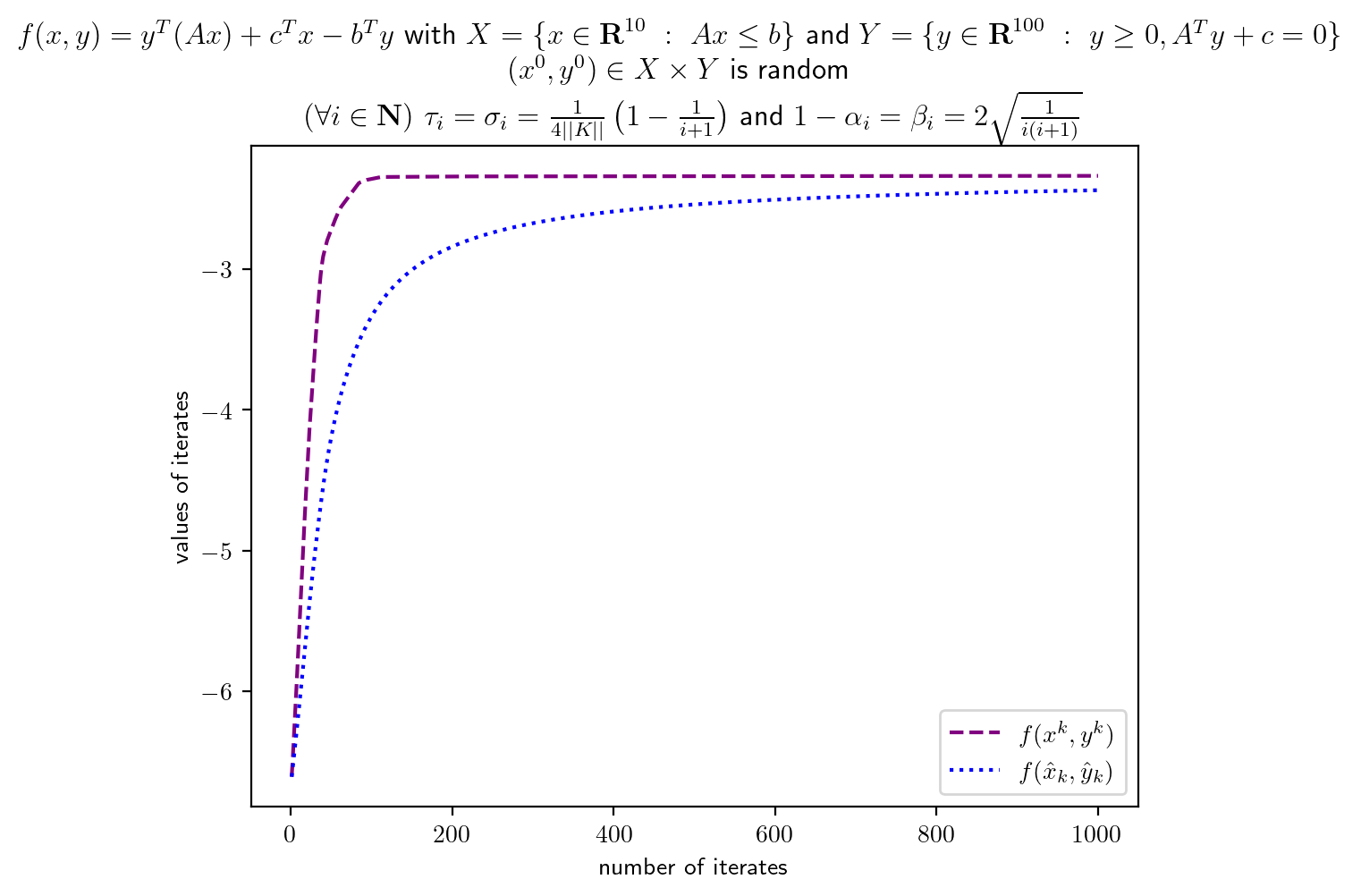}
	\caption{Convergence result with a random initial point}
	\label{fig:proximal_point_algorithm_lp_nocomparison_00.png}
\end{figure} 
 
Motivated by \cite[Theorem~1]{ChambollePock2011},
we consider $((\alpha_{k}, \beta_{k},\tau_{k}, \sigma_{k}))_{k \in \mathbf{N}}
 =\left(\left(1,0,\tfrac{1}{\eta \norm{K}}, \tfrac{1}{\eta \norm{K}}\right)\right)_{k \in 
 \mathbf{N}}$ 
 and try exploring  the influence of $\eta$  
 to the convergence rate  of sequences $\left(  f(x^{k}, y^{k})  \right)_{k \in 
 \mathbf{N}}$ and $\left(  f(\hat{x}_{k},\hat{y}_{k}) \right)_{k \in \mathbf{N}}$.
 In our related experiments, we also calculate $\left(  f(x^{k}, y^{k})  \right)_{k \in 
 \mathbf{N}}$ and $\left(  f(\hat{x}_{k},\hat{y}_{k}) \right)_{k \in \mathbf{N}}$
 with $(\forall k \in \mathbf{N})$ 
 $\tau_{k} =\sigma_{k} = \frac{1}{4 \norm{K}} \left(1 - \frac{1}{k+1}\right)$
 and  $1- \alpha_{k} =\beta_{k} = 2 \sqrt{\frac{1}{k(k+1)}}$ as a reference. 
  According to \cref{fig:proximal_point_algorithm_lp_tausigma_00.png}, 
  we observe that $\left(  f(x^{k}, 
  y^{k})  \right)_{1 \leq k \leq 1000}$ converges faster than $\left(  
  f(\hat{x}_{k},\hat{y}_{k}) \right)_{1 \leq k \leq 1000}$ in this case again 
  and that the smaller the value of $\eta$, 
  the better the convergence rate obtained by the corresponding 
  sequences of iterates. 
 Recall from \cref{subsection:matrixgame} that it is not  the case generally. 
  In practice, it is not easy to find satisfactory constants for the involved parameters.
However, from \cref{fig:proximal_point_algorithm_lp_tausigma_00.png},
  we see that our iterate scheme is a reasonable choice. 
  
  Notice that based on assumptions of 
  \cref{theorem:convergencef,theorem:pointsconvergence},  we have that when 
  $(\tau_{k})_{k \in \mathbf{N}} $ and $(\sigma_{k})_{k \in \mathbf{N}} $ are 
  sequences of constants, either $\norm{K} =0$ or $\left(( \alpha_{k}, \beta_{k}) 
  \right)_{k \in \mathbf{N}}=\left( (1,0) \right)_{k \in \mathbf{N}}$; 
  and we also require  
$\norm{K} \limsup_{i \to \infty} \tau_{i} < \frac{1}{2}$ and $\norm{K} \limsup_{i \to \infty} 
\sigma_{i} < \frac{1}{2}$. 
Furthermore, in \cite[Theorem~1]{ChambollePock2011}, 
it is required that $\tau \sigma \norm{K}^{2} <1$.
So,  to the best of our knowledge, 
 the convergence associated with $\eta < 1$   
 does not have any theoretical support as of now.

 \begin{figure}[H]
 	\centering
 	\includegraphics[width=0.8\textwidth]{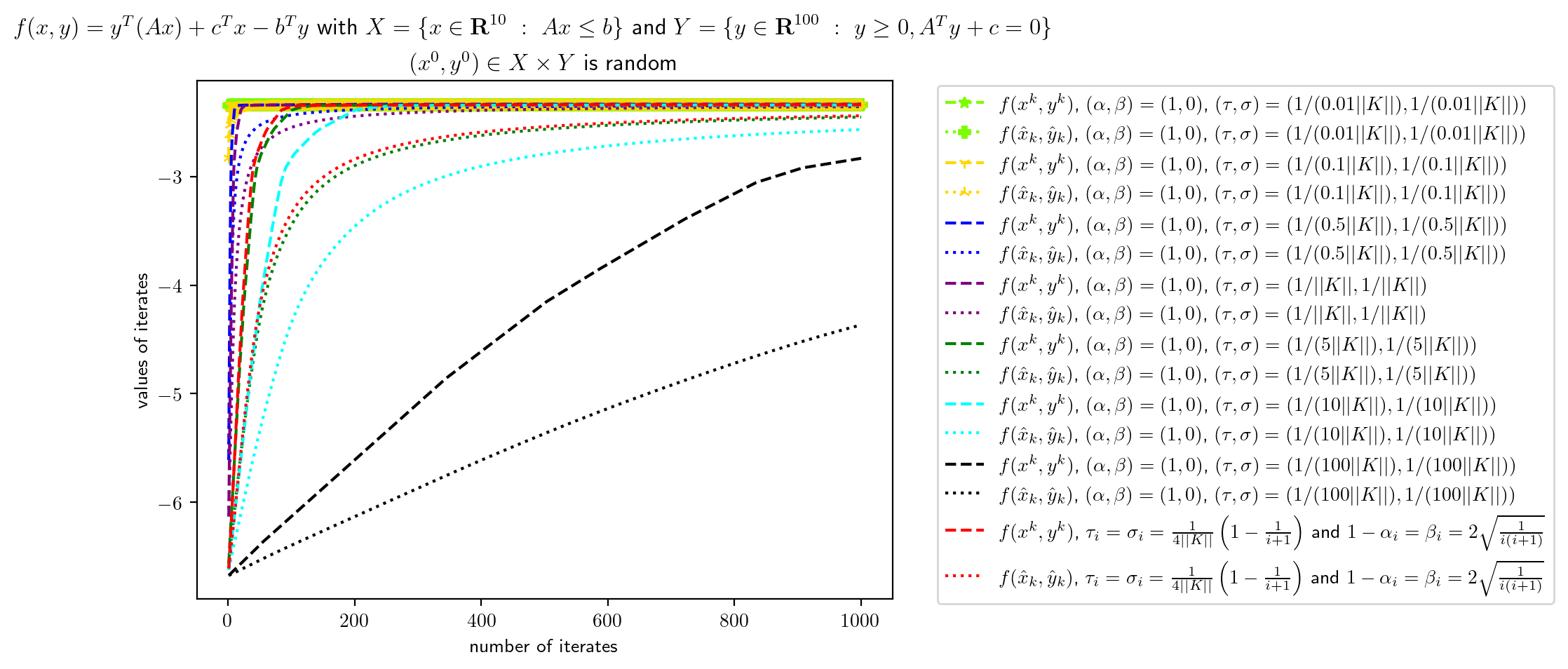}
 	\caption{Comparison of  different  
 		$(\tau_{k})_{k \in \mathbf{N}}$ and  $(\sigma_{k})_{k \in \mathbf{N}}$}
 	\label{fig:proximal_point_algorithm_lp_tausigma_00.png}
 \end{figure}

Additionally,  
we consider $((\alpha_{k}, \beta_{k},\tau_{k}, \sigma_{k}))_{k \in \mathbf{N}}
=\left( \left(\alpha, \beta,\tfrac{1}{\eta \norm{K}}, \tfrac{1}{\eta \norm{K}}\right) \right)_{k 
\in 
\mathbf{N}}$ and fix $\eta$ as $0.01$, $0.1$, $0.5$, $1$, $10$, and $100$.
For each value of $\eta$, we explore the influence of $(\alpha, \beta)$ to the  
convergence rate  of sequences
$\left(  f(x^{k}, y^{k})  \right)_{k \in 	\mathbf{N}}$ and $\left(  
f(\hat{x}_{k},\hat{y}_{k}) \right)_{k \in \mathbf{N}}$.
We present one typical result  in  
\cref{fig:proximal_point_algorithm_lp_alphabetatausigma_00.png} and 
conclude that, in this example,
  $\left(  f(x^{k}, y^{k})  \right)_{k \in \mathbf{N}}$ performs better than 
$\left(  f(\hat{x}_{k},\hat{y}_{k}) \right)_{k \in \mathbf{N}}$ 
and that the sequences $\left(  f(x^{k}, y^{k})  \right)_{k \in \mathbf{N}}$ and $\left(  
f(\hat{x}_{k},\hat{y}_{k}) \right)_{k \in \mathbf{N}}$ are independent of the change 
of $(\alpha, \beta)$.

\begin{figure}[H]
\centering
\includegraphics[width=0.6\textwidth]
{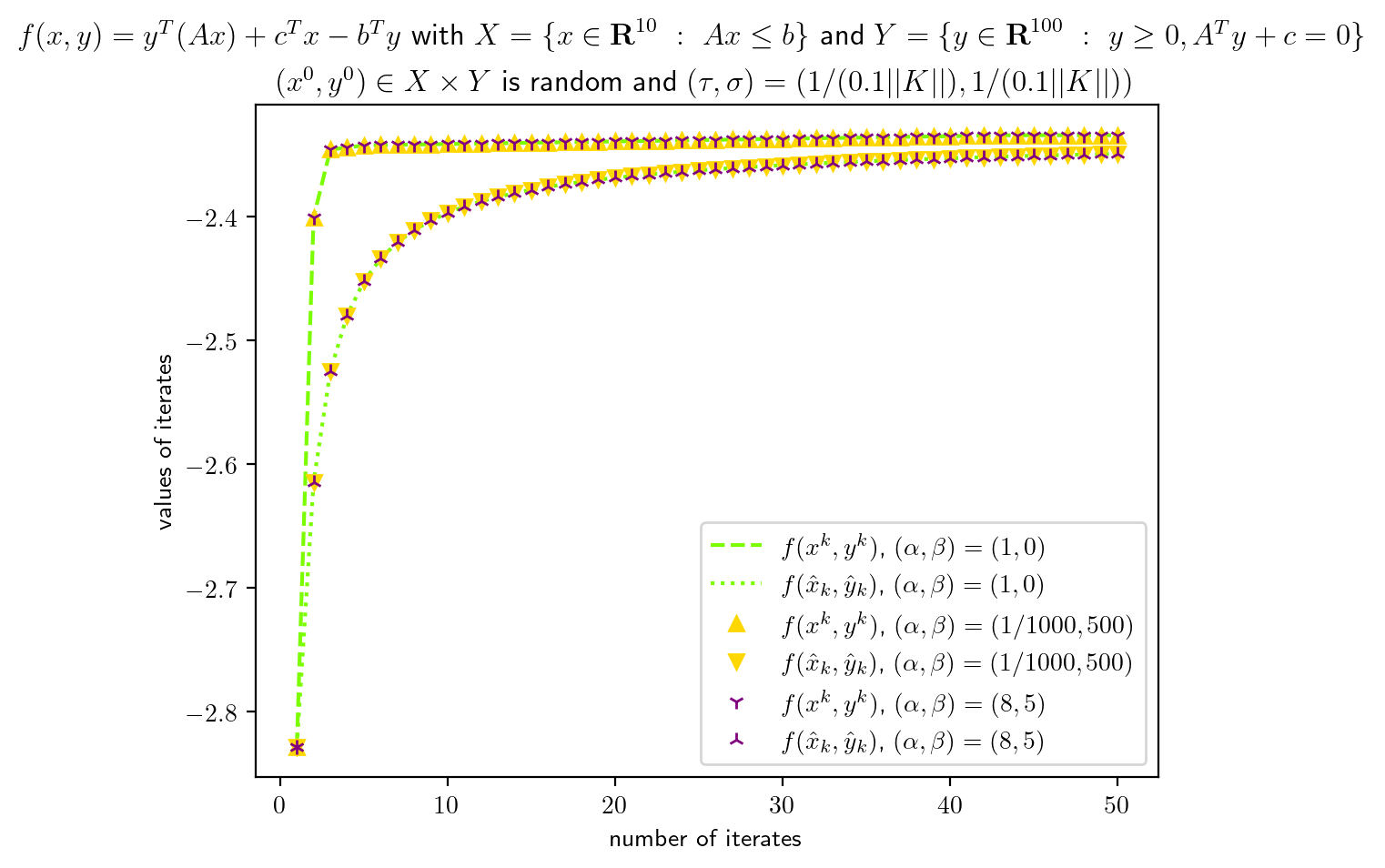}
\caption{Comparison of   different  
$(\alpha_{k})_{k \in \mathbf{N}}$ and  $(\beta_{k})_{k \in \mathbf{N}}$}
\label{fig:proximal_point_algorithm_lp_alphabetatausigma_00.png}
\end{figure}

\subsection{Least-squares problem with \texorpdfstring{$\ell_{1}$}{l} 
regularization} 
\label{subsubsection:example:lsl1}
Let $A \in \mathbf{R}^{m \times n}$, $b  \in \mathbf{R}^{m}$, and $\gamma \in 
\mathbf{R}_{++}$.  
Below we consider the Lagrangian  
\[ 
f: \mathbf{R}^{n+m} \times \mathbf{R}^{m} \to \mathbf{R}
: ((x,u),y) \mapsto
  f((x,u),y) = \frac{1}{2} \norm{u}^{2}_{2} + \gamma \norm{x}_{1} + y^{T}( Ax -b -u)
\] 
of a least-squares problem with $\ell_{1}$ regularization (see, e.g., 
\cite[Example~1.2(iii)]{Oy2023subgradient} for details). 

In our experiments, we set $\gamma =1$,
$X =\mathbf{R}^{100+50}$, and $Y = \{ y \in \mathbf{R}^{100} ~:~ 
\norm{A^{T}y}_{\infty} \leq \gamma  \}$.
Then we randomly choose $A \in \mathbf{R}^{100 \times 50}$ and $b \in 
\mathbf{R}^{100}$.
Note that in this case the problem is always feasible and bounded. 
We apply CVXPY  to solve related primal and dual problems as a reference later. 
We randomly choose $6$ initial points $(x^{0},y^{0}) \in X  \times Y$ and,
for each initial point, we 
calculate  the sequences $\left(  f(x^{k}, y^{k})  \right)_{1 \leq k \leq 2000}$ and 
$\left(  f(\hat{x}_{k},\hat{y}_{k}) \right)_{1 \leq k \leq 2000}$ with $(\forall k \in 
\mathbf{N})$ $\tau_{k} =\sigma_{k} = \frac{1}{4 \norm{K}} \left(1 - 
\frac{1}{k+1}\right)$  and $1- \alpha_{k} =\beta_{k} = 2 \sqrt{\frac{1}{k(k+1)}}$. 
In our experiments, both $\left(  f(x^{k}, y^{k})  \right)_{k \in \mathbf{N}}$ and 
$\left(  f(\hat{x}_{k},\hat{y}_{k}) \right)_{k \in \mathbf{N}}$
converge to the optimal value obtained from CVXPY after a large enough number of
iterates. 
From \cref{fig:proximal_point_algorithm_lsl1_nocomparison_00.png},
we see that $\left(  f(x^{k}, y^{k})  \right)_{1 \leq k \leq 2000}$ beats
$\left(  f(\hat{x}_{k},\hat{y}_{k}) \right)_{1 \leq k \leq 2000}$ in this example as well. 

\begin{figure}[H]
\centering
\includegraphics[width=0.6\textwidth]
{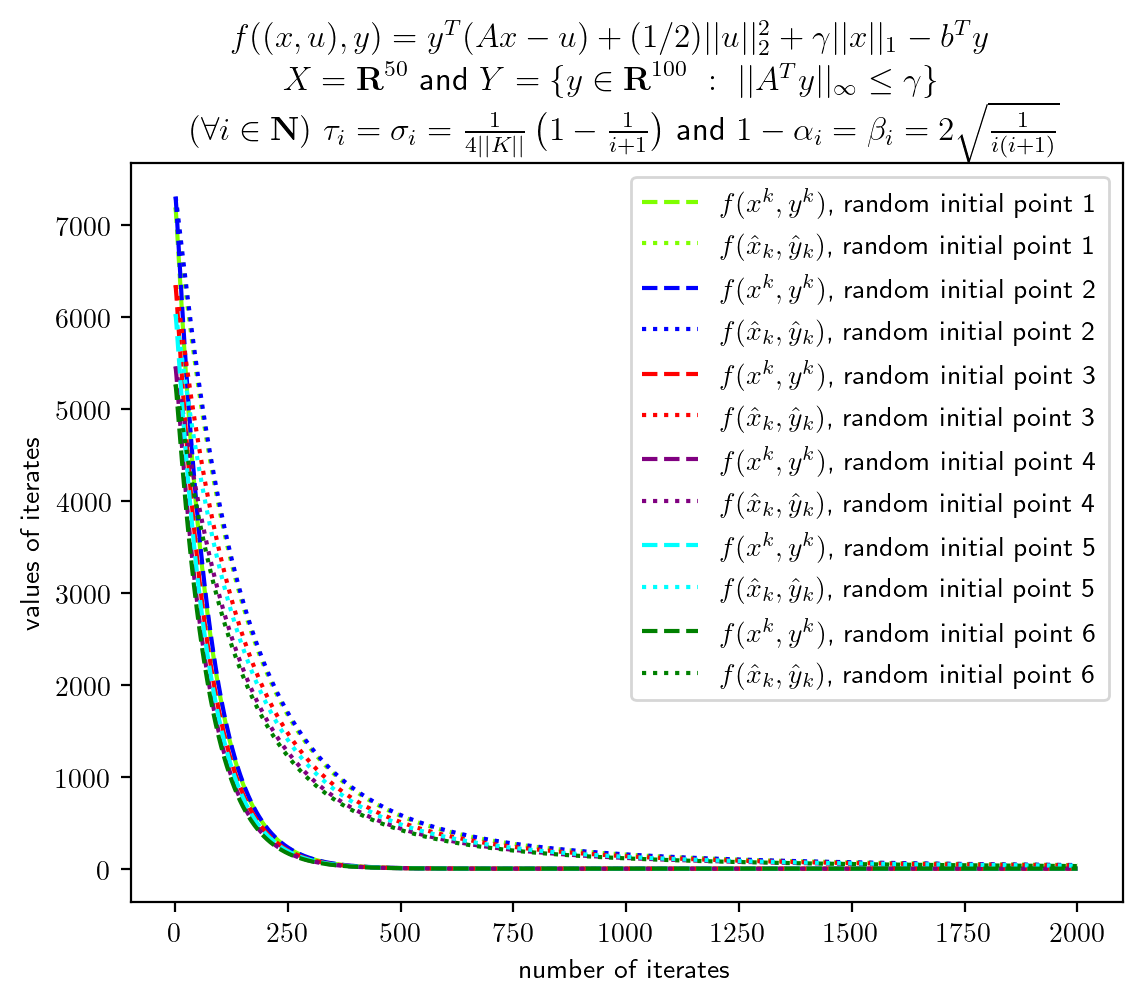}
\caption{Convergence  with   random initial points}
\label{fig:proximal_point_algorithm_lsl1_nocomparison_00.png}
\end{figure}

For this example, 
we also consider $((\alpha_{k}, \beta_{k},\tau_{k}, \sigma_{k}))_{k \in \mathbf{N}}
=\left(1,0,\tfrac{1}{\eta \norm{K}}, \tfrac{1}{\eta \norm{K}}\right)_{k \in \mathbf{N}}$ 
and explore  the influence of $\eta$  
to the convergence rate of sequences $\left(  f(x^{k}, y^{k})  \right)_{k \in 
\mathbf{N}}$ and $\left(  f(\hat{x}_{k},\hat{y}_{k}) \right)_{k \in \mathbf{N}}$.
In our related experiments, we also calculate $\left(  f(x^{k}, y^{k})  \right)_{k \in 
\mathbf{N}}$ and $\left(  f(\hat{x}_{k},\hat{y}_{k}) \right)_{k \in \mathbf{N}}$
with $(\forall k \in \mathbf{N})$ 
$\tau_{k} =\sigma_{k} = \frac{1}{4 \norm{K}} \left(1 - \frac{1}{k+1}\right)$
and  $1- \alpha_{k} =\beta_{k} = 2 \sqrt{\frac{1}{k(k+1)}}$ as a reference. 
According to \cref{fig:proximal_point_algorithm_lsl1_tausigma_00.png}, 
we observe that $\left(  f(x^{k}, y^{k})  \right)_{1 \leq k \leq 1000}$ 
converges faster than $\left(  
f(\hat{x}_{k},\hat{y}_{k}) \right)_{1 \leq k \leq 1000}$ 
and that as $\eta$  decreases, 
the convergence rate of the corresponding sequences of iterates improves.
In this example, we see that our iterate scheme with non-constant involved parameters
is also a favorable choice when  the optimal $\eta$ is unknown.
\begin{figure}[H]
	\centering
	\includegraphics[width=0.8\textwidth]{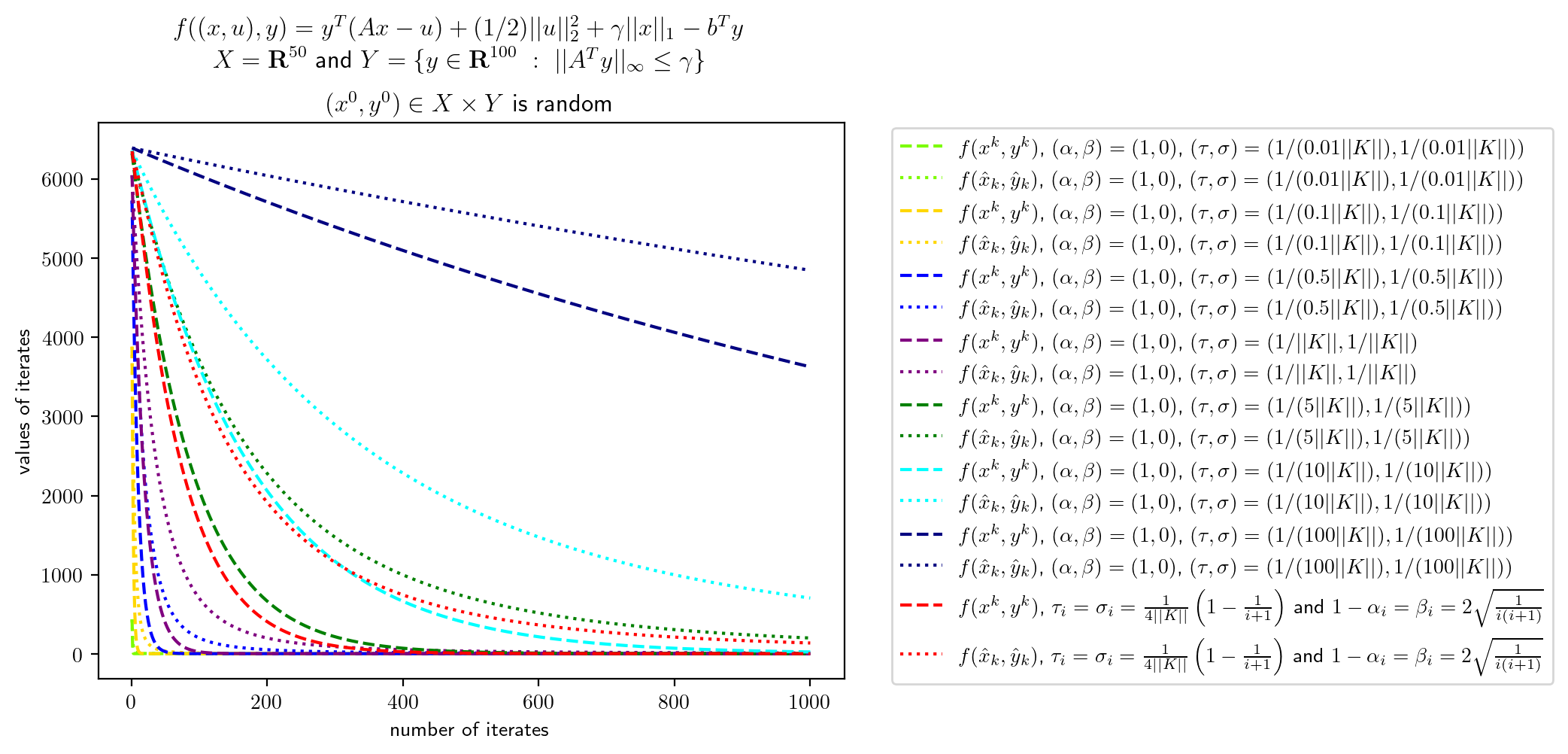}
	\caption{Convergence  with different  $((\alpha_{k}, \beta_{k},\tau_{k}, 
	\sigma_{k}))_{k \in \mathbf{N}}$}
	\label{fig:proximal_point_algorithm_lsl1_tausigma_00.png}
\end{figure}

In view of \cref{fig:proximal_point_algorithm_lsl1_alphabetatausigma_00.png}, 
when we fix  $(\tau, \sigma) = (\tfrac{1}{0.01\norm{K}},\tfrac{1}{0.01\norm{K}})$, 
neither $\left(  f(x^{k}, y^{k})  \right)_{1 \leq k \leq 50}$ nor $\left(  
f(\hat{x}_{k},\hat{y}_{k}) \right)_{1 \leq k \leq 50}$ is affected by the change of 
$(\alpha, \beta)$. 
(We also set $(\tau, \sigma) = 
(\tfrac{1}{\norm{K}},\tfrac{1}{\norm{K}})$,  $(\tau, \sigma) = 
(\tfrac{1}{0.1\norm{K}},\tfrac{1}{0.1\norm{K}})$,  
and other values, but,  in none of our experiments, the change of $(\alpha, \beta)$ 
affects 
$\left(  f(x^{k}, y^{k})  \right)_{1 \leq k \leq 50}$ or $\left(  f(\hat{x}_{k},\hat{y}_{k}) 
\right)_{1 \leq k \leq 50}$). 
In addition, $\left(  f(x^{k}, y^{k})  \right)_{k \in \mathbf{N}}$ performs better than $\left(  
f(\hat{x}_{k},\hat{y}_{k}) \right)_{k \in \mathbf{N}}$ 
in all of our experiments for this example as well. 

\begin{figure}[H]
	\centering
	\includegraphics[width=0.6\textwidth]
	{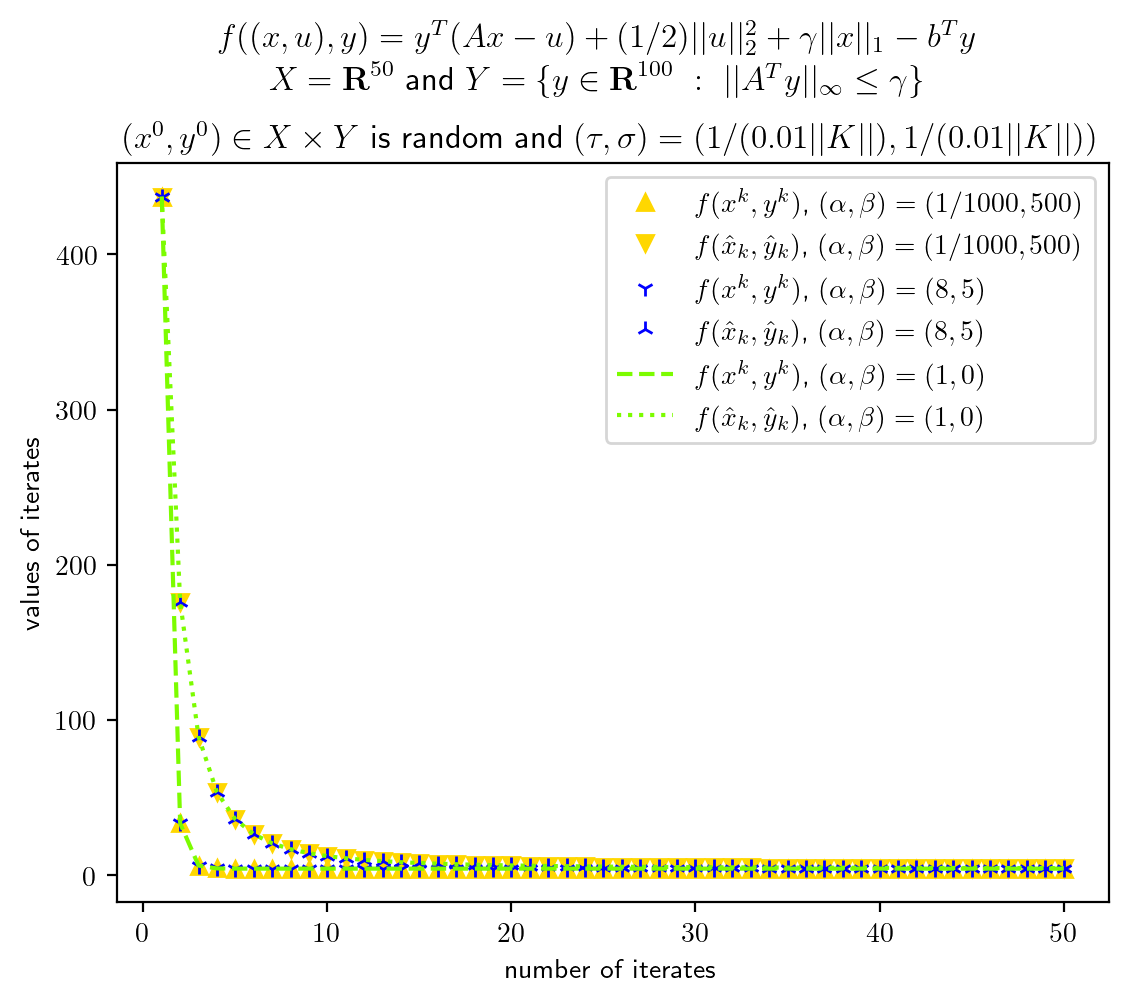}
	\caption{Convergence  with different $(\alpha, \beta)$}
	\label{fig:proximal_point_algorithm_lsl1_alphabetatausigma_00.png}
\end{figure}

 \section{Conclusion}\label{section:Conclusion}
 In this work, we introduced an iterate scheme to solve convex-concave saddle-point
 problems associated with a general convex-concave function and call it 
 alternating proximal mapping method. 
 We worked specifically on convex-concave saddle-point problems 
associated with a convex-concave function 
$f: X \times Y \to \mathbf{R} \cup \{- \infty, +\infty\}$ defined as
\[ 
(\forall (x,y) \in X \times Y) \quad 	f(x,y) =\innp{Kx, y} + g(x)  -h(y),
\] 
where $K \in \mathcal{B}(\mathcal{H}_{1} , \mathcal{H}_{2} )$, and 
$g: \mathcal{H}_{1} \to \mathbf{R} \cup \{ +\infty\}$ 
and $h: \mathcal{H}_{2} \to  \mathbf{R} \cup \{ +\infty\}$  
are proper, convex, and lower semicontinuous. 
Suppose that  $((x^{k}, y^{k}))_{k \in \mathbf{N}}$ is the iteration sequence 
generated by our iterate scheme. 
Set $(\forall k \in \mathbf{N})$ $\hat{x}_{k} := \frac{1}{k+1} \sum^{k}_{i=0} x^{i+1} $ 
and $\hat{y}_{k} := \frac{1}{k+1} \sum^{k}_{i=0} y^{i+1} $.  
Let  $(x^{*}, y^{*}) \in X \times Y$ be a saddle-point of the function $f$.
We demonstrated the convergence of $f(\hat{x}_{k}, \hat{y}_{k}) \to f(x^{*}, y^{*})  $
under some restrictions on parameters used in our iterate scheme. 
Under similar assumptions, 
we showed that every weakly sequential cluster point  of 
$\left((\hat{x}_{k},\hat{y}_{k}) \right)_{k \in \mathbf{N}}$  is a saddle-point of $f$,
and that 
$\left((x^{k},y^{k}) \right)_{k \in \mathbf{N}}$ converges to a 
saddle-point of $f$. 
 
We applied our alternating proximal mapping method
to a matrix game,
a linear program in inequality form,
and a least-squares problem with $\ell_{1}$ regularization.
We also compared our algorithm with some other primal-dual algorithms
in which involved parameters are constants.
Our numerical experiments not only validated our theoretical results, 
but they also  discovered some new and interesting results.
We founded that when we consider a fixed problem in our numerical experiments,
it is either $\left(f(\hat{x}_{k},\hat{y}_{k}) \right)_{k \in \mathbf{N}}$  or 
$\left(f(x^{k},y^{k}) \right)_{k \in \mathbf{N}}$ converges faster in every case 
(no matter how do we change initial points, problem data, and iterate schemes);
in some examples, changes of some parameters don't affect the convergence rate 
of $\left(f(\hat{x}_{k},\hat{y}_{k}) \right)_{k \in \mathbf{N}}$  or 
$\left(f(x^{k},y^{k}) \right)_{k \in \mathbf{N}}$;
in practice, when we are unsure of how to determine suitable constants for
parameters in iterate schemes, 
our algorithm is consistently a reasonable choice. 
 
  \section*{Acknowledgments}
  Hui Ouyang thanks Professor Boyd Stephen for introducing the author to the field of 
  saddle-point problem, and for his  insight and expertise comments 
   on the topic of saddle-point problems and all unselfish support.
 Hui Ouyang acknowledges 
 the Natural Sciences and Engineering Research Council of Canada (NSERC), 
 [funding reference number PDF – 567644 – 2022]. 
 
 
 \addcontentsline{toc}{section}{References}
 \bibliographystyle{abbrv}
 \bibliography{ccspp_proximity}

 \appendix
\section{Citations of the Arrow-Hurwicz Algorithm}
 We explain below the relationship between the Arrow-Hurwicz method and the  
 primal-dual proximal-point method presented in \cref{eq:pdppm}. 
 
 Let $A: \mathcal{H}_{2} \to \mathcal{H}_{1}$ be linear and bounded,  and let $z \in 
 \mathcal{H}_{1}$ and $\lambda \in \mathbf{R}_{++}$. Consider the special 
 convex-concave function $f:  
 \mathcal{H}_{1} \times \mathcal{H}_{2} \to \mathbf{R}$ defined as 
 \[ 
 (\forall (x,y) \in \mathcal{H}_{1} \times \mathcal{H}_{2}) 
 \quad f(x,y) =\innp{x, Ay} + \frac{\lambda}{2} \norm{x-z}^{2},
 \]
 which is a convex-concave function considered in \cite{ZhuChan2008efficient}.
 Let $(\bar{x},\bar{y}) \in \mathcal{H}_{1} \times \mathcal{H}_{2} $. 
 Note that
 \begin{subequations}
 	\begin{align}
 		&\nabla_{x} f(\bar{x},\bar{y}) =  A\bar{y} +\lambda(\bar{x}-z); \label{eq:nablax}\\
 		&\nabla_{y} f(\bar{x},\bar{y}) = A^{*}\bar{x}. \label{eq:nablay}
 	\end{align} 
 \end{subequations}
 Furthermore, in this case we have that
 \begin{align}  \label{eq:fnabla:y}
 	(\forall y \in \mathcal{H}_{2})	 (\forall w \in \mathcal{H}_{2}) \quad f(\bar{x},y) - 
 	f(\bar{x},\bar{y}) = \innp{\bar{x},Ay} -\innp{\bar{x}, A\bar{y} }= 
 	\innp{A^{*}\bar{x}, y-\bar{y}} = \innp{\nabla_{y} f(\bar{x},w), y-\bar{y} },
 \end{align}
 and that 
 \begin{align}    \label{eq:fnabla:x}
 	(\forall x \in \mathcal{H}_{1})	(\forall w \in \mathcal{H}_{2})  \quad \nabla_{x} f(x,w) 
 	-\nabla_{x} f(\bar{x},w)
 	=  \lambda (x-\bar{x}). 
 \end{align}
 
 Let $\alpha$ and $\beta$ be in $\mathbf{R}_{++}$,  let $k \in \mathbf{N}$, and let $Y$ 
 be a nonempty, closed,  and convex subset of  $\mathcal{H}_{2}$. 
 Based on \cref{fact:proxsubdiff} with $X=\mathcal{H}_{1}$ and \cref{eq:nablax}, 
 \begin{align*}
 	&x^{k+1} = \argmin_{x \in \mathcal{H}_{1}} \{ f(x, y^{k+1}) + \frac{1}{2 \alpha} 
 	\norm{x 
 		-x^{k}}^{2} \} \\
 	\Leftrightarrow & 0 =\nabla_{x} f(x^{k+1} ,y^{k+1})  + \frac{1}{  \alpha} (x^{k+1} 
 	-x^{k})\\
 	\Leftrightarrow & 0 = \nabla_{x} f(x^{k} ,y^{k+1})   + (\nabla_{x} f(x^{k+1} ,y^{k+1})   
 	-\nabla_{x} f(x^{k} ,y^{k+1})   )+ \frac{1}{  \alpha} (x^{k+1} -x^{k})\\
 	\stackrel{\cref{eq:fnabla:x}}{\Leftrightarrow}  & 0 = \nabla_{x} f(x^{k} ,y^{k+1})   +
 	\lambda (x^{k+1}-x^{k}) + \frac{1}{  \alpha} (x^{k+1} -x^{k})\\
 	\Leftrightarrow &  (1+\alpha \lambda ) x^{k+1} =(1+\alpha \lambda) x^{k} - \alpha 
 	\nabla_{x} f(x^{k} ,y^{k+1})  \\
 	\Leftrightarrow & x^{k+1} =  x^{k} - \frac{ \alpha }{1+\alpha \lambda } \nabla_{x} 
 	f(x^{k} 
 	,y^{k+1})  \\
 	\Leftrightarrow & x^{k+1}= \Pro_{\mathcal{H}_{1}}( x^{k} - \frac{ \alpha }{1+\alpha 
 		\lambda }  \nabla_{x} 
 	f(x^{k},y^{k+1})   ).
 \end{align*}
 In addition, in view of \cref{fact:proxsubdiff},
 \begin{align*}
 	&y^{k+1} = \argmax_{y  \in Y} \{ f(x^{k}, y) - \frac{1}{2 \beta} \norm{y
 		-y^{k}}^{2} \} = \argmin_{y  \in Y} \{ -\beta f(x^{k}, y) + \frac{1}{2 } \norm{y
 		-y^{k}}^{2} \} \\
 	\Leftrightarrow & (\forall y \in Y) \quad \innp{y^{k} -y^{k+1}, y-y^{k+1}}  - \beta 
 	f(x^{k}, 	y^{k+1})  \leq - \beta f(x^{k}, y) \\
 	\Leftrightarrow &  (\forall y \in Y) \quad \innp{y^{k} -y^{k+1}, y-y^{k+1}}   \leq \beta 
 	f(x^{k}, 	y^{k+1}) - \beta f(x^{k}, y) \\
 	\stackrel{\cref{eq:fnabla:y}}{\Leftrightarrow} &(\forall y \in Y) \quad \innp{y^{k} 
 		-y^{k+1}, y-y^{k+1}}   \leq \beta \innp{\nabla_{y} f(x^{k},y^{k}), y^{k+1}-y }\\
 	\Leftrightarrow &(\forall y \in Y) \quad  \innp{y^{k} + \beta \nabla_{y} f(x^{k} ,y^{k}) 
 		-y^{k+1}, y-y^{k+1}}   \leq   0\\
 	\Leftrightarrow & y^{k+1} =\Pro_{Y} (y^{k} + \beta \nabla_{y} f(x^{k} ,y^{k}) ),
 \end{align*}
 where in the last equivalence we use the projection theorem.

 In the following result, we consider transforming proximity mappings   to 
 expressions with (projected) subgradients, which generalizes results obtained  above. 
 \begin{lemma}  \label{lemma:pmsm}
 	Let $g:  \mathcal{H}    \to \mathbf{R} \cup \{ -\infty \}$ 
 	be proper, convex, and differentiable. 
 	Let $\alpha $   be in $\mathbf{R}_{++}$. 
 	Let $C \subseteq  \mathcal{H}  $ be nonempty, closed, and 	convex. 
 	Consider the \emph{proximity mapping of $\alpha g$ associated with $C$}
 	\[ 
 	(\forall \bar{x} \in \mathcal{H}) \quad 	\Prox_{\alpha g}^{C} \bar{x} : = \argmin_{x 
 		\in C} 
 	\left\{  g(x) +\frac{1}{2 \alpha}\norm{x-\bar{x}}^{2} \right\}.
 	\]
 	$($Clearly, $\Prox_{\alpha g}^{\mathcal{H}} =\Prox_{\alpha g}  $.$)$ Let $\bar{x}$ 
 	and 	$p$ 	be in $ \mathcal{H}$. 
 	Then we have the following results. 
 	\begin{enumerate}
 		\item  \label{lemma:pmsm:H} 
 		Suppose that $C = \mathcal{H}$ and that $(\forall x 
 		\in 	\mathcal{H})$ $\nabla 	g(x) - \nabla g(\bar{x}) = \lambda (x-\bar{x})$. 
 		$($For example, $(\forall x \in \mathcal{H})$ $g(x)=\innp{x, Ay} + \frac{\lambda}{2} 
 		\norm{x-z}^{2}$, where $A: \mathcal{K} \to 	\mathcal{H}$, $\mathcal{K}$ is a 
 		Hilbert 	space, $y$ is in $\mathcal{K}$, and $z$ is in  $ \mathcal{H}$.$)$
 		Then 
 		\[   
 		p = \Prox^{C}_{\alpha g} \bar{x} \Leftrightarrow p =\bar{x} - 
 		\frac{\alpha}{1+\alpha \lambda}\nabla g(\bar{x}),
 		\]
 		that is, 
 		\[ 
 		p= \argmin_{x \in C} \left\{  g(x) +\frac{1}{2 \alpha}\norm{x-\bar{x}}^{2} \right\}
 		\Leftrightarrow p= \Pro_{C} \left( \bar{x} - \frac{\alpha}{1+ \alpha \lambda}\nabla 
 		g(\bar{x}) \right).
 		\]
 		\item \label{lemma:pmsm:C} Suppose that
 		\begin{align}   \label{gx-gp}
 			(\forall x \in C)  \quad g(x)- g(p) = \innp{\nabla g(\bar{x}),x-p}.
 		\end{align}
 		$($For example, $(\forall x \in \mathcal{H})$  $g(x)=\innp{y, Ax} + \frac{\lambda}{2} 
 		\norm{y-z}^{2}$, 
 		where  $A: \mathcal{H} \to \mathcal{K}$, $\mathcal{K}$ is a Hilbert 
 		space, $y$ is in $\mathcal{K}$, and $z$ is in  $ \mathcal{H}$.$)$
 		Then 
 		\[
 		p = \Prox^{C}_{\alpha g} \bar{x} \Leftrightarrow p =\Pro_{C} \left( \bar{x} - \alpha 
 		\nabla g(\bar{x})  \right).
 		\]
 		that is, 
 		\[ 
 		p= \argmin_{x \in C} \left\{  g(x) +\frac{1}{2 \alpha}\norm{x-\bar{x}}^{2} \right\}
 		\Leftrightarrow p=\Pro_{C} \left( \bar{x} - \alpha \nabla g(\bar{x})  \right).
 		\]
 	\end{enumerate}
 \end{lemma}
 
 \begin{proof}
 	\cref{lemma:pmsm:H}:  According to Fermat's rule or \cref{fact:proxsubdiff}, in this 
 	case,
 	\begin{align*}
 		&p = \Prox^{C}_{\alpha g} \bar{x} \\
 		\Leftrightarrow &0 = \nabla g(p) +\frac{1}{ \alpha}  (p - \bar{x}) \\
 		\Leftrightarrow &0 =\alpha \nabla g( \bar{x}) + \alpha (\nabla g(p) - \nabla g( 
 		\bar{x}) ) +  
 		(p -	\bar{x}) \\
 		\Leftrightarrow &0 =\alpha \nabla g( \bar{x}) + \alpha \lambda (p -  \bar{x} ) +  
 		(p -	\bar{x}) \\
 		\Leftrightarrow & (1+\alpha \lambda ) p = (1+\alpha \lambda )\bar{x} - \alpha \nabla 
 		g( 
 		\bar{x})\\
 		\Leftrightarrow & p= \bar{x} -\frac{\alpha}{1+\alpha \lambda} \nabla g( \bar{x}).
 	\end{align*}
 	\cref{lemma:pmsm:C}: In view of \cref{fact:proxsubdiff}, 
 	\begin{align*}
 		& p = \argmin_{x \in C}  \{ g(x) + \frac{1}{2 \alpha } \norm{x -\bar{x}}^{2} \}\\
 		\Leftrightarrow &  (\forall x \in C) \quad \innp{\bar{x}-p, x-p}  \leq  \alpha g(x) -  
 		\alpha 
 		g(p)\\
 		\stackrel{\cref{gx-gp}}{\Leftrightarrow} & (\forall x \in C) \quad \innp{\bar{x}-p, x-p}  
 		\leq  \alpha \innp{ \nabla g(\bar{x}),x-p}\\
 		\Leftrightarrow & (\forall x \in C) \quad \innp{\bar{x} - \alpha \nabla 
 			g(\bar{x})-p,x-p}  	
 		\leq 	0\\
 		\Leftrightarrow &  p= \Pro_{C} \left( \bar{x} - \alpha \nabla g(\bar{x})  \right),
 	\end{align*}
 	where, in the last equivalence, we employ the projection theorem. 
 \end{proof}

 When papers cited the Arrow-Hurwicz method, they always cited the book
 \cite{ArrowHurwiczUzawa1958}.
 Although the Arrow-Hurwicz method was formulated in differential equations on 
 \cite[Page~118]{ArrowHurwiczUzawa1958} (it is actually referred to as the 
 Arrow-Hurwicz differential equation on \cite[Page~245]{Rockafellar1970}), 
 as it is described in \cite[Section~4]{RockafellarSaddlePoints1971} and 
 \cite{Popov1980modification}, the idea of the Arrow-Hurwicz method is essentially an 
 alternating projected subgradient method when the corresponding convex-concave 
 function has some differentiable properties and its domain satisfies certain 
 requirements. 
 It is written on  \cite[Page~245]{Rockafellar1970} that in certain cases where 
 $X=Y\times Z$ is finite-dimensional and $T$ is the monotone operator associated 
 with a saddle-function $K$ of the form 
 \begin{align*}
 	K(y,z) =
 	\begin{cases}
 		L(y,z) \quad \text{if } y \in C \text{ and } z \in D,\\
 		+\infty \quad \text{if } y \in C \text{ and } z \notin D,\\
 		-\infty \quad \text{if } y \notin C, 
 	\end{cases}
 \end{align*} 
 with $C$ and $D$ polyhedral and $L$ differentiable, the general \enquote{evolution 
 	equation}
 \[   
 -\dot{x}(t) \in T(x(t)) \quad \text{for almost all }t,
 \]
 where $t \to x(t)$ is (in a suitable sense) an absolutely continuous function from $[0, 
 +\infty]$ to $X$ with derivative $t \to \dot{x}(t)$,
 reduces to the Arrow-Hurwicz differential equation 
 \cite[page~118]{ArrowHurwiczUzawa1958}.
 
 In conclusion, bearing our analysis of \cref{lemma:pmsm} in mind, 
 we see that algorithms with projected subgradients  (e.g., some special cases of the 
 Arrow-Hurwicz method)
 may have some corresponding counterparts  of algorithms with  proximity mappings
 (e.g., primal-dual proximal-point methods) 
 when the associated convex-concave function satisfies some special conditions, 
 but we think it's an open question if all algorithms with projected subgradients have 
 their counterparts  of algorithms with  proximity mappings in closed forms.

\end{document}